\documentclass[12pt,reqno]{amsart}

\usepackage[colorlinks,linkcolor={blue},
citecolor={blue},urlcolor={blue}]{hyperref}

\usepackage{mathrsfs}
\usepackage{bbm}
\usepackage{amsfonts}
\usepackage{amsfonts}
\usepackage{amsfonts}
\usepackage{amsfonts}
\usepackage{amsfonts}
\usepackage{amsfonts}
\usepackage{amsfonts}
\usepackage{amsfonts}
\usepackage{amsfonts}
\usepackage{amsfonts}
\usepackage{amsfonts}
\usepackage{amsfonts}
\usepackage{amsfonts}
\usepackage{amsfonts}
\usepackage{amsfonts}
\usepackage{amsfonts}
\usepackage{amsfonts}
\usepackage{amsfonts}
\usepackage{amsfonts}
\usepackage{}
\usepackage{pifont}
\usepackage{mathptmx} % Times New Roman
\usepackage[shortlabels]{enumitem}
\usepackage{empheq}
\usepackage{amsfonts}
\usepackage{amsmath, amsrefs, amsthm, amssymb}
\usepackage[active]{srcltx}		%reversed search for a dvi file
\usepackage{pdfsync}					%reversed
\usepackage{graphicx}
\usepackage [latin1]{inputenc}
\usepackage{bbm}
\usepackage{amsfonts}
\usepackage{amsthm}
\usepackage{graphicx}
\usepackage{framed}
\usepackage{multicol,multienum}
\usepackage{graphicx} % Allows including images
\usepackage{booktabs} % Allows the use of \toprule, \midrule and \bottomrule in tables
\usepackage{mathrsfs}
\usepackage{tcolorbox}
\usepackage{color}
\usepackage{xcolor}

\newtheorem{theorem}{Theorem}[section]
\newtheorem{proposition}[theorem]{Proposition}
\newtheorem{lemma}[theorem]{Lemma}
\newtheorem{corollary}[theorem]{Corollary}

\theoremstyle{definition}
\newtheorem{definition}[theorem]{Definition}

\newtheorem{remark}[theorem]{Remark}

\numberwithin{equation}{section}

\begin{document}
	
\title [Weakly dissipative HKS equations]{Hyperbolic Keller-Segel equations with sensitivity adjustment in Besov spaces: Global well-posedness and blow-up criteria}\thanks{Corresponding author: Lei Zhang (lei\_zhang@hust.edu.cn)}

\author{Bo Bi}
\address{School of Mathematics and Statistics, Huazhong University of Science and Technology,  Wuhan 430074, Hubei, P.R. China.}
\email{bobi@hust.edu.cn (B. Bi)}

\author{Lei Zhang}
\address{School of Mathematics and Statistics, Hubei Key Laboratory of Engineering Modeling  and Scientific Computing, Huazhong University of Science and Technology,  Wuhan 430074, Hubei, P.R. China.}
\email{lei\_zhang@hust.edu.cn (L. Zhang)}


\date{\today}
	
\begin{abstract}
This paper focuses on the initial value problem for the hyperbolic Keller-Segel (HKS) equation with sensitivity adjustment in Besov sapces over $\mathbb{R}^d$: $\partial_t u + \nabla \cdot \left(\varrho (t) u (1 - u) \nabla S \right)= 0$, $\Delta S = S - u$, where $\varrho (t) (1 - u)$ denotes the adjustment of the classical chemotaxis sensitively by virtue of a time-dependent function $\varrho$. In the case of $\varrho\equiv 1$, existed results [Zhou et al., \emph{J. Differ. Equ.}, 302 (2021) 662-679] and [Zhang et al., \emph{J. Differ. Equ.}, 334 (2022) 451-489] have shown that the HKS equation  admits local-in-time Besov solution, whereas the global theory in Besov space still remains an unsolved problem. The main novelty of our observation lies in the fact that if the chemotaxis sensitivity is adjusted by the function $\varrho(t)$ with suitable integrability over $[0,\infty)$, then the associated HKS equation possesses a unique global-in-time Besov solution. As an application, we conclude that the HKS equation with weakly dissipation $-\lambda u$ (i.e., a nonzero interaction-related source term) can also be globally solved in the framework of Besov spaces. Moreover, we derive two types of  blow-up criteria for strong solutions in both critical and non-critical Besov spaces, and explicitly characterize the lower bound of blow-up time. These findings reveals how time-dependent parameters (especially, the dissipation parameter $\lambda$) affect the global existence of solutions.
\end{abstract}

\maketitle
%\newpage
%\tableofcontents
%----------------------------------------------------
\section{Introduction}
%----------------------------------------------------

\subsection{Motivation}
Chemotaxis, the directed movement of somatic cells, bacteria, unicellular organisms, and multicellular life forms in response to specific chemical gradients, is an ubiquitous biological phenomenon \cite{Hillen2009}. To characterize such phenomena, Keller and Segel \cite{Keller1971,Keller1980} originally proposed the Keller-Segel (KS) equation, which nowadays can be generally  formulated by
\begin{equation}\label{a1}
\left\{
\begin{aligned}
& \partial_t u = \nabla \cdot \left(D_u(u,S)\nabla u - \chi(u,S)u \nabla S\right) + H(u,S),\\
& \tau \partial_t S = D_S(u,S) \Delta S + K(u,S).
\end{aligned}
\right.
\end{equation}
Here, $u$ denotes cell/organism density and $S$ denotes chemoattractant concentration; $H(u,S)$ and $K(u,S)$ are interaction-related source terms. The nonlinear positive-definite functions $D_u(u,S)$ and $D_S(u,S)$ describe cell and chemoattractant diffusivity, respectively. $\chi(u,S)$ is the chemotactic sensitivity, and $\tau \geq 0$ is a relaxation time scale with $1/\tau$ characterizing the equilibrium convergence rate. During the past decades, the KS model \eqref{a1} with different kinds of choice of the parameters have received much attention, we refer to  \cite{Biler1998,Horstmann2001,Jaeger1992,Nagai1997,Perthame2007,Winkler2010,Winkler2013} and their references therein.

In \cite{Hillen2001}, Hillen and Painter  studied the following classical parabolic-elliptic system by taking         $D_u(u,S) = \epsilon$, $\chi(u,S) = 1 - u$, $D_S(u,S) = 1$, $K(u,S) = S - u$, and $H(u,S) = 0$ in \eqref{a1}:
\begin{equation}\label{a2}
\left\{
\begin{aligned}
& \partial_t u -\epsilon \Delta u + \nabla \cdot (u(1-u) \nabla S) = 0, ~\epsilon>0,\\
& \Delta S = S-u.
\end{aligned}
\right.
\end{equation}
For \eqref{a2}, by considering the one-dimensional setting with homogeneous Neumann boundary conditions, Dolak and Schmeiser \cite{Dolak2005} established two results: first, they proved the global well-posedness of the problem under appropriate assumptions on the initial data; second, they showed that as the parameter $\epsilon \to 0$, the corresponding solutions $(u^{\epsilon}, S^{\epsilon})$ converge to the entropy solutions of the so-called hyperbolic Keller-Segel (HKS) equations (due to the absence of viscosity term):
\begin{equation}\label{a3}
\left\{
\begin{aligned}
& \partial_t u + \nabla \cdot (u(1-u) \nabla S) = 0 \quad \textrm{or} \quad \partial_t u + (1 - 2u)  \nabla u \cdot  \nabla S + (u - u^2) \Delta S = 0,\\
& \Delta S = S-u.
\end{aligned}
\right.
\end{equation}
In higher dimensions, Burger, Dolak-Struss and Schmeiser \cite{Burger2008} demonstrated short-time convergence of \eqref{a2} to \eqref{a3} as $\epsilon \to 0$ under suitable conditions, where the $\epsilon \to 0$ limit was established for global-in-time weak solutions. Similar results have also been established in the work \cite{Perthame2009}. Regarding the Cauchy problem of \eqref{a2}, Burger, Francesco and Schmeiser \cite{Burger2006} investigated the global existence and uniqueness of weak solutions, and also analyzed the long time behavior of these solutions. In \cite{Lee2015}, Lee and Liu identified a sub-threshold for finite-time shock formation in \eqref{a3} on $\mathbb{R}$. Most recently, Meng et al. \cite{Meng2024} proved global well-posedness for \eqref{a3} with small equilibrium-near initial data, and constructed initial data leading to finite-time blow-up¡ªall within $H^s$ spaces ($s > 1 + \frac{d}{2}$). Notably, these results are all set in Sobolev spaces.

Within the framework of Besov spaces, Zhou, Zhang and Mu \cite{Zhou2021} established local well-posedness for the Cauchy problem of \eqref{a3} in $B^s_{p,r}$ ($1 \leq p,r \leq \infty$, $s > 1 + \frac{d}{p}$), and further showed that the data-to-solution map is not uniformly continuous. Subsequently, Zhang, Mu and Zhou \cite{Zhang2022} proved the discontinuity of the solution map at $t=0$ in the Besov space $B^s_{2,\infty}$ with $s > \frac{3}{2}$. Moreover, they established a Hadamard local well-posedness result for the high dimensional HKS equation in the larger Besov space $B^{1+\frac{d}{p}}_{p,1}$ ($1 \leq p < \infty$), and derived two blow-up criteria for strong solutions via Littlewood-Paley theory. Let us mention that these last two works mainly consider the local-in-time solvability problem in Besov spaces, while its global well-posedness in suitable Besov spaces is challenging.

In view of the works previously discussed, especially the Cauchy problems addressed in \cite{Zhou2021,Zhang2022}, it is natural to ask the following question:
\begin{itemize}
\item [\textbf{(Q)}:]  What types of conditions ensure that the HKS equations \eqref{a3} possess a unique global Hadamard solution in an appropriate Besov space?
\end{itemize}
In a recent work by Meng et al. \cite{Meng2024}, under the restriction that solutions belong to Sobolev spaces, the authors partially addressed this problem (Q) by assuming that the solutions evolve on the torus $\mathbb{T}^d$, with small initial data in $H^s(\mathbb{T}^d)$ ($s > 1 + \frac{d}{2}$) near the equilibrium state. As far as we can ascertain, the corresponding problem within the framework of Besov space, even when restricted to the torus $\mathbb{T}^d$, has not yet been resolved. A key difficulty stems from the equations' substantial hyperbolic structure, which impairs the application of spatial regularity afforded by the viscosity term in \eqref{a2}.

The \emph{primary objective of the present study} is to provide a partial positive answer to a weaker version of the question (Q), that is, when the chemotactic sensitivity is adjusted by suitable time-dependent integrable functions. More specifically, by considering the chemotactic sensitivity $\xi(u,S)=\varrho (t)(1-u)$ in \eqref{a2}, we focus on the following Cauchy problem associated with the sensitivity-adjusted HKS equation over $\mathbb{R}^d (d \geq 1)$:
\begin{equation}\label{SCNS-c1}
\left\{
\begin{aligned}
& \partial_t u + \nabla \cdot \left(\varrho (t) u (1 - u) \nabla S \right)= 0,   &  x \in \mathbb{R}^d,\ t \geq 0,\\
&  \Delta S = S - u,  & x \in \mathbb{R}^d,\ t \geq 0,\\
& u(0,x) = u_0(x), & x \in \mathbb{R}^d,\ t = 0,
\end{aligned}
\right.
\end{equation}
where $\varrho (t)$ depends only on the $t$-variable. Equivalently, by using the relationship $S = (1-\Delta)^{-1}u$, one can rewrite \eqref{SCNS-c1} as a nonlocal and nonlinear transport equation:
\begin{equation}\label{SCNS-c11}
\left\{
\begin{aligned}
& \partial_t u +  (\varrho(t) - 2\varrho(t)u) \nabla S\cdot \nabla u + (\varrho(t)u-\varrho(t)u^2)\Delta S= 0,   &  x \in \mathbb{R}^d,\ t \geq 0,\\
&   S = (1-\Delta)^{-1}u,  & x \in \mathbb{R}^d,\ t \geq 0,\\
& u(0,x) = u_0(x), & x \in \mathbb{R}^d,\ t = 0.
\end{aligned}
\right.
\end{equation}
Provided that the time-dependent function
$\varrho(t)$ meets the appropriate integrable conditions, we will demonstrate in subsequent sections that  the equation \eqref{SCNS-c1} or \eqref{SCNS-c11} admits a global strong solution within the framework of Besov spaces.

To the best of our knowledge, there are relatively few existing studies that address the KS equations when its chemotactic sensitivity $\chi(u,S)$ (a quantity usually formulated as a function of $u$ and $S$) is adjusted by certain time-, space-, or spacetime-dependent functions. From this perspective, the key findings of this study demonstrate that modifying parameters can substantially change the solution properties of the HKS equation, which is interesting at least from the mathematical perspective of partial differential equations (PDEs). Along this line, it would be valuable to investigate whether similar phenomena exist for other KS equations with singular structures.

\subsection{Related problems}
An interesting observation in this paper is that Eq. \eqref{SCNS-c1} or Eq. \eqref{SCNS-c11} exhibits a close connection to the HKS equation that contains an interaction-associated source term given by $H(u,S)= -\lambda u$ (recall that in \eqref{a2}, $H(u,S)$ is taken to be zero).  In this regard, the HKS equations under consideration can be formulated as
\begin{equation}\label{SCNS-c2}
\left\{
\begin{aligned}
& \partial_t u + (1 - 2u) \nabla S \cdot \nabla u + (u - u^2) \Delta S + \lambda u = 0, & x \in \mathbb{R}^d,\ t \geq 0,\\
&  S = (I - \Delta)^{-1} u,  & x \in \mathbb{R}^d,\ t \geq 0,\\
& u(0,x) = u_0(x), & x \in \mathbb{R}^d,\ t = 0.
\end{aligned}
\right.
\end{equation}
From a mathematical perspective, the source term $\lambda u$ endows the solutions with a damping property. Consequently, there is reason to expect that these solutions exhibit superior characteristics compared to those of the equation \eqref{a3}. In fact, weakly dissipative  PDEs have been extensively investigated in the context of various other types of PDEs, see  for instance  \cite{ji2023optimal,pan2021global,chen2024global,Zhang2021}. The term $\lambda u$ is typically referred to as the \emph{weakly dissipation}, as it induces a dissipative effect that is less pronounced than that of the term $-\epsilon\Delta u$ in \eqref{a2}.

The relationship between \eqref{SCNS-c11} and \eqref{SCNS-c2} can be established by introducing the  transformation:
$\tilde{u}(t,x) = e^{\lambda t}u(t,x) $,
which  serves to simplify \eqref{SCNS-c2}, ultimately reducing it to a HKS equation that contains a time-dependent parameter:
\begin{equation}\label{SCNS-t}
\left\{
\begin{aligned}
& \partial_t \tilde{u} + (e^{- \lambda t} - 2e^{- 2\lambda t}\tilde{u}) \nabla \tilde{S} \cdot \nabla \tilde{u} + (e^{- \lambda t} \tilde{u} - e^{- 2\lambda t}\tilde{u}^2) \Delta \tilde{S} = 0, & x \in \mathbb{R}^d,\ t \geq 0,\\
&  \tilde{S} = (I - \Delta)^{-1} \tilde{u},  & x \in \mathbb{R}^d,\ t \geq 0,\\
& \tilde{u}(0,x) = u_0(x), & x \in \mathbb{R}^d,\ t = 0.
\end{aligned}
\right.
\end{equation}
By comparing \eqref{SCNS-c11} and \eqref{SCNS-t}, a key common feature comes to light: the nonlinear terms in each equation have parameter perturbations that depend solely on time.

 Within the theoretical framework of PDEs, it will be more convenient to integrate equations \eqref{SCNS-c11} and \eqref{SCNS-t} into a unified form as shown below:
\begin{equation}\label{SCNS-t1}
\left\{
\begin{aligned}
& \partial_t u + (\alpha(t) + \beta(t) u) \nabla S \cdot \nabla u +(\gamma(t) u + \xi(t) u^2) \Delta S =0, & x \in \mathbb{R}^d,\ t \geq 0,\\
& S = (I - \Delta)^{-1} u,  & x \in \mathbb{R}^d,\ t \geq 0,\\
& u(0,x) = u_0(x), & x \in \mathbb{R}^d,\ t = 0,
\end{aligned}
\right.
\end{equation}
where $\alpha$, $\beta$, $\gamma$ and $\xi$ are time-dependent parameters.

It can be readily observed that using the following specific parameters in Eq. (1.8) leads to the simplification of this equation into Eq. \eqref{a3}, Eq. \eqref{SCNS-c11} and \eqref{SCNS-t} respectively:
\begin{itemize}
\item [$\bullet$]  For \eqref{a3}:  $\alpha(t) = \beta(t) = \gamma(t) = \xi(t) = 1$.

\item [$\bullet$]  For \eqref{SCNS-c11}: $\alpha(t) = \varrho(t)$,   $\beta(t) = -2 \varrho(t)$, $ \gamma(t) = \varrho(t)$, $ \xi(t) = -\varrho(t)$, $\xi(t) \neq \textrm{constant.}$
\item [$\bullet$] For \eqref{SCNS-t}:  $\alpha(t) = e^{-\lambda t}$, $ \beta(t) = -2 e^{-2\lambda t}$, $ \gamma(t) = e^{-\lambda t}$, $ \xi(t) = e^{-2\lambda t}$.

\end{itemize}
One can conclude that if the equation \eqref{SCNS-t1} can be solved globally in Besov spaces with suitable condition imposed on the parameters $\alpha$, $\beta$, $\gamma$ and $\xi$, then one can also immediately obtain the global solvability results for \eqref{a3}, \eqref{SCNS-c11} and \eqref{SCNS-t}, and hence provide a positive answer to the aforementioned question (Q).
\subsection{Main results}

In the following, we aim to solve the problems associated with the equations \eqref{SCNS-c11} and \eqref{SCNS-t}, by virtue of the powerful tools from the Littlewood-Paley theory and transport theory in Besov spaces.

\begin{theorem}[Non-critical Besov spaces]\label{thm}
Let $d \geq 1$, $1 \leq p,r \leq +\infty$ and $ s>\max\{1+\frac{d}{p} , \frac{3}{2}\}$. Assume that $u_0 \in B^{s}_{p,r}(\mathbb{R}^d)$, and the time-dependent parameters $\alpha, \beta, \gamma, \xi \in L^1([0,\infty);\mathbb{R})$ such that
$$
\int_0^\infty (|\alpha(\tau)| + |\beta(\tau)| + |\gamma(\tau)| + |\xi(\tau)|) d\tau \leq \frac{\ln 2}{12C^3 h^2(1+\|u_0\|_{B^s_{p,r}})},
$$
where
\begin{equation}\label{hx}
\begin{aligned}
h(x) \stackrel{\text{def}}{=} & e^{4C^3 x^2 \int_0^\infty [|\alpha(\tau)| + |\beta(\tau)|] d\tau} \\
& \times \left(x + 8C^3 x^3 \int_0^\infty [|\alpha(\tau)| + |\beta(\tau)| + |\gamma(\tau)| + |\xi(\tau)|] d\tau\right).
\end{aligned}
\end{equation}
Then for any $T>0$, the system \eqref{SCNS-t1} has a unique global strong solution
$$
u \in E^s_{p,r}(T) \stackrel{\text{def}}{=} C([0,T];B^s_{p,r}(\mathbb{R}^d)) \cap C^1([0,T];B^{s-1}_{p,r} (\mathbb{R}^d)).
$$
Moreover, the data-to-solution map $u_0 \to u$ is continuous from a neighborhood of $u_0$ in $B^s_{p,r}(\mathbb{R}^d)$ to $E^{s'}_{p,r}(T)$ for $s' < s$ when $r=\infty$ and $s'=s$ whereas $r < \infty$.
\end{theorem}

\begin{theorem}[Critical Besov spaces]\label{thm1}
Let $d \geq 1$. Assume that $u_0 \in B^{1+\frac{d}{2}}_{2,1}(\mathbb{R}^d)$, and the time-dependent parameters $\alpha, \beta, \gamma, \xi \in L^1([0,\infty);\mathbb{R})$ such that
$$
\int_0^\infty (|\alpha(\tau)| + |\beta(\tau)| + |\gamma(\tau)| + |\xi(\tau)|) d\tau \leq \frac{\ln 2}{12C^3 h^2(1+\|u_0\|_{B^{1+\frac{d}{2}}_{2,1}})},
$$
where $h(x)$ is defined in \eqref{hx}.
Then for any $T>0$, the system \eqref{SCNS-t1} has a unique global strong solution
$$
u \in E^{1+\frac{d}{2}}_{2,1}(T) \stackrel{\text{def}}{=} C([0,T];B^{1+\frac{d}{2}}_{2,1}(\mathbb{R}^d)) \cap C^1([0,T];B^{\frac{d}{2}}_{2,1} (\mathbb{R}^d)).
$$
Moreover, the data-to-solution map $u_0 \to u$ is continuous from a neighborhood of $u_0$ in $B^{1+\frac{d}{2}}_{2,1}(\mathbb{R}^d)$ to $E^{1+\frac{d}{2}}_{2,1}(T)$.
\end{theorem}
\begin{remark}
As it is shown in Theorem \ref{thm} and Theorem \ref{thm1}, if the time-dependent parameters are integrable over $[0,\infty)$ with suitable bound from above, then the equation \eqref{SCNS-t1} admits a global solution, which partially answers the question (Q) by considering time-dependent parameters. How to improve our results for the case where $\alpha =  \beta=\gamma= \xi = 1$ (which corresponds to \eqref{a3}) remains a highly challenging problem.
\end{remark}

\begin{remark}
Besides the exponential perturbations in \eqref{SCNS-t}, the assumption also contains the rational function in the type of $\frac{1}{a+bt^2}(a,b>0)$, which is widely considered in damped fluid equations, such as \cite{ji2023optimal,pan2021global,chen2024global} and so on.  Moreover, we \emph{conjecture} that the above two theorems can also be established in the Triebel-Lizorkin space by employing the method used in this paper and the technique developed in \cite{chae2002well}, and we shall deal with this problem in the forthcoming works.
\end{remark}

\begin{remark}
Since the proof techniques for the cases $s=1+\frac{d}{2}$ and $s>1+\frac{d}{2}$ differ, despite both utilizing the same approximation systems, we present the main results separately for clarity. Moreover, in \cite{Meng2024}, the authors demonstrated that the HKS equations are ill-posed in $H^{ \frac{3}{2}}(\mathbb{T})$ but well-posed in $H^{s}(\mathbb{T}^d)$ for $s >1+ \frac{d}{2}$. Similar phenomena occurs in the local theory within Besov spaces, as shown in \cite{Zhang2022,Zhou2021}. This insight prompts us to designate $B_{2,1}^{1+\frac{d}{2}}$ as the critical space, where $s=1+\frac{d}{2}$ is the critical index.
\end{remark}

\begin{remark}
Observing that the function
\[
f(x) \stackrel{\text{def}}{=} 12C^3 x e^{8C^3 (1+\|u_0\|_{B^s_{p,r}})^2 x}\left(1+\|u_0\|_{B^s_{p,r}} + 8C^3 (1+\|u_0\|_{B^s_{p,r}})^3 x \right)^2,
\]
is an increasing and continuous function for all $x \in [0,\infty)$ with $f(0) = 0 $. The integrability condition provided in Theorem \ref{thm} and Theorem \ref{thm1} hold true as long as the $L^1$-integral $|\|\alpha\|_{L^1} + \|\beta\|_{L^1} + \|\gamma\|_{L^1} + \|\xi\|_{L^1}$ is sufficiently small, which answers the question of global solvability of HKS in Besov space, see the corollary stated below.
\end{remark}

\begin{remark}
It is known that if the term $\lambda u$ in \eqref{SCNS-c2} is replaced by the stronger dissipation $-\Delta u$, the model admits a unique global solution \cite{Hillen2001,Dolak2005}. However, without the term $-\Delta u$, existing studies have demonstrated that the solution is locally well-posed but blows up in finite time \cite{Zhang2022,Zhou2021}. Our findings indicate that, compared to strong dissipation, the weak dissipation $\lambda u$ is sufficient to ensure the existence of global solutions in Besov space. On the other hand, the weak dissipation $\lambda u$ in \eqref{SCNS-t1} can be seen as a limiting case of the standard logistic source term $H(u)=a u - bu^\alpha$ ($a \geq 0$, $b > 0$, and $\alpha > 1$), by setting $a=0$, $b > 0$, and $\alpha \rightarrow 1$. In chemotaxis PDE theory, the condition $\alpha > 1$ introduces favorable properties for solutions, thus ensuring global existence, while little is known about the limiting case where $a=0$ and $\alpha = 1$. In this context, Theorem \ref{thm} and \ref{thm1} reveal that, for the hyperbolic-type KS model \eqref{a3}, the weaker logistic source term $\lambda u$ ($\lambda > 0$) overcomes the challenges posed by the absence of $-\Delta u$.
\end{remark}

As an application,  we immediately conclude from Theorem \ref{thm} the following result.

\begin{corollary}\label{thm2}
Let $s>\max\{1+\frac{d}{p} , \frac{3}{2}\}$ and $(p,r)\in [0,\infty]^2$. Assume that $u_0 \in B^s_{p,r}(\mathbb{R}^d)$ and dissipative constant $\lambda >0$ is sufficiently large such that
\[
e^{\frac{16C^3}{\lambda} (1+\|u_0\|_{B^s_{p,r}})^2} \left(1+\|u_0\|_{B^s_{p,r}} + \frac{28C^3}{\lambda} (1+\|u_0\|_{B^s_{p,r}})^3\right)^2 \leq \frac{\lambda \ln 2}{42C^3}.
\]
Then the Cauchy problem \eqref{SCNS-c2} or \eqref{SCNS-t} has a unique global strong solution $u \in E^s_{p,r}(\infty)$.
\end{corollary}

\begin{remark}
As an example, we provide a sufficient condition which satisfies the inequality in Corollary \ref{thm2}:
$
\lambda \geq \frac{120 C^3}{\ln 2} (1 + \|u_0\|_{B^s_{p,r}}  )^2.
$
Observing that, for any given $\lambda >0$, Corollary \ref{thm2} with this inequality actually gives the \emph{small data global-in-time result} for the system \eqref{SCNS-t} in Besov spaces, where the initial data satisfies $\|u_0\|_{B^s_{p,r}} \leq \sqrt{\frac{\lambda \ln 2}{120 C^3}} -1$. The case in the critical Besov space $B^{1+\frac{d}{2}}_{2,1}$ is the same as that in the non-critical Besov space, and further elaboration will be omitted here.
\end{remark}

Our second goal in the present paper is to investigate the finite time blow-up regime for the system \eqref{SCNS-t1} in Besov spaces, which in some sense tell us how the time-dependent parameters $\alpha(t)$, $\beta(t)$, $\gamma(t)$ and $\xi(t)$ affect the singularity formation.

\begin{theorem}\label{blow-up_1}
Assume the parameters $\alpha$, $\beta$, $\gamma$, $\xi \in L^1_{loc}(0,\infty; \mathbb{R})$, and the initial data $u_0 \in B^{1+\frac{d}{2}}_{2,1}(\mathbb{R}^d)$. If the corresponding solution $u$ blows up at finite time $T^*$, then
\[
\int_{0}^{T^{*}} (|\alpha(r)| + |\beta(r)| + |\gamma(r)| + |\xi(r)|) (\| u \|^2_{\dot{B}^0_{\infty,1}} + \| \nabla u \|^2_{\dot{B}^0_{\infty,1}})dr = \infty,
\]
and the blow-up time $T^{*}$ is estimated as follows
\[
T^{*} \geq T(u_0) \overset{\cdot}{=} \sup_{t > 0} \left\{\int_{0}^{t} (|\alpha(r)| + |\beta(r)| + |\gamma(r)| + |\xi(r)|)dr \leq \frac{1}{C (1+\| u_0 \|_{B^{1+\frac{d}{2}}_{2,1}})^2}\right\}.
\]
\end{theorem}

Our third goal in the paper is to construct the blow-up criteria with the initial data in $B^{s}_{p,r}$ for $s>1+\frac{d}{p}$. The following theorem presents this result.
\begin{theorem}\label{blow-up_2}
Assume the parameters $\alpha$, $\beta$, $\gamma$, $\xi \in L^1_{ {loc}}(0,\infty; \mathbb{R})$. Let $s>1+\frac{d}{p}$ with $d \geq 1$, $1 \leq p \leq \infty$ and $1 \leq r \leq \infty$. Assume that $u \in C([0,T^{*});B^s_{p,r}(\mathbb{R}^d))$ is the unique solution to system \eqref{SCNS-t1} with respect to the initial data $u_0 \in B^s_{p,r}(\mathbb{R}^d)$. If the corresponding solution $u$ blows up in the finite time $T^{*}$, then
\[
\int_{0}^{T^{*}} (|\alpha(r)| + |\beta(r)| + |\gamma(r)| + |\xi(r)|) (\| u \|^2_{\dot{B}^0_{\infty,2}} + \| \nabla u \|^2_{\dot{B}^0_{\infty,2}}) \, dr = \infty,
\]
and the blow-up time $T^{*}$ is estimated as follows
\[
T^{*} \geq T'(u_0) \overset{\cdot}{=} \sup_{t > 0} \left\{ \int_{0}^{t} (|\alpha(r)| + |\beta(r)| + |\gamma(r)| + |\xi(r)|) \, dr \leq \frac{1}{C\left(e + \|u_0\|_{B^s_{p,r}}\right)^6} \right\}.
\]
\end{theorem}

\begin{remark}

The blow-up criteria in Theorem \ref{blow-up_1} does not contradict Theorem \ref{thm}, which indicates that the $L^1$-integrability of parameters $\alpha$, $\beta$, $\gamma$, and $\xi$ determines the behavior of solutions to system \eqref{SCNS-t1}. Specifically, if $\| \alpha\|_{L^1(0,\infty)} + \| \beta \|_{L^1(0,\infty)} + \| \gamma \|_{L^1(0,\infty)} + \| \xi \|_{L^1(0,\infty)} = \frac{1}{2C (1+\| u_0 \|_{B^{1+\frac{d}{2}}_{2,1}})^2}$, then Theorem \ref{blow-up_1} implies $T^* = T(u_0) = \infty$, meaning the solution $u$ exists globally. The specific blow-up time can be determined by applying the parameters to Theorem \ref{blow-up_1}.

\end{remark}

The remaining of this paper is organized as follows. Sections 3 and 4 are devoted to proving the global Hadamard well-posedness of the damped system \eqref{SCNS-t1} in non-critical Besov spaces and critical Besov spaces, respectively. In section 5, we derive two kinds of blow-up criteria for the strong solutions to the system \eqref{SCNS-t1}. The lower bound of the blow-up time are also addressed. In appendix, we recall the Littlewood-Paley theory, and some well-known results of the transport theory in Besov spaces.

\section{Global well-posedness in non-critical Besov Spaces}
%----------------------------------------------------
In this section, we mainly focus on the proof of Theorem \ref{thm}, which will be divided into the following several steps. Since the proof of Corollary \ref{thm2} is completely same to Theorem \ref{thm}, we shall omit the detail here.

\subsection{Approximate solutions}
We construct the approximate solutions via the Friedrichs iterative method. Starting from $u^{(1)} \overset{\cdot}{=} S_1 u_0$, we recursively define a sequence of functions $(u^{(n)})_{n \geq 1}$ by solving the following linear transport equations
\begin{equation}\label{SCNS-t2}
\left\{
\begin{aligned}
& \partial_t u^{(n+1)} + (\alpha(t) + \beta(t) u^{(n)}) \nabla S^{(n)} \cdot \nabla u^{(n+1)} = -(\gamma(t) u^{(n)} + \xi(t) (u^{(n)})^2) \Delta S^{(n)},\\
& S^{(n)} = (I - \Delta)^{-1} u^{(n)},\\
& u^{(n+1)}\big|_{t = 0} = S_{n+1}u_0.
\end{aligned}
\right.
\end{equation}
Assuming that $u^{(n)} \in E^{s}_{p,r}(T)$ for any positive $T$, it is seen that the right hand side of the system \eqref{SCNS-t2} belongs to $L^1_{loc}([0,\infty);B^{s}_{p,r})$ since the parameters $\alpha,\beta,\gamma,\xi \in L^1([0,\infty);\mathbb{R})$. Hence, applying Lemma \ref{euth} ensures that the Cauchy problem \eqref{SCNS-t2} has a global solution $u^{(n+1)}$ which belongs to $E^{s}_{p,r}(T)$ for any positive $T$.

\subsection{Uniform bound}
If $s>1+\frac{d}{p}$, applying the Lemma \ref{prior estimates} to \eqref{SCNS-t2}, we have:
\begin{equation}\label{SCNS-t3}
\begin{aligned}
  \|u^{(n+1)}(t)\|_{B^s_{p,r}} & \leq \|S_{n+1} u_0\|_{B^s_{p,r}} + \int_0^t \|[\gamma(\tau)u^{(n)} + \xi(\tau)(u^{(n)})^2]\Delta S^{(n)}\|_{B^s_{p,r}} d\tau \\
 &\quad + C \int_0^t V'(\tau) \|u^{(n+1)}(\tau)\|_{B^s_{p,r}} d\tau,\\
\end{aligned}
\end{equation}
where $V(t) = \int_0^t \|\nabla [(\alpha(\tau) + \beta(\tau) u^{(n)}) \nabla S^{(n)}](\tau)\|_{B^{s-1}_{p,r}}d\tau$. For the terms on the right-hand side \eqref{SCNS-t3}, we first get from the property of $\chi(\cdot)$ that $\|S_{n+1}u_0\|_{B^s_{p,r}} \le C\|u_0\|_{B^s_{p,r}}$, for some positive constant independent of $n$. Since $B^s_{p,r}$ is a Banach algebra for any $s>1+\frac{d}{p}$, and the operator $(I - \Delta)^{-1}$ is a $S^{-2}$-multiplier, we have
\begin{equation}
\begin{aligned}
\|[\gamma(\tau)u^{(n)} + \xi(\tau)(u^{(n)})^2]\Delta S^{(n)}\|_{B^s_{p,r}}
& \le C \|\gamma(\tau)u^{(n)} + \xi(\tau)(u^{(n)})^2\|_{B^s_{p,r}} \cdot \|(I - \Delta)^{-1} \Delta u^{(n)}\|_{B^s_{p,r}} \\
& \le C [|\gamma(\tau)| \cdot \|u^{(n)}\|_{B^s_{p,r}} + |\xi(\tau)| \cdot \|u^{(n)}\|_{B^s_{p,r}}^2 ] \cdot \|\Delta u^{(n)}\|_{B^{s-2}_{p,r}} \\
%& \le C [|\gamma(\tau)| \cdot \|u^{(n)}\|_{B^s_{p,r}} +  |\xi(\tau)| \cdot \|u^{(n)}\|_{B^s_{p,r}}^2 ] \cdot \|u^{(n)}\|_{B^s_{p,r}} \\
& \le C [|\gamma(\tau)| \cdot \|u^{(n)}\|_{B^s_{p,r}}^2 +  |\xi(\tau)| \cdot \|u^{(n)}\|_{B^s_{p,r}}^3 ], \notag
\end{aligned}
\end{equation}
and
\begin{equation}
\begin{aligned}
\|\nabla [(\alpha(\tau) + \beta(\tau) u^{(n)}) \nabla S^{(n)}]\|_{B^{s-1}_{p,r}}
& \le C \|[(\alpha(\tau) + \beta(\tau) u^{(n)}) \cdot (I - \Delta)^{-1} \nabla u^{(n)}]\|_{B^{s}_{p,r}} \\
& \le C [|\alpha(\tau)| + |\beta(\tau)| \cdot \|u^{(n)}\|_{B^s_{p,r}} ] \cdot \|\nabla u^{(n)}\|_{B^{s-2}_{p,r}} \\
%& \le C [|\alpha(\tau)| + |\beta(\tau)| \cdot \|u^{(n)}\|_{B^s_{p,r}} ] \cdot \| u^{(n)}\|_{B^{s-1}_{p,r}} \\
& \le C [|\alpha(\tau)| \cdot \|u^{(n)}\|_{B^s_{p,r}} + |\beta(\tau)| \cdot \|u^{(n)}\|_{B^s_{p,r}}^2 ]. \notag
\end{aligned}
\end{equation}
Plugging the last two estimates into \eqref{SCNS-t3} and using the Young's inequality lead to
\begin{equation}
\begin{aligned}
  \|u^{(n+1)}(t)\|_{B^s_{p,r}} & \leq C \|u_0\|_{B^s_{p,r}} + C \int_0^t [|\gamma(\tau)| \cdot \|u^{(n)}(\tau)\|_{B^s_{p,r}}^2 + |\xi(\tau)| \cdot \|u^{(n)}(\tau)\|_{B^s_{p,r}}^3 ] d\tau \\
& \quad + C \int_0^t [|\alpha(\tau)| \cdot \|u^{(n)}(\tau)\|_{B^s_{p,r}} + |\beta(\tau)| \cdot \|u^{(n)}(\tau)\|_{B^s_{p,r}}^2 ] \cdot \|u^{(n+1)}(\tau)\|_{B^s_{p,r}}  d\tau,\\
& \leq C \|u_0\|_{B^s_{p,r}} + C \int_0^t [|\gamma(\tau)| + |\xi(\tau)|](1+\|u^{(n)}(\tau)\|_{B^s_{p,r}})^3 d\tau \\
& \quad + C \int_0^t [|\alpha(\tau)| + |\beta(\tau)|](1+\|u^{(n)}(\tau)\|_{B^s_{p,r}})^2 \cdot \|u^{(n+1)}(\tau)\|_{B^s_{p,r}} d\tau. \notag
\end{aligned}
\end{equation}
Using Gronwall lemma to last inequality leads to
\begin{equation}\label{SCNS-t4}
\begin{aligned}
  1 & + \|u^{(n+1)}(t)\|_{B^s_{p,r}} \leq C e^{C \int_0^t [|\alpha(\tau)| + |\beta(\tau)|](1+\|u^{(n)}(\tau)\|_{B^s_{p,r}})^2 d\tau} \\
& \times \left(1 + \|u_0\|_{B^s_{p,r}} + \int_0^t [|\gamma(\tau)| + |\xi(\tau)|](1+\|u^{(n)}(\tau)\|_{B^s_{p,r}})^3 d\tau\right).
\end{aligned}
\end{equation}

Define
\begin{equation}
  H_0 \stackrel{\text{def}}{=} H^{(n)}(0)=1+\|u_0\|_{B^s_{p,r}}, \notag
\end{equation}
and
\begin{equation}
  H^{(n)}(t) \stackrel{\text{def}}{=} 1+\|u^{(n)}(t)\|_{B^s_{p,r}},
\quad \text{for all } n \geq 1. \notag
\end{equation}
It follows from the inequality \eqref{SCNS-t4} that
\begin{equation}\label{SCNS-t5}
\begin{aligned}
  H^{(n+1)}(t) & \leq C e^{C \int_0^t [|\alpha(\tau)| + |\beta(\tau)|](H^{(n)}(\tau))^2 d\tau} \\
& \quad \times [  H_0  + \int_0^t [|\gamma(\tau)| + |\xi(\tau)|](H^{(n)}(\tau))^3 d\tau].
\end{aligned}
\end{equation}
Notice that the iterative inequality \eqref{SCNS-t5} contains additional factor $|\alpha(\tau)| + |\beta(\tau)|$ and $|\gamma(\tau)| + |\xi(\tau)|$, the classical method used in \cite{Danchin2001} is inapplicable in present case. To overcome this difficulty, we use another iterative method to derive the uniform bound.

For $n =1$, it follows from \eqref{SCNS-t5} that
\begin{equation}\label{SCNS-t20}
\begin{aligned}
\sup_{t \in [0, \infty)} H^{(1)}(t) & \leq C e^{C \int_0^\infty [|\alpha(\tau)| + |\beta(\tau)|](2C H_0)^2 d\tau}    [H_0  + \int_0^\infty [|\gamma(\tau)| + |\xi(\tau)|](2C H_0)^3 d\tau] \\
& \leq 2C e^{4C^3 {H_0}^2 \int_0^\infty [|\alpha(\tau)| + |\beta(\tau)|] d\tau} (H_0 + 8C^3 {H_0}^3 \int_0^\infty [|\gamma(\tau)| + |\xi(\tau)|] d\tau)  \\
& \leq 2C e^{4C^3 {H_0}^2 \int_0^\infty [|\alpha(\tau)| + |\beta(\tau)|] d\tau} \\
& \quad \times (H_0 + 8C^3 {H_0}^3 \int_0^\infty [|\alpha(\tau)| + |\beta(\tau)| + |\gamma(\tau)| + |\xi(\tau)|] d\tau) \\
& \doteq 2C h(H_0),
\end{aligned}
\end{equation}
where
$
  h(x) \stackrel{\text{def}}{=} e^{4C^3 x^2 \int_0^\infty [|\alpha(\tau)| + |\beta(\tau)|] d\tau}(x + 8C^3 x^3 \int_0^\infty [|\alpha(\tau)| + |\beta(\tau)| + |\gamma(\tau)| + |\xi(\tau)|] d\tau) . \notag
$
It is clear that $h(0) = 0$ and the function $h(x)$ is a modulus of continuity defined on $\mathbb{R}^+$, which is independent of the initial data $u_0$.

For $n =2$, we deduce from \eqref{SCNS-t5} that
\begin{equation}\label{SCNS-t6}
\begin{aligned}
\sup_{t \in [0, \infty)} H^{(2)}(t) & \leq C e^{C \int_0^\infty [|\alpha(\tau)| + |\beta(\tau)|](2C h(H_0))^2 d\tau}  [H_0  + \int_0^\infty [|\gamma(\tau)| + |\xi(\tau)|](2C h(H_0))^3 d\tau] \\
& \leq C e^{4C^3 h^2(H_0) \int_0^\infty [|\alpha(\tau)| + |\beta(\tau)|] d\tau} (H_0 + 8C^3 h^3(H_0) \int_0^\infty [|\gamma(\tau)| + |\xi(\tau)|] d\tau)  \\
& \leq C h(H_0)e^{4C^3 h^2(H_0) \int_0^\infty [|\alpha(\tau)| + |\beta(\tau)|] d\tau} \\
& \quad \times (1 + 8C^3 h^2(H_0) \int_0^\infty [|\gamma(\tau)| + |\xi(\tau)|] d\tau)  \\
& \leq C h(H_0)e^{12C^3 h^2(H_0) \int_0^\infty [|\alpha(\tau)| + |\beta(\tau)| + |\gamma(\tau)| + |\xi(\tau)|] d\tau}.
\end{aligned}
\end{equation}
In the last inequality of \eqref{SCNS-t6},we have used the facts of $h(H_0) \geq H_0$ and $1+x \leq e^x$ for any $x \geq 0$. To estimate the right hand side of \eqref{SCNS-t6}, we assume that
\begin{equation*}
\begin{aligned}
  \int_0^\infty & [|\alpha(\tau)| + |\beta(\tau)| + |\gamma(\tau)| + |\xi(\tau)|] d\tau \\
& \leq \frac{\ln 2}{12C^3 e^{8C^3 {H_0}^2 \int_0^\infty [|\alpha(\tau)| + |\beta(\tau)|] d\tau}(H_0 + 8C^3 {H_0}^3 \int_0^\infty [|\alpha(\tau)| + |\beta(\tau)| + |\gamma(\tau)| + |\xi(\tau)|] d\tau)^2 } \\
& = \frac{\ln 2}{12C^3 h^2(H_0)},
\end{aligned}
\end{equation*}
which combined with \eqref{SCNS-t6} imply that
\begin{equation}\label{SCNS-t7}
\begin{aligned}
\sup_{t \in [0, \infty)} H^{(2)}(t) \leq C h(H_0)e^{12C^3 h^2(H_0) \int_0^\infty [|\alpha(\tau)| + |\beta(\tau)| + |\gamma(\tau)| + |\xi(\tau)|] d\tau} \leq 2C h(H_0).
\end{aligned}
\end{equation}

Assuming for any given $n \geq 3$ that
$
\sup_{t \in [0, \infty)} H^{(n)}(t) \leq 2C h(H_0).
$
For $H^{(n+1)}(t)$, it follows from \eqref{SCNS-t5} and \eqref{SCNS-t7} that
\begin{equation*}
\begin{aligned}
\sup_{t \in [0, \infty)} H^{(n+1)}(t) & \leq C e^{C \int_0^\infty [|\alpha(\tau)| + |\beta(\tau)|](H^{(n)}))^2 d\tau}  [H_0  + \int_0^\infty [|\gamma(\tau)| + |\xi(\tau)|](H^{(n)})^3 d\tau] \\
& \leq C e^{C \int_0^\infty [|\alpha(\tau)| + |\beta(\tau)|](2C h(H_0))^2 d\tau} [H_0  + \int_0^\infty [|\gamma(\tau)| + |\xi(\tau)|](2C h(H_0))^3 d\tau] \\
& \leq C h(H_0)e^{4C^3 h^2(H_0) \int_0^\infty [|\alpha(\tau)| + |\beta(\tau)|] d\tau} \\
& \quad \times (1 + 8C^3 h^2(H_0) \int_0^\infty [|\gamma(\tau)| + |\xi(\tau)|] d\tau)  \\
& \leq C h(H_0)e^{12C^3 h^2(H_0) \int_0^\infty [|\alpha(\tau)| + |\beta(\tau)| + |\gamma(\tau)| + |\xi(\tau)|] d\tau} \leq 2C h(H_0).
\end{aligned}
\end{equation*}
Using the mathematical induction with respect to n, we get
$
\sup_{n \geq 1}\sup_{t \in [0, \infty)} H^{(n+1)}(t) \leq 2C h(H_0),
$
which implies the uniform bound
\begin{equation}\label{uniform-bound}
\sup_{n \geq 1} \| u^{(n)} \|_{L^\infty([0, \infty); B^s_{p,r})} \leq 2C h(H_0).
\end{equation}
As a result, the approximate solutions $(u^{(n)})_{n \geq 1}$ is uniformly bounded in $C([0,\infty);B^{s}_{p,r})$. Moreover, using the system \eqref{SCNS-t2} itself, one can verify that the sequence $(\partial_t u^{(n)})_{n \geq 1}$ is uniformly bounded in $C([0,\infty);B^{s-1}_{p,r})$. Therefore, we obtain that $(u^{(n)})_{n \geq 1}$ is uniformly bounded in $E^{s}_{p,r}(\infty)$.

\subsection{Convergence}
We show that $(u^{(n)})_{n \geq 1}$ is a Cauchy sequence in the space $C([0, \infty);B^{s-1}_{p,r})$. Indeed, according to \eqref{SCNS-t2}, we obtain that for any $n,m \geq 1$
\begin{equation}\label{SCNS-t8}
\begin{aligned}
[\partial_t &+ (\alpha(t) + \beta(t) u^{(n+m)}) \nabla S^{(n+m)}\cdot \nabla]  (u^{(n+m+1)}-u^{(n+1)})
&= -[(\alpha(t) + \beta(t) u^{(n+m)}) \cdot \nabla S^{(n+m)} - (\alpha(t) + \beta(t) u^{(n)}) \cdot \nabla S^{(n)}] \\
& \quad \times \nabla u^{(n+1)} - f(u^{(n)},u^{(n+m)}),
\end{aligned}
\end{equation}
with the initial conditions
\begin{equation}
\begin{aligned}
  (u^{(n+m+1)} - u^{(n+1)})(0,x) = S_{n+m+1}u_0(x) - S_{n+1}u_0(x),\notag
\end{aligned}
\end{equation}
where
$f(u^{(n)},u^{(n+m)}) = (\gamma(t) u^{(n+m)} + \xi(t) (u^{(n+m)})^2) \Delta S^{(n+m)} - (\gamma(t) u^{(n)} + \xi(t) (u^{(n)})^2) \Delta S^{(n)}$.

Applying the Lemma \ref{prior estimates} to \eqref{SCNS-t8}, we deduce that
\begin{equation}\label{SCNS-t9}
\begin{aligned}
  e & ^{-C \int_0^t \| \nabla [(\alpha(\tau) + \beta(\tau) \cdot u^{(n+m)}) \nabla S^{(n+m)}] \|_{B^{\frac{d}{p}}_{p,r} \cap L^\infty} \, d\tau} \| (u^{(n+m+1)} - u^{(n+1)})(t) \|_{B^{s-1}_{p,r}} \\
& \leq \| S_{n+m+1}u_0 - S_{n+1}u_0 \|_{B^{s-1}_{p,r}} + \int_0^t e ^{-C \int_0^{\tau} \| \nabla [(\alpha(t') + \beta(t') \cdot u^{(n+m)}) \nabla S^{(n+m)}] \|_{B^{\frac{d}{p}}_{p,r} \cap L^\infty} \, dt'} \\
& \quad \times (\| [(\alpha(\tau) + \beta(\tau) u^{(n+m)}) \cdot \nabla S^{(n+m)} - (\alpha(\tau) + \beta(\tau) u^{(n)}) \cdot \nabla S^{(n)}] \cdot \nabla u^{(n+1)} \|_{B^{s-1}_{p,r}} \\
& \quad + \|f(u^{(n)},u^{(n+m)}) \|_{B^{s-1}_{p,r}}) d\tau. \\
\end{aligned}
\end{equation}
For $s > 1+ \frac{d}{p} $, the space $B^{s-1}_{p,r}$ is a Banach algebra and $B^{s-1}_{p,r}$ continuously embedded into $B^{\frac{d}{p}}_{p,r} \cap L^\infty$, we get from the uniform bound \eqref{uniform-bound} that
\begin{equation}
\begin{aligned}
  \int_0^t & \| \nabla [(\alpha(\tau) + \beta(\tau) \cdot u^{(n+m)}) \nabla S^{(n+m)}] \|_{B^{\frac{d}{p}}_{p,r} \cap L^\infty} \, d\tau \\
& \leq  C \int_0^t \| \nabla [(\alpha(\tau) + \beta(\tau) \cdot u^{(n+m)}) \nabla S^{(n+m)}] \|_{B^{s-1}_{p,r}} \, d\tau \\
& \leq C \int_0^t \| (\alpha(\tau) + \beta(\tau) \cdot u^{(n+m)}) u^{(n+m)} \|_{B^{s-1}_{p,r}} \, d\tau \\
%& \leq C \int_0^t  (|\alpha(\tau)| \cdot \|u^{(n+m)} \|_{B^{s}_{p,r}}+ |\beta(\tau)| \cdot \|u^{(n+m)} \|^2_{B^{s}_{p,r}})  \, d\tau \\
& \leq C \int_0^t  (|\alpha(\tau)| + |\beta(\tau)|)  (\|u^{(n+m)} \|_{B^{s}_{p,r}} +\|u^{(n+m)} \|^2_{B^{s}_{p,r}})\, d\tau  ,
\end{aligned}
\end{equation}

\begin{equation}
\begin{aligned}
  \| [ & (\alpha(\tau) + \beta(\tau) u^{(n+m)}) \cdot \nabla S^{(n+m)} - (\alpha(\tau) + \beta(\tau) u^{(n)}) \cdot \nabla S^{(n)}] \cdot \nabla u^{(n+1)} \|_{B^{s-1}_{p,r}} \\
& \leq  C \| \alpha(\tau)(\nabla S^{(n+m)} - \nabla S^{(n)})  + \beta(\tau) u^{(n+m)} \nabla S^{(n+m)} \\
& \quad - \beta(\tau) u^{(n)} \nabla S^{(n+m)} + \beta(\tau) u^{(n)} \nabla S^{(n+m)} - \beta(\tau) u^{(n)} \nabla S^{(n)} \|_{B^{s-1}_{p,r}} \cdot \|u^{(n+1)}\|_{B^{s}_{p,r}} \\
%& = C \| \alpha(\tau)(\nabla S^{(n+m)} - \nabla S^{(n)})  + \beta(\tau) (u^{(n+m)} - u^{(n)}) \nabla S^{(n+m)} \\
& \quad + \beta(\tau) u^{(n)} (\nabla S^{(n+m)} - \nabla S^{(n)}) \|_{B^{s-1}_{p,r}} \cdot \|u^{(n+1)}\|_{B^{s}_{p,r}} \\
& = C \| (\alpha(\tau) + \beta(\tau)u^{(n)})(\nabla S^{(n+m)} - \nabla S^{(n)})  \\
& \quad + \beta(\tau) (u^{(n+m)} - u^{(n)}) \nabla S^{(n+m)} \|_{B^{s-1}_{p,r}} \cdot \|u^{(n+1)}\|_{B^{s}_{p,r}} \\
& \leq  C \| (u^{(n+m)} - u^{(n)})(\alpha(\tau) + \beta(\tau)u^{(n)} + \beta(\tau)u^{(n+m)}) \|_{B^{s-2}_{p,r}} \cdot \|u^{(n+1)}\|_{B^{s}_{p,r}} \\
& \leq  C (|\alpha(\tau)| + |\beta(\tau)|)(1 + \| u^{(n)} \|_{B^{s}_{p,r}} + + \| u^{(n+m)} \|_{B^{s}_{p,r}}) \\
& \quad \times \|u^{(n+m)} - u^{(n)}\|_{B^{s-1}_{p,r}} \cdot \|u^{(n+1)}\|_{B^{s}_{p,r}},
\end{aligned}
\end{equation}
and
\begin{equation}
\begin{aligned}
  \|f(u^{(n)},u^{(n+m)}) \|_{B^{s-1}_{p,r}}  & = \|(\gamma(t) u^{(n+m)} + \xi(t) (u^{(n+m)})^2) \Delta S^{(n+m)}  - (\gamma(t) u^{(n)} + \xi(t) (u^{(n)})^2) \Delta S^{(n)}\|_{B^{s-1}_{p,r}} \\
& \leq C \|(\gamma(t) (u^{(n+m)})^2 + \xi(t) (u^{(n+m)})^3) - (\gamma(t) (u^{(n)})^2 - \xi(t) (u^{(n)})^3)\|_{B^{s-1}_{p,r}}\\
 &\leq C (|\gamma(t)| + |\xi(t)|)\|u^{(n+m)} - u^{(n)}\|_{B^{s-1}_{p,r}} (1 + \|u^{(n+m)}\|_{B^{s}_{p,r}} + \|u^{(n)}\|_{B^{s}_{p,r}})^2.
\end{aligned}
\end{equation}

According to the definition of the Littlewood-Paley blocks $\Delta_p$ and the almost orthogonal property $\Delta_i \Delta_j =0$ for $|i-j| \geq 2$, we have
\begin{equation*}
\begin{aligned}
  \| S_{n+m+1}u_0 - S_{n+1}u_0 \|^r_{B^{s-1}_{p,r}} %& = \left\| \sum_{j=n+1}^{n+m} \Delta_j u_0 \right\|^r_{B^{s-1}_{p,r}} \\
%& =  \sum_{j \geq -1} 2^{jr(s-1)} \left\| \Delta_j \sum_{i=n+1}^{n+m} \Delta_i u_0 \right\|^r_{L^p} \\
 \leq C \sum_{j=n}^{n+m+1} (2^{-jr} 2^{jrs} \sum_{i=n+1}^{n+m}\| \Delta_j  \Delta_i u_0 \|^r_{L^p}) \\
 \leq C 2^{-rn} \sum_{j=n}^{n+m+1} 2^{jrs} \| \Delta_j u_0 \|^r_{L^p}  \leq C 2^{-rn} \| u_0 \|^r_{B^{s}_{p,r}}.
\end{aligned}
\end{equation*}
Hence, we obtain
\begin{equation}\label{SCNS-t10}
\begin{aligned}
  \| S_{n+m+1}u_0 - S_{n+1}u_0 \|_{B^{s-1}_{p,r}} \leq C 2^{-n}.
\end{aligned}
\end{equation}

Plugging the estimates \eqref{SCNS-t9}-\eqref{SCNS-t10} together, we obtain
\begin{equation}
\begin{aligned}
  \| (u^{(n+m+1)} - u^{(n+1)})(t) \|_{B^{s-1}_{p,r}} & \leq C(2^{-n} + \int_{0}^{t}(|\alpha(\tau)| + |\beta(\tau)| + |\gamma(\tau)| + |\xi(\tau)|) \\
& \quad \times \| (u^{(n+m)} - u^{(n)})(t) \|_{B^{s-1}_{p,r}} d\tau).
\end{aligned}
\end{equation}
To derive an estimate for $\| (u^{(n+m)} - u^{(n)})(t) \|_{B^{s-1}_{p,r}}$, we first prove the following useful lemma.

\begin{lemma}\label{lem:RL}
Let ${a_n}$ be a positive constant, and the nonnegetive function $\mu(t) \in L^1(0,\infty).$  Assume that the sequence of nonnegative functions $ (g_n(t))_{n \geq 1}$ satisfies the inequality
$$
g_{n+1}(t) \leq a_n + \int_{0}^{t} \mu(t') g_n(t')dt',
$$
where $g_0$ is a nonnegative number. Then we have
$$
\sup\limits_{t \in [0,\infty)} g_{n+1}(t) \leq \sum_{k=0}^{n} \frac{a_{n-k} \| \mu\|^k_{L^{1}(0,\infty)}}{k!} + \frac{g_0 \| \mu\|^{n+1}_{L^{1}(0,\infty)}}{(n+1)!}.
$$
\end{lemma}
\begin{proof}[\textbf{\emph{Proof.}}]
Using the iterative inequality, we first have
\begin{equation}\label{p1}
\begin{aligned}
g_{n+1}(t) & \leq a_n + \int_{0}^{t} \mu(t_2) \left(a_{n-1} + \int_{0}^{t_2} \mu(t_1) g_{n-1}(t_1) dt_1 \right)d t_2 \\
& \leq a_n + a_{n-1} \int_{0}^{t} \mu(t_2)dt_2 + \int_{0}^{t} \int_{0}^{t_2} \mu(t_1) g_{n-1}(t_1) \mu(t_2)dt_1 dt_2 \\
& = a_n + a_{n-1} \int_{0}^{t} \mu(t_2) dt_2 + \int_{0}^{t} \mu(t_1) \left(\int_{t_1}^{t} \mu(t_2) dt_2 \right) g_{n-1}(t_1) dt_1.
\end{aligned}
\end{equation}
Inserting the inequality $g_{n-1}(t_1) \leq a_{n-2} + \int_{0}^{t_1} \mu(t_3) g_{n-2}(t_3)dt_3$ into \eqref{p1}, we have
\begin{equation}\label{p2}
\begin{aligned}
g_{n+1}(t) & \leq a_n + a_{n-1} \int_{0}^{t} \mu(t_2) dt_2 + a_{n-2} \int_{0}^{t} \mu(t_1) \left(\int_{t_1}^{t} \mu(t_2) dt_2\right) dt_1 \\
& \quad + \int_{0}^{t} \mu(t_1) \left(\int_{t_1}^{t} \mu(t_2) dt_2\right) \left(\int_{0}^{t_1} \mu(t_3) g_{n-2}(t_3) dt_3\right) dt_1.
\end{aligned}
\end{equation}
For the last two terms on the right hand side of \eqref{p2}, we have
$$
\int_{0}^{t} \mu(t_1) \left(\int_{t_1}^{t} \mu(t_2) dt_2\right) dt_1 = \frac{1}{2} \left(\int_{0}^{t} \mu(t_2) dt_2\right)^2,
$$
and
\begin{equation*}
\begin{aligned}
\int_{0}^{t} & \mu(t_1) \left(\int_{t_1}^{t} \mu(t_2) dt_2\right) \left(\int_{0}^{t_1} \mu(t_3) g_{n-2}(t_3) dt_3 \right) dt_1 \\
& = - \frac{1}{2} \int_{0}^{t} \left(\int_{0}^{t_1} \mu(t_3)g_{n-2}(t_3) dt_3\right) dt_1 \left[ \left( \int_{t_1}^{t} \mu(t_2) dt_2 \right)^2 \right] \\
& = \frac{1}{2} \int_{0}^{t} \mu(t_1) \left(\int_{t_1}^{t} \mu(t_2) dt_2\right)^2 g_{n-2}(t_1)dt_1.
\end{aligned}
\end{equation*}
Hence we get
\begin{equation}\label{q3}
\begin{aligned}
g_{n+1}(t) & \leq a_n + a_{n-1} \int_{0}^{t} \mu(t_2) dt_2 + \frac{a_{n-2}}{2} \left(\int_{0}^{t} \mu(t_2) dt_2\right)^2 \\
& + \frac{1}{2} \int_{0}^{t} \mu(t_1) \left(\int_{t_1}^{t} \mu(t_2) dt_2\right)^2 g_{n-2}(t_1) dt_1.
\end{aligned}
\end{equation}
For the integrand $g_{n-2}(t_1) $ in \eqref{q3}, we use the iterative inequality again and obtain
\begin{equation*}
\begin{aligned}
g_{n+1}(t) & \leq a_n + a_{n-1} \int_{0}^{t} \mu(t_2) dt_2 + \frac{a_{n-2}}{2} \left(\int_{0}^{t} \mu(t_2) dt_2\right)^2 \\
& \quad + \frac{a_{n-3}}{2} \int_{0}^{t} \mu(t_1) \left(\int_{t_1}^{t} \mu(t_2) dt_2\right)^2 dt_1 \\
& \quad + \frac{1}{2} \int_{0}^{t} \mu(t_1) \left(\int_{t_1}^{t} \mu(t_2) dt_2\right)^2 \left(\int_{0}^{t_1} \mu(t_4) g_{n-3}(t_4) dt_4\right) dt_1 \\
& \leq \sum_{k=0}^{3} \frac{a_{n-k}}{k!} \left(\int_{0}^{t} \mu(t_2) dt_2\right)^k + \frac{1}{3!} \int_{0}^{t} \mu(t_1) \left(\int_{t_1}^{t} \mu(t_2) dt_2\right)^3 g_{n-3}(t_1)dt_1.
\end{aligned}
\end{equation*}
Following the similar procedure for several times, we finally obtain
\begin{equation*}
\begin{aligned}
g_{n+1}(t) & \leq \sum_{k=0}^{n} \frac{a_{n-k}}{k!} \left(\int_{0}^{t} \mu(t_2) dt_2\right)^k + \frac{g_0}{n!} \int_{0}^{t} \mu(t_1) \left( \int_{t_1}^{t} \mu(t_2) dt_2\right)^n dt_1 \\
& = \sum_{k=0}^{n} \frac{a_{n-k}}{k!} \left(\int_{0}^{t} \mu(t_2) dt_2\right)^k + \frac{g_0}{{n+1}!} \left(\int_{0}^{t} \mu(t_2) dt_2\right)^{n+1} \\
& \leq \sum_{k=0}^{n} \frac{a_{n-k} \| \mu \|^{k}_{L^1(0,\infty)}}{k!} + \frac{g_0 \| \mu \|^{n+1}_{L^1(0,\infty)}}{(n+1)!}.
\end{aligned}
\end{equation*}
The proof of Lemma \ref{lem:RL} is completed.
\end{proof}

By taking $a_n = C 2^{-n}$ and $ \mu(t)=C(|\alpha(t)| + |\beta(t)| + |\gamma(t)| + |\xi(t)|)$, we get from Lemma \ref{lem:RL} that
\begin{equation}
\begin{aligned}
\sup\limits_{t \in [0,\infty)} \left\| (u^{(n+m+1)} - u^{(n+1)})(t) \right\|_{B^{s-1}_{p,r}}
\leq \sum_{k=0}^{n} \frac{C 2^{-n}}{k!} \left(\frac{\ln 2}{6C^3 h^2(H_0)}\right)^k
+ \frac{g_0}{(n+1)!} \left(\frac{\ln 2}{12C^3 h^2(H_0)}\right)^{n+1}, \notag
\end{aligned}
\end{equation}
which implies that
\begin{equation}
\begin{aligned}
\lim\limits_{n \to \infty} \sup\limits_{t \in [0,\infty)} \left\| (u^{(n+m+1)} - u^{(n+1)})(t) \right\|_{B^{s-1}_{p,r}} = 0. \notag
\end{aligned}
\end{equation}
Therefore, the sequence $(u^{(n)})_{n \geq 1}$ converges strongly in the Banach space $C([0,\infty);B^{s-1}_{p,r})$, and we denote the limit function by $u$.

\subsection{Existence}
In this step we shall verify the limit function $u$ indeed belongs to $E^s_{p,r}(\infty)$ and is a strong solution to the system \eqref{SCNS-t1}. Since the sequence $(u^{(n)})_{n \geq 1}$ is uniformly bounded in $ L^{\infty}([0,\infty);B^{s}_{p,r})$, and $u^{(n)} \to u$ strongly in $B^{s-1}_{p,r} \hookrightarrow \mathscr{S}' $ as $n \to \infty$, it follows from the Fatou's lemma (cf. Lemma \ref{lem:FT}) that $u \in L^{\infty}([0,\infty);B^{s}_{p,r})$.

On the other hand, as $(u^{(n)})_{n \geq 1}$ converges strongly to $u$ in $C([0,\infty);B^{s-1}_{p,r})$, an interpolation argument insures that the convergence holds in $C([0,\infty);B^{s'}_{p,r})$, for any $s' < s$. It is then easy to pass to the limit in the system \eqref{SCNS-t2} and to conclude that $u$ is indeed a solution to \eqref{SCNS-t1}. Thanks to the fact that $u$ belongs to $C([0,\infty);B^{s}_{p,r})$, we know that the right-hand side of the equation
\begin{equation}
\begin{aligned}
 \partial_t u + (\alpha(t) + \beta(t) u) \nabla S \cdot \nabla u = -(\gamma(t) u + \xi(t) u^2) \Delta S, \notag
\end{aligned}
\end{equation}
belongs to $L^1([0,\infty);B^{s}_{p,r})$. In particular, for the case $r < \infty$, applying Lemma \ref{prior estimates} implies that $u \in C([0,\infty);B^{s'}_{p,r})$ , for any $s' < s$. Finally, using the system \eqref{SCNS-t1} again, we see that $ \partial_t u \in C([0,\infty);B^{s'-1}_{p,r})$ for $r < \infty$, and in $ L^{\infty}([0,\infty);B^{s-1}_{p,r}) $ otherwise. Therefore, $u$ belongs to $E^s_{p,r}(\infty)$. Moreover, a standard use of a sequence of viscosity approximate solutions $(u_{\varepsilon})_{\varepsilon > 0}$ for system \eqref{SCNS-t1} which converges uniformly in
\begin{equation}
\begin{aligned}
  C([0,\infty);B^{s}_{p,r}) \cap C^1 ([0,\infty);B^{s-1}_{p,r}), \notag
\end{aligned}
\end{equation}
gives the continuity of solution $u$ in $E^s_{p,r}(\infty)$.

\subsection{Uniqueness and stability}
Uniqueness and continuity with respect to the initial data are immediate consequence of the following result.
\begin{proposition}
Let $1 \leq p,r \leq +\infty$ and $ s > 1+\frac{d}{p}$. Suppose that there exists two solutions $u_1(t),u_2(t) \in \{L^{\infty}([0,\infty);B^{s}_{p,r}) \cap C([0,\infty);\mathscr{S}')\}^2$ of the initial-value problem \eqref{SCNS-t1} with the initial data $u_1(0),u_2(0) \in B^{s}_{p,r}$. Then for all $ t \in [0,\infty)$, we have
\begin{enumerate}[label=(\roman*)]
\item If $s > 1+ \frac{d}{p}$, and $s \ne 2+ \frac{d}{p}$, we have
\begin{equation}\label{SCNS-t11}
\begin{aligned}
  \| u_1(t) & - u_2(t) \|_{B^{s-1}_{p,r}} \leq \| u_1(0) -u_2(0) \|_{B^{s-1}_{p,r}} \\
& \times e^{C \int_{0}^{\infty} (|\alpha(\tau)| + |\beta(\tau)| + |\gamma(\tau)| + |\xi(\tau)|) \cdot (1 + \| u_1 \|_{B^{s}_{p,r}} + \| u_2 \|_{B^{s}_{p,r}})^2 d\tau}.
\end{aligned}
\end{equation}
\item If $s = 2+ \frac{d}{p}$, we have
\begin{equation}\label{SCNS-t12}
\begin{aligned}
  \| u_1(t) & - u_2(t) \|_{B^{s-1}_{p,r}} \leq \| u_1(0) -u_2(0) \|^{\theta}_{B^{s-1}_{p,r}} (\| u_1\|_{B^{s}_{p,r}} + \| u_2 \|_{B^{s}_{p,r}})^{1 - \theta}\\
& \times e^{C \theta \int_{0}^{\infty} (|\alpha(\tau)| + |\beta(\tau)| + |\gamma(\tau)| + |\xi(\tau)|) \cdot (1 + \| u_1 \|_{B^{s}_{p,r}} + \| u_2 \|_{B^{s}_{p,r}})^2 d\tau},
\end{aligned}
\end{equation}
where $ \theta \in (0,1)$.
\end{enumerate}
\end{proposition}

\begin{proof}[\textbf{\emph{Proof.}}]
Let $\omega = u_1 - u_2$, we can know that $ \omega \in L^{\infty}([0,\infty);B^{s}_{p,r}) \cap C([0,\infty);\mathscr{S}') $, which implies that $ \omega \in C([0,\infty);B^{s-1}_{p,r}) $, and $\omega$ is the solution of the following transport equation
\begin{equation}
\left\{
\begin{aligned}
& \partial_t \omega + v_1 \nabla \omega = (v_2 - v_1) \nabla u_2 + f(u_1,u_2), \\
& S_1 = (I - \Delta)^{-1} u_1,S_2 = (I - \Delta)^{-1} u_2, \\
& \omega(0,x) = u_1(0,x) - u_2(0,x),
\end{aligned}
\right.
\end{equation}
with
\begin{equation}
\begin{aligned}
  f(u_1,u_2)= (\gamma(t) u_2 + \xi(t) (u_2)^2) \Delta S_2-(\gamma(t) u_1 + \xi(t) (u_1)^2) \Delta S_1,\notag
\end{aligned}
\end{equation}
and
\begin{equation}
\begin{aligned}
 v_i=(\alpha(t) + \beta(t) u_i) \nabla S_i , \quad i=1,2. \notag
\end{aligned}
\end{equation}
Now, we prove that the case (i) holds. Lemma \ref{prior estimates} implies that
\begin{equation}\label{SCNS-a1}
\begin{aligned}
  \| \omega(t) \|_{B^{s-1}_{p,r}}  & \leq \| \omega(0) \|_{B^{s-1}_{p,r}} + C \int_{0}^{t} V'(\tau) \| \omega \|_{B^{s-1}_{p,r}} d\tau \\
& + \int_{0}^{t}(\| (v_2 - v_1)\nabla u_2 \|_{B^{s-1}_{p,r}} + \| f(u_1,u_2) \|_{B^{s-1}_{p,r}})d\tau,
\end{aligned}
\end{equation}
where $ V(t) = \int_{0}^{t} \| \nabla v_1(\tau) \|_{B^{\frac{d}{p}}_{p,r} \cap L^{\infty}} d\tau$ if $s < 2+\frac{d}{p}$, and $ V(t) = \int_{0}^{t} \| \nabla v_1(\tau) \|_{B^{s-2}_{p,r}} d\tau$ else.

For the case $1+\frac{d}{p} < s < 2+\frac{d}{p}$, using the embedding $B^{s-1}_{p,r} \hookrightarrow L^{\infty} (s-1>\frac{d}{p})$ directly, we get
\begin{equation}\label{SCNS-a2}
\begin{aligned}
  \| \nabla v_1(\tau) \|_{B^{\frac{d}{p}}_{p,r} \cap L^{\infty}} & \leq C \| \nabla v_1(\tau) \|_{B^{s-1}_{p,r}} = C \| \nabla [(\alpha(\tau) + \beta(\tau) u_1) \nabla S_1] \|_{B^{s-1}_{p,r}} \\
%& \leq C \| (\alpha(\tau) + \beta(\tau) u_1) \nabla S_1 \|_{B^{s}_{p,r}} \\
& \leq C (|\alpha(\tau)| \cdot \| \nabla S_1 \|_{B^{s}_{p,r}} + |\beta(\tau)| \cdot \| u_1 \nabla S_1 \|_{B^{s}_{p,r}}) \\
& \leq C (|\alpha(\tau)| \cdot \|\nabla (I - \Delta)^{-1} u_1 \|_{B^{s}_{p,r}} \\
& \quad + |\beta(\tau)| \cdot \| u_1 \|_{B^{s}_{p,r}} \|\nabla (I - \Delta)^{-1} u_1 \|_{B^{s}_{p,r}}) \\
& \leq C (|\alpha(\tau)| \cdot \|u_1 \|_{B^{s}_{p,r}} + |\beta(\tau)| \cdot \| u_1 \|^2_{B^{s}_{p,r}}).
\end{aligned}
\end{equation}
by noticing the fact that $B^{s}_{p,r}$ was an algebra provided $s > 1+\frac{d}{p}$.

Applying this algebraic property for $B^{s-1}_{p,r}$ with $s>1+\frac{d}{p}$ again, we have
\begin{equation}
\begin{aligned}
  \| (v_2 - v_1)\nabla u_2 \|_{B^{s-1}_{p,r}} & = \| v_2 - v_1 \|_{B^{s-1}_{p,r}} \cdot \|\nabla u_2 \|_{B^{s-1}_{p,r}} \\
%& \leq C(\| (\alpha(\tau) + \beta(\tau) u_1) \nabla S_1 -(\alpha(\tau) + \beta(\tau) u_2) \nabla S_2 \|_{B^{s-1}_{p,r}}) \cdot \| u_2 \|_{B^{s}_{p,r}} \\
& \leq C ( |\alpha(\tau)| \cdot \| \nabla S_1 - \nabla S_2\|_{B^{s-1}_{p,r}} + |\beta(\tau)| \\
& \quad \times \| u_1 \nabla S_1 - u_2 \nabla S_2\|_{B^{s-1}_{p,r}}) \cdot \| u_2 \|_{B^{s}_{p,r}} \\
& \leq C \bigg( |\alpha(\tau)| \cdot \| \omega\|_{B^{s-1}_{p,r}} + |\beta(\tau)| \cdot (\| \omega\|_{B^{s-1}_{p,r}} \| \nabla (I - \Delta)^{-1} u_1\|_{B^{s-1}_{p,r}} \\
& \quad + \| \omega\|_{B^{s-1}_{p,r}} \cdot \| u_2\|_{B^{s-1}_{p,r}}) \bigg) \cdot \| u_2 \|_{B^{s}_{p,r}} \\
& \leq C \bigg(|\alpha(\tau)| \cdot \| u_2\|_{B^{s}_{p,r}} + |\beta(\tau)| \cdot (\| u_1\|_{B^{s}_{p,r}} \| u_2\|_{B^{s}_{p,r}} +\| u_2\|^2_{B^{s}_{p,r}}) \bigg) \cdot \| \omega \|_{B^{s-1}_{p,r}}.
\end{aligned}
\end{equation}
and
\begin{equation}\label{SCNS-a3}
\begin{aligned}
  \| f(u_1,u_2) \|_{B^{s-1}_{p,r}} & = \| (\gamma(\tau) u_2 + \xi(\tau) (u_2)^2) \Delta S_2-(\gamma(\tau) u_1 + \xi(\tau) (u_1)^2) \Delta S_1 \|_{B^{s-1}_{p,r}} \\
%& \leq C ( |\gamma(\tau)| \cdot \| u_2 \Delta S_2 - u_1 \Delta S_1\|_{B^{s-1}_{p,r}} + |\xi(\tau)| \cdot \| (u_2)^2 \Delta S_2 - (u_1)^2 \Delta S_1\|_{B^{s-1}_{p,r}}) \\
& \leq C ( |\gamma(\tau)| \cdot \| \omega \Delta S_2 + u_1 (\Delta S_2 - \Delta S_1)\|_{B^{s-1}_{p,r}} \\
& \quad + |\xi(\tau)| \cdot \| ((u_2)^2 - (u_1)^2) \Delta S_2 + (u_1)^2 (\Delta S_2 -\Delta S_1)\|_{B^{s-1}_{p,r}} )\\
&  \leq C ( |\gamma(\tau)| \cdot \| \omega \|_{B^{s-1}_{p,r}} (\| u_1 \|_{B^{s}_{p,r}} + \| u_2 \|_{B^{s}_{p,r}}) + |\xi(\tau)| \cdot \| \omega \|_{B^{s-1}_{p,r}} (\| u_1 \|_{B^{s}_{p,r}} + \| u_2 \|_{B^{s}_{p,r}})^2) \\
&  \leq C (|\gamma(\tau)| + |\xi(\tau)|) \cdot \| \omega \|_{B^{s-1}_{p,r}} \cdot (1 + \| u_1 \|_{B^{s}_{p,r}} + \| u_2 \|_{B^{s}_{p,r}})^2.
\end{aligned}
\end{equation}
Therefore, submitting \eqref{SCNS-a2} - \eqref{SCNS-a3} into \eqref{SCNS-a1} yields
\begin{equation}
\begin{aligned}
  \| \omega(t) \|_{B^{s-1}_{p,r}} & \leq \| \omega(0) \|_{B^{s-1}_{p,r}} + C \int_{0}^{\infty}(|\alpha(\tau)| + |\beta(\tau)| + |\gamma(\tau)| + |\xi(\tau)|) \\
& \quad \times (1 + \| u_1 \|_{B^{s}_{p,r}} + \| u_2 \|_{B^{s}_{p,r}})^2 \| \omega \|_{B^{s-1}_{p,r}} d\tau.
\end{aligned}
\end{equation}

On the other hand, it is easy to see that the above inequality can also hold for $s>2+\frac{d}{p}$. Hence, applying Gronwall's inequality can lead to the result inequality \eqref{SCNS-t11}.

At last, we use the complex interpolation property to prove (ii). Let $ 1+\frac{d}{p} = \theta
(1 + \frac{d}{2p}) + (1 - \theta)(2 + \frac{d}{2p})$, and $\theta = 1 - \frac{d}{2p} \in (0,1)$. According to Lemma \ref{interpolation}(1), we get
\begin{equation}
\begin{aligned}
  \| u_1 - u_2 \|_{B^{s-1}_{p,r}} & \leq \| u_1 - u_2 \|^{\theta}_{B^{1 + \frac{d}{2p}}_{p,r}} \| u_1 - u_2 \|^{1 - \theta}_{B^{2 + \frac{d}{2p}}_{p,r}} \\
& \leq C \| u_1 - u_2 \|^{\theta}_{B^{1 + \frac{d}{2p}}_{p,r}} (\| u_1 \|_{B^{2 + \frac{d}{2p}}_{p,r}} + \| u_2 \|_{B^{2 + \frac{d}{2p}}_{p,r}})^{1 - \theta} \\
& \leq ( \| u_1 \|_{B^{2 + \frac{d}{2p}}_{p,r}} + \| u_2 \|_{B^{2 + \frac{d}{2p}}_{p,r}} )^{1-\theta} \| u_1(0) - u_2(0) \|^{\theta}_{B^{1 + \frac{d}{2p}}_{p,r}}
\\
& \times e^{C\theta \int_{0}^{\infty} \left( |\alpha(\tau)| + |\beta(\tau)| + |\gamma(\tau)| + |\xi(\tau)| \right) \cdot (1 + \| u_1 \|_{B^{2 + \frac{d}{2p}}_{p,r}} + \| u_2 \|_{B^{2 +\frac{d}{2p}}_{p,r}})^2 \, d\tau} \\
& \leq (\| u_1 \|_{B^{s}_{p,r}} + \| u_2 \|_{B^{s}_{p,r}})^{1-\theta} \| u_1(0) - u_2(0) \|^{\theta}_{B^{s-1}_{p,r}} \\
& \times e^{C\theta \int_{0}^{\infty}(|\alpha(\tau)| + |\beta(\tau)| + |\gamma(\tau)| + |\xi(\tau)|) \cdot (1 + \| u_1 \|_{B^{s}_{p,r}} + \| u_2 \|_{B^{s}_{p,r}})^2 d\tau},\\
\end{aligned}
\end{equation}
which yields \eqref{SCNS-t12}.

Therefore, the inequality \eqref{SCNS-t11} and \eqref{SCNS-t12} implies the uniqueness and continuity with respect to the initial data.
\end{proof}

\section{Global well-posedness in critical Besov Spaces}
%----------------------------------------------------
\subsection{Approximate solutions}
We construct the approximate solutions via the Friedrichs iterative method. Starting from $u^{(1)} \overset{\cdot}{=} S_1 u_0$, we recursively define a sequence of functions $(u^{(n)})_{n \geq 1}$ by solving the following linear transport equations
\begin{equation}\label{SCNS-d1}
\left\{
\begin{aligned}
& \partial_t u^{(n+1)} + (\alpha(t) + \beta(t) u^{(n)}) \nabla S^{(n)} \cdot \nabla u^{(n+1)} = -(\gamma(t) u^{(n)} + \xi(t) (u^{(n)})^2) \Delta S^{(n)},\\
& S^{(n)} = (I - \Delta)^{-1} u^{(n)},\\
& u^{(n+1)}\big|_{t = 0} = S_{n+1}u_0.
\end{aligned}
\right.
\end{equation}
This approximation system \eqref{SCNS-d1} is similar to the non-cridical case, we adopt it here just for convenience. It follows from the definition of frequency truncation operator that $S_{n+1} u_0 \in \mathop{\cap}\limits_{s \in \mathbb{R}} B^s_{2,1}$. Assume by induction that, given $n \in  \mathbb{N}$ and any positive $T$, the approximate solution $u^{(n)} \in L^{\infty}([0,T];B^{1+\frac{d}{2}}_{2,1})$. Since the Besov space $B^{1+\frac{d}{2}}_{2,1}$ is a Banach algebra, we have
\begin{equation}\label{d2}
\begin{aligned}
\int_{0}^{T}\|[\gamma(t)u^{(n)} &  + \xi(t)(u^{(n)})^2]\Delta S^{(n)}\|_{B^{1+\frac{d}{2}}_{2,1}} dt \\
& \le C \int_{0}^{T} \|\gamma(t)u^{(n)} + \xi(t)(u^{(n)})^2\|_{B^{1+\frac{d}{2}}_{2,1}} \cdot \|(I - \Delta)^{-1} \Delta u^{(n)}\|_{B^{1+\frac{d}{2}}_{2,1}} dt \\
%& \le C \int_{0}^{T} [|\gamma(t)| \cdot \|u^{(n)}\|_{B^{1+\frac{d}{2}}_{2,1}} + |\xi(t)| \cdot \|u^{(n)}\|_{B^{1+\frac{d}{2}}_{2,1}}^2 ] \cdot \| u^{(n)}\|_{B^{1 +\frac{d}{2}}_{2,1}} dt \\
& \le C \int_{0}^{T} [|\gamma(t)| \cdot \|u^{(n)}\|_{B^{1+\frac{d}{2}}_{2,1}}^2 +  |\xi(t)| \cdot \|u^{(n)}\|_{B^{1+\frac{d}{2}}_{2,1}}^3 ] dt \\
& \le C \int_{0}^{T} (|\gamma(t)| + |\xi(t)|) (\|u^{(n)}\|_{B^{1+\frac{d}{2}}_{2,1}}^2 + \|u^{(n)}\|_{B^{1+\frac{d}{2}}_{2,1}}^3 ) dt < \infty ,
\end{aligned}
\end{equation}
where we have used $\alpha$, $\beta$, $\gamma$, $\xi \in L^1([0,\infty);\mathbb{R})$. And the above estimate implies that $[\gamma(t)u^{(n)} + \xi(t)(u^{(n)})^2]\Delta S^{(n)} \in L^1([0,T];B^{1+\frac{d}{2}}_{2,1})$ for any positive $T$. Thereby the nonlinear terms on the right hand side of system \eqref{SCNS-d1} belong to $L^1([0,T];B^{1+\frac{d}{2}}_{2,1})$ for any positive $T$.

Furthermore, since
\begin{equation}\label{d3}
\begin{aligned}
\int_{0}^{T} \|\nabla [(\alpha(t) & + \beta(t) u^{(n)}) \nabla S^{(n)}]\|_{B^{\frac{d}{2}}_{2,1}} dt \\
& \le C \int_{0}^{T} \|(\alpha(t) + \beta(t) u^{(n)}) \cdot (I - \Delta)^{-1} \nabla u^{(n)}\|_{B^{1+\frac{d}{2}}_{2,1}} dt \\
%& \le C \int_{0}^{T} [|\alpha(t)| + |\beta(t)| \cdot \|u^{(n)}\|_{B^{1+\frac{d}{2}}_{2,1}} ] \cdot \| u^{(n)}\|_{B^{\frac{d}{2}}_{2,1}} dt \\
& \le C \int_{0}^{T} [|\alpha(t)| \cdot \|u^{(n)}\|_{B^{1+\frac{d}{2}}_{2,1}} + |\beta(t)| \cdot \|u^{(n)}\|_{B^{1+\frac{d}{2}}_{2,1}}^2 ] dt \\
& \le C \int_{0}^{T} (|\alpha(t)| +|\beta(t)|) (\|u^{(n)}\|_{B^{1 +\frac{d}{2}}_{2,1}} + \|u^{(n)}\|_{B^{1 + \frac{d}{2}}_{2,1}}^2)  dt < \infty ,
\end{aligned}
\end{equation}
one can deduce that $\nabla [(\alpha(t) + \beta(t) u^{(n)}) \nabla S^{(n)}] \in L^1([0,T];B^{\frac{d}{2}}_{2,1})$ for any positive $T$. Thanks to Lemma \ref{euth}, the system \eqref{SCNS-d1} has a unique solution $u^{(n+1)} \in C([0,\infty];B^{1 + \frac{d}{2}}_{2,1})$.

\subsection{Uniform bound}
By applying Lemma \ref{prior estimates} to the system \eqref{SCNS-d1} with respect to $u^{(n)}$, we get
\begin{equation}\label{SCNS-d2}
\begin{aligned}
  \|u^{(n+1)}(t)\|_{B^{1 + \frac{d}{2}}_{2,1}} & \leq \|S_{n+1} u_0\|_{B^{1 + \frac{d}{2}}_{2,1}} + \int_0^t \|[\gamma(\tau)u^{(n)} + \xi(\tau)(u^{(n)})^2]\Delta S^{(n)}\|_{B^{1 + \frac{d}{2}}_{2,1}} d\tau \\
& \quad + C \int_0^t V'(\tau) \|u^{(n+1)}(\tau)\|_{B^{1 + \frac{d}{2}}_{2,1}} d\tau,\\
\end{aligned}
\end{equation}
where $V(t) = \int_0^t \|\nabla [(\alpha(\tau) + \beta(\tau) u^{(n)}) \nabla S^{(n)}](\tau)\|_{B^{\frac{d}{2}}_{2,1}}d\tau$.

For the terms on the right-hand side of \eqref{SCNS-d2}, we first get from the property of $\chi \left( \cdot \right)$ that $\|S_{n+1}u_0\|_{B^s_{p,r}} \le C\|u_0\|_{B^s_{p,r}}$, for some positive constant independent of n. Since $B^{1 + \frac{d}{2}}_{2,1}(\mathbb{R}^d)$ is a Banach algebra, and the operator $(I - \Delta)^{-1}$ is a $S^{-2}$-multiplier, we can obtain the following results from \eqref{d2} and \eqref{d3}.
\begin{equation}
\begin{aligned}
\|[\gamma(\tau)u^{(n)} + \xi(\tau)(u^{(n)})^2]\Delta S^{(n)}\|_{B^{1 + \frac{d}{2}}_{2,1}} \le C (|\gamma(\tau)| + |\xi(\tau)|) (\|u^{(n)}\|_{B^{1+\frac{d}{2}}_{2,1}}^2 + \|u^{(n)}\|_{B^{1+\frac{d}{2}}_{2,1}}^3 ),\notag
\end{aligned}
\end{equation}
and
\begin{equation}
\begin{aligned}
\|\nabla [(\alpha(\tau) + \beta(\tau) u^{(n)}) \nabla S^{(n)}]\|_{B^{\frac{d}{2}}_{2,1}} \le C (|\alpha(\tau)| +|\beta(\tau)|) (\|u^{(n)}\|_{B^{1 +\frac{d}{2}}_{2,1}} + \|u^{(n)}\|_{B^{1 + \frac{d}{2}}_{2,1}}^2). \notag
\end{aligned}
\end{equation}
Plugging the last two estimates into \eqref{SCNS-d2} and using the Young's inequality lead to
\begin{equation}
\begin{aligned}
  \|u^{(n+1)}(t)\|_{B^{1 + \frac{d}{2}}_{2,1}} & \leq C \|u_0\|_{B^{1 + \frac{d}{2}}_{2,1}} + C \int_0^t (|\gamma(\tau)| + |\xi(\tau)|) (\|u^{(n)}\|_{B^{1+\frac{d}{2}}_{2,1}}^2 + \|u^{(n)}\|_{B^{1+\frac{d}{2}}_{2,1}}^3 ) d\tau \\
& \quad + C \int_0^t (|\alpha(\tau)| +|\beta(\tau)|) (\|u^{(n)}\|_{B^{1 +\frac{d}{2}}_{2,1}} + \|u^{(n)}\|_{B^{1 + \frac{d}{2}}_{2,1}}^2) \cdot \|u^{(n+1)}(\tau)\|_{B^{1 + \frac{d}{2}}_{2,1}} d\tau\\
& \leq C \|u_0\|_{B^{1 + \frac{d}{2}}_{2,1}} + C \int_0^t (|\gamma(\tau)| + |\xi(\tau)|)(1+\|u^{(n)}(\tau)\|_{B^{1 + \frac{d}{2}}_{2,1}})^3 d\tau \\
& \quad + C \int_0^t (|\alpha(\tau)| + |\beta(\tau)|)(1+\|u^{(n)}(\tau)\|_{B^{1 + \frac{d}{2}}_{2,1}})^2 \cdot \|u^{(n+1)}(\tau)\|_{B^{1 + \frac{d}{2}}_{2,1}} d\tau. \notag
\end{aligned}
\end{equation}
Using Gronwall lemma to last inequality leads to
\begin{equation}\label{SCNS-d3}
\begin{aligned}
  1 + \|u^{(n+1)}(t)\|_{B^{1 + \frac{d}{2}}_{2,1}} & \leq C e^{C \int_0^t [|\alpha(\tau)| + |\beta(\tau)|](1+\|u^{(n)}(\tau)\|_{B^{1 + \frac{d}{2}}_{2,1}})^2 d\tau} \\
& \quad \times \left(1 + \|u_0\|_{B^{1 + \frac{d}{2}}_{2,1}} + \int_0^t [|\gamma(\tau)| + |\xi(\tau)|](1+\|u^{(n)}(\tau)\|_{B^{1 + \frac{d}{2}}_{2,1}})^3 d\tau \right).
\end{aligned}
\end{equation}

Define
\begin{equation}
	  H_0 \stackrel{\text{def}}{=} H^{(n)}(0)=1+\|u_0\|_{B^{1 + \frac{d}{2}}_{2,1}}, \notag
\end{equation}
and
\begin{equation}
	  H^{(n)}(t) \stackrel{\text{def}}{=} 1+\|u^{(n)}(t)\|_{B^{1 + \frac{d}{2}}_{2,1}},
	\quad \text{for all } n \geq 1. \notag
\end{equation}
It follows from the inequality \eqref{SCNS-d3} that
\begin{equation}\label{SCNS-d5}
\begin{aligned}
  H^{(n+1)}(t) & \leq C e^{C \int_0^t [|\alpha(\tau)| + |\beta(\tau)|](H^{(n)}(\tau))^2 d\tau} \\
& \quad \times \left(H_0  + \int_0^t [|\gamma(\tau)| + |\xi(\tau)|](H^{(n)}(\tau))^3 d\tau \right).
\end{aligned}
\end{equation}

Similar to \eqref{SCNS-t5} in Section 3, the uniform boundedness of the $H^{(n+1)}(t)$ can be proved by the same method. Therefore, we can also obtain the following conclusion by means of mathematical induction with respect to $n$, namely
\begin{equation}
\sup_{n \geq 1}\sup_{t \in [0, \infty)} H^{(n+1)}(t) \leq 2C h(H_0),\notag
\end{equation}
which implies the uniform bound
\begin{equation}\label{c-uniform-bound}
\sup_{n \geq 1} \| u^{(n)} \|_{L^\infty([0, \infty); B^{1 + \frac{d}{2}}_{2,1})} \leq 2C h(H_0).
\end{equation}

As a result, the approximate solutions $(u^{(n)})_{n \geq 1}$ is uniformly bounded in $C([0,\infty);B^{1 + \frac{d}{2}}_{2,1})$. Moreover, using the system \eqref{SCNS-d1} itself, one can verify that the sequence $(\partial_t u^{(n)})_{n \geq 1}$ is uniformly bounded in $C([0,\infty);B^{\frac{d}{2}}_{2,1})$. Therefore, we obtain that $(u^{(n)})_{n \geq 1}$ is uniformly bounded in $E^{1 + \frac{d}{2}}_{2,1}(\infty)$.

\subsection{Convergence}
We first show that the approximate solutions $(u^{(n)})_{n \geq 1}$ is a Cauchy sequence in $C([0,\infty);B^{\frac{d}{2}}_{2,\infty}) $, and then extend the convergent result to $C([0,\infty);B^{\frac{d}{2}}_{2,1}) $ by using an interpolation argument, which will be more technique than the non-critical case. To this end, for any $(m,n) \in \mathbb{N^+} \times \mathbb{N^+}$, we get from \eqref{SCNS-d1} that
\begin{equation}\label{SCNS-d8}
\begin{aligned}
  [\partial_t &+ (\alpha(t) + \beta(t) u^{(n+m)}) \nabla S^{(n+m)}\cdot \nabla]  (u^{(n+m+1)}-u^{(n+1)})
&= -[(\alpha(t) + \beta(t) u^{(n+m)}) \cdot \nabla S^{(n+m)} - (\alpha(t) + \beta(t) u^{(n)}) \nabla S^{(n)}] \\
& \quad \times  \nabla u^{(n+1)} - f(u^{(n)},u^{(n+m)}),
\end{aligned}
\end{equation}
with the initial conditions
\begin{equation}
\begin{aligned}
  (u^{(n+m+1)} - u^{(n+1)})(0,x) = S_{n+m+1}u_0(x) - S_{n+1}u_0(x),\notag
\end{aligned}
\end{equation}
where
$$f(u^{(n)},u^{(n+m)}) = (\gamma(t) u^{(n+m)} + \xi(t) (u^{(n+m)})^2) \Delta S^{(n+m)} - (\gamma(t) u^{(n)} + \xi(t) (u^{(n)})^2) \Delta S^{(n)}.$$

Applying the Lemma \ref{prior estimates} to \eqref{SCNS-d8}, we deduce that
\begin{equation}\label{SCNS-d9}
\begin{aligned}
\| & (u^{(n+m+1)} - u^{(n+1)})(t) \|_{B^{\frac{d}{2}}_{2,\infty} } \\
& \leq e ^{C \int_0^t \| \nabla [(\alpha(\tau) + \beta(\tau) \cdot u^{(n+m)}) \nabla S^{(n+m)}] \|_{B^{\frac{d}{2}}_{2,\infty} \cap L^\infty} \, d\tau} \| S_{n+m+1}u_0 - S_{n+1}u_0 \|_{B^{\frac{d}{2}}_{2,\infty} } \\
& \quad + \int_0^t e ^{C \int_{\tau}^t \| \nabla [ (\alpha(t') + \beta(t') \cdot u^{(n+m)}) \nabla S^{(n+m)}] \|_{B^{\frac{d}{2}}_{2,\infty} \cap L^\infty} \, dt'} \\
& \quad \times \large(\| [(\alpha(\tau) + \beta(\tau) u^{(n+m)}) \cdot \nabla S^{(n+m)} - (\alpha(\tau) + \beta(\tau) u^{(n)}) \cdot \nabla S^{(n)}] \cdot \nabla u^{(n+1)} \|_{B^{\frac{d}{2}}_{2,\infty}} \\
& \quad + \|f(u^{(n)},u^{(n+m)}) \|_{B^{\frac{d}{2}}_{2,\infty}} \large) d\tau. \\
\end{aligned}
\end{equation}
Let us estimate the three terms on the right-hand side of \eqref{SCNS-d9}. First,we get by using the uniform bound for approximations that, for any $ 0 \leq \tau \leq t \leq \infty$,
\begin{equation}
\begin{aligned}
\int_{\tau}^t \| & \nabla [(\alpha(r) + \beta(r) \cdot u^{(n+m)}) \nabla S^{(n+m)}] \|_{B^{\frac{d}{2}}_{2,\infty} \cap L^{\infty} }dr  \\
& \leq C \int_{\tau}^t \left(|\alpha(r)| \cdot \|\nabla S^{(n+m)} \|_{B^{1 + \frac{d}{2}}_{2,\infty} \cap Lip} + |\beta(r)| \cdot \| u^{(n+m)} \nabla S^{(n+m)} \|_{B^{1 + \frac{d}{2}}_{2,\infty} \cap Lip} \right) dr  \\
& \leq C \int_{\tau}^t (|\alpha(r)| + |\beta(r)|) \cdot (1 + \|u^{(n+m)} \|_{B^{1 + \frac{d}{2}}_{2,1}} ) \cdot \| (I - \Delta)^{-1} u^{(n+m)} \|_{B^{2 + \frac{d}{2}}_{2,\infty}} \, dr  \\
& \leq C \int_{\tau}^t (|\alpha(r)| + |\beta(r)|) \cdot (\|u^{(n+m)} \|_{B^{1 + \frac{d}{2}}_{2,1}} + \|u^{(n+m)} \|^2_{B^{1 + \frac{d}{2}}_{2,1}} ) dr <  \infty. \\ \notag
\end{aligned}
\end{equation}
In terms of the definition of the cut-off low frequency operator, we have
\begin{equation*}
\begin{aligned}
\| S_{n+m+1}u_0 - S_{n+1}u_0 \|_{B^{\frac{d}{2}}_{2,\infty}} %& = \left\| \sum_{q=n+1}^{n+m} \Delta_q u_0 \right\|_{B^{\frac{d}{2}}_{2,\infty}} \\
& = \sup_{p\geq {-1}} 2^{\frac{d}{2} p } \left\| \Delta_p \sum_{q=n+1}^{n+m} \Delta_q u_0 \right\|_{L^2} \\
& \leq C \sum_{p=n}^{n+m+1} \big(2^{-p} 2^{(1+\frac{d}{2})p} \sum_{|p-q| \leq 1}  \left\| \Delta_p \Delta_q u_0 \right\|_{L^2} \big ) \\
& \leq C 2^{-n} \sum_{p=n}^{n+m+1} 2^{(1+\frac{d}{2})p} \left\| \Delta_p  u_0 \right\|_{L^2} \\
& \leq C 2^{-n} \| u_0 \|_{B^{1+\frac{d}{2}}_{2,1}}, \\
\end{aligned}
\end{equation*}
where we have used the commutativity of $\Delta_q S_n = S_n\Delta_q$, the orthogonality $\Delta_p\Delta_q = 0$ for $|p-q | \geq 2$ as well as the uniform estimates $\| \Delta_n u \|_{L^p}$, $\| S_n u \|_{L^p} \leq C \| u \|_{L^p}$, for all $1 \leq p \leq \infty$, $n \geq -1$, due to the Young's inequality, where the positive constant $C$ is independent of $n$.

For the second term on the right-hand side of \eqref{SCNS-d9}, we have
\begin{equation*}
\begin{aligned}
\big\| & \big[(\alpha(\tau) + \beta(\tau) u^{(n+m)}) \cdot \nabla S^{(n+m)} - (\alpha(\tau) + \beta(\tau) u^{(n)}) \cdot \nabla S^{(n)}] \cdot \nabla u^{(n+1)} \big \|_{B^{\frac{d}{2}}_{2,\infty}} \\
& \leq C [\big(|\alpha(\tau)| + |\beta(\tau)| \cdot \| u_n \|_{B^{\frac{d}{2}}_{2,\infty}\cap L^{\infty}} \big) \| \nabla (I - \Delta)^{-1} (u^{(n+m)}-u^{(n)}) \|_{B^{\frac{d}{2}}_{2,\infty}\cap L^{\infty}} \\
& \quad + |\beta(\tau)| \cdot \| (u^{(n+m)}-u^{(n)} \|_{B^{\frac{d}{2}}_{2,\infty}\cap L^{\infty}} \|  (I - \Delta)^{-1} \nabla u^{(n+m)} \|_{B^{\frac{d}{2}}_{2,\infty}\cap L^{\infty}} \big] \cdot \| \nabla u^{(n+1)} \|_{B^{\frac{d}{2}}_{2,\infty}\cap L^{\infty}}\\
%& \leq C [\big(|\alpha(\tau)| + |\beta(\tau)| \cdot \| u_n \|_{B^{\frac{d}{2}}_{2,1}} \big) \| \nabla (I - \Delta)^{-1} (u^{(n+m)}-u^{(n)}) \|_{B^{\frac{d}{2}}_{2,1}} \\
%& \quad + |\beta(\tau)| \cdot \| (u^{(n+m)}-u^{(n)} \|_{B^{\frac{d}{2}}_{2,1}} \|  (I - \Delta)^{-1} \nabla u^{(n+m)} \|_{B^{\frac{d}{2}}_{2,1}} \big] \cdot \| \nabla u^{(n+1)} \|_{B^{\frac{d}{2}}_{2,1}}\\
& \leq C [\big(|\alpha(\tau)| + |\beta(\tau)| \cdot \| u_n \|_{B^{1+\frac{d}{2}}_{2,1}} \big) \| (u^{(n+m)}-u^{(n)}) \|_{B^{\frac{d}{2}}_{2,1}} \\
& \quad + |\beta(\tau)| \cdot \| (u^{(n+m)}-u^{(n)} \|_{B^{\frac{d}{2}}_{2,1}} \| u^{(n+m)} \|_{B^{1+\frac{d}{2}}_{2,1}} \big] \cdot \| u^{(n+1)} \|_{B^{1+\frac{d}{2}}_{2,1}}\\
& \leq C \| (u^{(n+m)}-u^{(n)}) \|_{B^{\frac{d}{2}}_{2,1}} (|\alpha(\tau)| + |\beta(\tau)| \cdot \| u_n \|_{B^{1+\frac{d}{2}}_{2,1}} + |\beta(\tau)| \cdot \| u^{(n+m)} \|_{B^{1+\frac{d}{2}}_{2,1}}) \cdot \| u^{(n+1)} \|_{B^{1+\frac{d}{2}}_{2,1}} \\
%& \leq C (|\alpha(\tau)| + |\beta(\tau)| )(1+ \| u^{(n)} \|_{B^{1+\frac{d}{2}}_{2,1}} + \| u^{(n+m)} \|_{B^{1+\frac{d}{2}}_{2,1}}) \cdot \| (u^{(n+m)}-u^{(n)}) \|_{B^{\frac{d}{2}}_{2,1}} \cdot \| u^{(n+1)} \|_{B^{1+\frac{d}{2}}_{2,1}} \\
& \leq C (|\alpha(\tau)| + |\beta(\tau)| )(1+ \| u^{(n)} \|_{B^{1+\frac{d}{2}}_{2,1}} + \| u^{(n+m)} \|_{B^{1+\frac{d}{2}}_{2,1}})^2 \cdot \| (u^{(n+m)}-u^{(n)}) \|_{B^{\frac{d}{2}}_{2,1}}.
\end{aligned}
\end{equation*}
Similarly, we can also estimate the third term on the right-hand side of \eqref{SCNS-d9} as
\begin{equation*}
\begin{aligned}
\|f(u^{(n)},u^{(n+m)}) \|_{B^{\frac{d}{2}}_{2,\infty}} & = \|(\gamma(t) u^{(n+m)} + \xi(t) (u^{(n+m)})^2) \Delta S^{(n+m)} - (\gamma(t) u^{(n)} + \xi(t) (u^{(n)})^2) \Delta S^{(n)} \|_{B^{\frac{d}{2}}_{2,\infty}} \\
& = \|(\gamma(t) (u^{(n+m)} - u^{(n)}) \Delta S^{(n+m)} + \gamma(t) u^{(n)} \Delta S^{(n+m)} - \gamma(t) u^{(n)} \Delta S^{(n)} \\
& \quad + \xi(t) (u^{(n+m)}+ u^{(n)})(u^{(n+m)}- u^{(n)}) \Delta S^{(n+m)} + \xi(t) (u^{(n)})^2 \Delta S^{(n+m)} \\
& \quad - \xi(t) (u^{(n)})^2) \Delta S^{(n)} \|_{B^{\frac{d}{2}}_{2,\infty}} \\
& = \| \gamma(t) (u^{(n+m)} - u^{(n)}) u^{(n+m)} + \gamma(t) (u^{(n+m)} - u^{(n)}) u^{(n)}\\
& \quad + \xi(t) (u^{(n+m)}- u^{(n)})((u^{(n+m)})^2+ u^{(n)}u^{(n+m)}) + \xi(t) (u^{(n+m)}- u^{(n)})(u^{(n)})^2 \|_{B^{\frac{d}{2}}_{2,\infty}} \\
& \leq C(|\gamma(t)| + |\xi(t)|) \cdot \|u^{(n+m)}- u^{(n)}\|_{B^{\frac{d}{2}}_{2,1}} (1+ \| u^{(n+m)} \|_{B^{1+\frac{d}{2}}_{2,1}} + \| u^{(n)} \|_{B^{1+\frac{d}{2}}_{2,1}})^2.
\end{aligned}
\end{equation*}
Plugging above several estimates into \eqref{SCNS-d9}, using the Young's inequality and the uniform bound for the sequence $(u^{(n)})_{n \geq 1}$, we get that
\begin{equation}\label{SCNS-d10}
\begin{aligned}
\| (u^{(n+m+1)} & - u^{(n+1)})(t) \|_{B^{\frac{d}{2}}_{2,\infty} } \leq CH_0 2^{-n} + C(1+4Ch(H_0))^2 \\
& \quad \times \int_0^t (|\alpha(\tau)| + |\beta(\tau)| + |\gamma(\tau)| + |\xi(\tau)|) \| (u^{(n+m)} - u^{(n)})(\tau) \|_{B^{\frac{d}{2}}_{2,1} } d\tau.
\end{aligned}
\end{equation}

We apply Lemma \ref{interpolation}(2), the Logarithmic type interpolation inequality that characterizes the difference between the Besov spaces $B^{s}_{2,1}$ and $B^{s}_{2,\infty}$ with the parameters $p=2$, $s=\frac{d}{2}$ and $\epsilon = 1$, to the term $u^{(n+m)} - u^{(n)}$. This yields the following estimate:
\begin{equation}
\begin{aligned}
\| u^{(n+m)} - u^{(n)} \|_{B^{\frac{d}{2}}_{2,1}}
\leq C \| u^{(n+m)} - u^{(n)} \|_{B^{\frac{d}{2}}_{2,\infty}}
\left(1 + \log \frac{\| u^{(n+m)} - u^{(n)} \|_{B^{1+\frac{d}{2}}_{2,\infty}}}{\| u^{(n+m)} - u^{(n)} \|_{B^{\frac{d}{2}}_{2,\infty}}} \right).
\end{aligned}
\end{equation}
For simplicity, we introduce the following notations:
$$
\mathscr{D}_{n,m}(t) := \| (u^{(n+m)} - u^{(n)}) (t)\|_{B^{\frac{d}{2}}_{2,\infty}},\quad \mathscr{D}_{n}(t) := \sup_{m \in \mathbb{N^+}}\mathscr{D}_{n,m}(t),
$$
$$
\mathscr{D}(t) := \mathop{\lim\sup}\limits_{n \to \infty}\mathscr{D}_{n}(t).
$$
It then follows from \eqref{SCNS-d10} that
\begin{equation}
\begin{aligned}
\mathscr{D}_{n+1,m}(t) & \leq CH_0 2^{-n} + C(1+4Ch(H_0))^2 \int_0^t (|\alpha(\tau)| + |\beta(\tau)| + |\gamma(\tau)| + |\xi(\tau)|) \\
& \quad \times \mathscr{D}_{n,m}(\tau) \left(1 + \log \frac{4Ch(H_0)}{\mathscr{D}_{n,m}(\tau)} \right) d\tau \\
& \leq CH_0 2^{-n} + C(1+4Ch(H_0))^2 \int_0^t (|\alpha(\tau)| + |\beta(\tau)| + |\gamma(\tau)| + |\xi(\tau)|) \\
& \quad \times \mathscr{D}_{n}(\tau) \left(1 + \log \frac{4Ch(H_0)}{\mathscr{D}_{n}(\tau)} \right) d\tau, \\
\end{aligned}
\end{equation}
where the second inequality used the facts that $\mathscr{D}_{n,m}(\tau) \leq \mathscr{D}_{n}(\tau)$, and $f(x) := x (1+ \log \frac{4Ch(H_0)}{x})$ is continuous and increasing for all $x \in (0, 4Ch(H_0)]$. As the right-hand side in the last inequality is independent of $m$, after taking the supremum on both sides with respect to $m$ to the left-hand side, we get
\begin{equation}\label{SCNS-d11}
\begin{aligned}
\mathscr{D}_{n+1}(t) & \leq CH_0 2^{-n} + C(1+4Ch(H_0))^2 \int_0^t (|\alpha(\tau)| + |\beta(\tau)| + |\gamma(\tau)| + |\xi(\tau)|) \\
& \quad \times \mathscr{D}_{n}(\tau) \left(1 + \log \frac{4Ch(H_0)}{\mathscr{D}_{n}(\tau)} \right) d\tau. \\
\end{aligned}
\end{equation}
In view of the definition of $\mathscr{D}(t)$, for any given $\epsilon > 0$,there exists an integer $N=N(\epsilon)>0$ such that
\begin{equation*}
\begin{aligned}
\mathscr{D}_{n}(t) \leq \mathscr{D}(t) + \epsilon, \quad \text{for any } n > N.
\end{aligned}
\end{equation*}
Therefore, the estimates \eqref{SCNS-d11} implies that
\begin{equation}\label{SCNS-d12}
\begin{aligned}
\mathscr{D}_{n+1}(t) & \leq CH_0 2^{-n} + C(1+4Ch(H_0))^2 \int_0^t (|\alpha(\tau)| + |\beta(\tau)| + |\gamma(\tau)| + |\xi(\tau)|) \\
& \quad \times (\mathscr{D}(\tau) + \epsilon) \left(1 + \log \frac{4Ch(H_0)}{\mathscr{D}(\tau) + \epsilon} \right) d\tau. \\
\end{aligned}
\end{equation}
Using the fact of
$$
  1+\log(x) = \log(ex) \leq (e+1)\log(e+x), \quad \forall x \geq 1,
$$
and taking the limit as $n \to +\infty$ in \eqref{SCNS-d12}, one can derive that
\begin{equation*}
\begin{aligned}
\mathscr D(t) & \leq C(1+4Ch(H_0))^2 \int_0^t (|\alpha(\tau)| + |\beta(\tau)| + |\gamma(\tau)| + |\xi(\tau)|) \\
& \quad \times (\mathscr{D}(\tau) + \epsilon) \left(1 + \log \frac{4Ch(H_0)}{\mathscr{D}(\tau) + \epsilon} \right) d\tau \\
& \leq C(e+1)(1+4Ch(H_0))^2 \int_0^t (|\alpha(\tau)| + |\beta(\tau)| + |\gamma(\tau)| + |\xi(\tau)|) \\
& \quad \times (\mathscr{D}(\tau) + \epsilon) \log \left(e+\frac{4Ch(H_0)}{\mathscr{D}(\tau) + \epsilon}\right) d\tau.
\end{aligned}
\end{equation*}
Let $\epsilon \to 0$, we get for any $t \in [0,\infty)$ that
\begin{equation}\label{SCNS-d13}
\begin{aligned}
\mathscr D(t) & \leq C(1+4Ch(H_0))^2 \int_0^{\infty} (|\alpha(\tau)| + |\beta(\tau)| + |\gamma(\tau)| + |\xi(\tau)|) \mu(\mathscr D(\tau)) d\tau, \\
\end{aligned}
\end{equation}
where
\begin{equation*}
\mu(x) =
\begin{cases}
x \log \left(e+\frac{4Ch(H_0)}{x}\right), & x \in (0,4Ch(H_0)], \\
0, & x=0. \\
\end{cases}
\end{equation*}
Notice that $\mu(x)$ is an increasing continuous function for $x \in (0,4Ch(H_0)]$, and hence a modulus of continuity. Moreover, by the substitution of variable $y = 4Ch(H_0)/x$, we have
\begin{equation*}
\begin{aligned}
\int_{0}^{4Ch(H_0)} \frac{dx}{\mu(x)} & = \int_{1}^{+\infty} \frac{dy}{y \log(e+y)} \\
& \geq \int_{1}^{+\infty} \frac{dy}{(e+y) \log(e+y)} \\
& =\log \log(e+y)|^{+\infty}_1 = +\infty,
\end{aligned}
\end{equation*}
which implies that the function $\mu(x)$ is actually an Osgood modulus of continuity on $[0,4Ch(H_0)]$.

Since $\alpha$, $\beta$, $\gamma$ and $\xi$ are all integrable functions from $[0,\infty)$ into $\mathbb{R}^+$, one can now apply the well-known Osgood lemma to \eqref{SCNS-d13} to obtain
\begin{equation*}
\begin{aligned}
\mathscr{D}(t) \equiv 0, \quad \text{for any } t \in [0,\infty).
\end{aligned}
\end{equation*}
From this property we deduce that
$$
0 \leq \mathop{\lim\inf}\limits_{n \to \infty}\mathscr{D}_{n}(t) \leq \mathop{\lim\sup}\limits_{n \to \infty}\mathscr{D}_{n}(t) = 0,
$$
which indicates that $\lim\limits_{n \to \infty} \mathscr{D}_{n}(t) = 0$, and therefore $(u^{(n)})_{n \geq 1}$ is a Cauchy sequence in the Banach space $C([0,\infty);B^{\frac{d}{2}}_{2,\infty})$.
To improve the convergence to a more regular space $C([0,\infty);B^{\frac{d}{2}}_{2,1})$, we shall apply an interpolation argument, which is based on the following lemma.
\begin{lemma}[Real interpolation inequality \cite{Chemin2004}]\label{lem:RI}
If $s_1$ and $s_2$ are real numbers such that $s_1 < s_2$, $\theta \in (0,1)$, and $(p,r)$ is in $[1,\infty]^2$, then we have
$$
\| f \|_{B^{\theta s_1 + (1-\theta) s_2}_{p,1}} \leq \frac{C}{s_1 - s_2} \left(\frac{1}{\theta} + \frac{1}{1-\theta}\right) \| f \|^{\theta}_{B^{s_1}_{p,\infty}} \| f \|^{1-\theta}_{B^{s_2}_{p,\infty}},
$$
for some positive constant C.
\end{lemma}

Indeed, for any $0 < \epsilon < 1$, there exists an index $\theta \in (0,1)$ such that
$$
\frac{d}{2} + \epsilon = \frac{d}{2} \theta + (1+\frac{d}{2})(1-\theta).
$$
Thanks to the Sobolev embedding $B^{\frac{d}{2} + \epsilon}_{2,1} \hookrightarrow B^{\frac{d}{2}}_{2,1}$ for any $\epsilon > 0$, and the Lemma \ref{lem:RI}, we deduce that
\begin{equation*}
\begin{aligned}
\| (u^{(n+m+1)} - u^{(n+1)})(t) \|_{B^{\frac{d}{2}}_{2,1} } & \leq C(\theta) \| (u^{(n+m+1)} - u^{(n+1)})(t) \|^{\theta}_{B^{\frac{d}{2}}_{2,\infty} } \| (u^{(n+m+1)} - u^{(n+1)})(t) \|^{1-\theta}_{B^{1+\frac{d}{2}}_{2,\infty} }\\
& \leq C(\theta)(4Ch(H_0))^{1-\theta} \| (u^{(n+m+1)} - u^{(n+1)})(t) \|^{\theta}_{B^{\frac{d}{2}}_{2,\infty} } \\
& \to 0,\quad \text{as } n,m \to \infty,
\end{aligned}
\end{equation*}
for some $\theta \in (0,1)$. Here the last limitation used the fact that $(u^{(n)})_{n \geq 1}$ is a Cauchy sequence in $C([0,\infty);B^{\frac{d}{2}}_{2,\infty}(\mathbb{R}^d))$, and hence one can find a function $u$
such that
\begin{equation}\label{SCNS-d14}
\begin{aligned}
u^n \to u \quad \text{strongly in} \quad C([0,\infty);B^{\frac{d}{2}}_{2,1}(\mathbb{R}^d)), \quad \text{as } n \to \infty.
\end{aligned}
\end{equation}

\subsection{Existence}
We verify that $u$ in \eqref{SCNS-d14} is actually the strong solution to the system \eqref{SCNS-d1}. As $B^{s}_{p,r} \hookrightarrow \mathscr{S}'$ for any $s \in \mathbb{R}$, we conclude from \eqref{SCNS-d14} that $u^{(n)} \overset{\mathscr{S}'}{\to} u$ as $n \to \infty$. Moreover, by using the fact that $(u^{(n)})_{n \geq 1}$ is uniformly bounded in $C([0,\infty);B^{1+\frac{d}{2}}_{2,1})$, it follows from Lemma \ref{lem:FT} that $u \in L^{\infty}([0,\infty);B^{1+\frac{d}{2}}_{2,1})$, it then follows from the fact of $r=1$ that actually $u \in C([0,\infty);B^{1+\frac{d}{2}}_{2,1})$ (cf. Lemma \ref{euth}). Furthermore, by virtue of the algebra property of $B^{\frac{d}{2}}_{2,1}$ and the identity $S = (I - \Delta)^{-1} u$, it is not difficult to verify that
\begin{equation*}
\begin{aligned}
(\alpha(t) + \beta(t) u) \nabla S \cdot \nabla u , \quad (\gamma(t) u + \xi(t) u^2) \Delta S \in C([0,\infty);B^{\frac{d}{2}}_{2,1}(\mathbb{R}^d)),
\end{aligned}
\end{equation*}
which together with the \eqref{SCNS-t1} itself imply that $\partial_t u$ belongs to $C([0,\infty);B^{\frac{d}{2}}_{2,1}(\mathbb{R}^d))$. Thanks to the strong convergence result \eqref{SCNS-d14}, it is then easy to pass to the limit in \eqref{SCNS-d1} and to demonstrate that $u \in E^{1+\frac{d}{2}}_{2,1}(\infty)$ is indeed a strong solution of the system \eqref{SCNS-t1}.

\subsection{Uniqueness}
In this section, we aim at establishing the uniqueness of the solution $u$ to the Cauchy problem \eqref{SCNS-t1}, which is a direct consequence of the following lemma.
\begin{lemma}\label{lem:d1}
Let $d \in \mathbb{N}^+$. Suppose that $u,v \in L^{\infty}([0,\infty);B^{1+\frac{d}{2}}_{2,\infty}(\mathbb{R}^d) \cap Lip(\mathbb{R}^d)) \cap C([0,\infty);B^{\frac{d}{2}}_{2,\infty}(\mathbb{R}^d))$ are two solutions to the system \eqref{SCNS-t1} with respect to the initial datum $u_0,v_0 \in B^{1+\frac{d}{2}}_{2,\infty}(\mathbb{R}^d) \cap Lip(\mathbb{R}^d)$, respectively. Denote $\omega := u-v$ and $\omega_0 := u_0-v_0$. If we denote
$$
\mathscr{A}(t) = 1 + \| u(t) \|^2_{B^{1+\frac{d}{2}}_{2,1}} + \| v(t) \|^2_{B^{1+\frac{d}{2}}_{2,1}},
$$
then we have
\begin{equation}\label{SCNS-e7}
\begin{aligned}
\frac{\| \omega(t) \|_{B^{\frac{d}{2}}_{2,\infty}}}{4eCh(2+\mathscr{A}_0)} \leq \exp \left\{16C^2h^2(2+\mathscr{A}_0) \int_{0}^{t} (|\alpha(t')| + |\beta(t')|)dt' \right\} \left(\frac{\| \omega_0 \|_{B^{\frac{d}{2}}_{2,\infty}}}{4eCh(2+\mathscr{A}_0)}\right)^{\sigma(t)},
\end{aligned}
\end{equation}
for all $t \in [0,\infty)$, where
$$
\sigma(t) = \exp \left\{-16C^2 \log(e+1) h^2(2+\mathscr{A}_0) \int_{0}^{t} (|\alpha(t')| + |\beta(t')| + |\gamma(t')| + |\xi(t')|)dt' \right\},
$$
and the function $h(x)$ is a modulus of continuity defined in \eqref{SCNS-t20}.
\end{lemma}

\begin{proof}[\textbf{\emph{Proof of Lemma \ref{lem:d1}.}}]
Apparently,the function $\omega$ solves the following hyperbolic equation:
\begin{equation}\label{SCNS-e1}
\left\{
\begin{aligned}
& \partial_t \omega  + (\alpha(t) + \beta(t) u) \nabla S \cdot \nabla \omega = -[(\alpha(t) + \beta(t) u)\nabla S - (\alpha(t) + \beta(t) v)\nabla Q] \cdot \nabla v \\
& \quad -[(\gamma(t) u + \xi(t) u^2) \Delta S - (\gamma(t) v + \xi(t) v^2) \Delta Q] := F(t,x), \\
& \omega(0,x) = (u_0 - v_0)(x),
\end{aligned}
\right.
\end{equation}
where $S = (I - \Delta)^{-1} u$, $Q = (I - \Delta)^{-1} v$.
By applying the prior estimate (cf. Lemma \ref{prior estimates}) for linear transport equation in Besov spaces to Eq.\eqref{SCNS-e1} with $s = \frac{d}{2}$, $p = 2$ and $r = \infty$, we get
\begin{equation}\label{SCNS-e2}
\begin{aligned}
\| \omega(t) \|_{B^{\frac{d}{2}}_{2,\infty}} \leq e^{CV(t)}(\| \omega_0 \|_{B^{\frac{d}{2}}_{2,\infty}} + \int_{0}^{t} e^{-CV(\tau)} \| F(\tau,\cdot) \|_{B^{\frac{d}{2}}_{2,\infty}}) d\tau,
\end{aligned}
\end{equation}
with $V(t) = \int_{0}^{t} \| \nabla [(\alpha(t) + \beta(t) u) \nabla S]\|_{B^{\frac{d}{2}}_{2,\infty} \cap L^{\infty}} d\tau$. Since the spaces $B^{\frac{d}{2}}_{2,\infty} \cap L^{\infty}$ are Banach algebra, we get from the embedding $B^{1+\frac{d}{2}}_{2,1} (\mathbb{R}^d) \subset Lip(\mathbb{R}^d)$ and $B^{1+\frac{d}{2}}_{2,1} (\mathbb{R}^d) \subset B^{1+\frac{d}{2}}_{2,\infty} (\mathbb{R}^d)$ that
\begin{equation*}
\begin{aligned}
\| \nabla [(\alpha(t) + \beta(t) u) \nabla S] \|_{B^{\frac{d}{2}}_{2,\infty} \cap L^{\infty}}
&\leq C\| (\alpha(t) + \beta(t) u) \nabla S \|_{B^{1+\frac{d}{2}}_{2,\infty} \cap Lip} \\
&\leq C \big( |\alpha(t)| \cdot \| u \|_{B^{1+\frac{d}{2}}_{2,1}} + |\beta(t)| \cdot \| u \|^2_{B^{1+\frac{d}{2}}_{2,1}} \big) \\
& \leq C (|\alpha(t)| + |\beta(t)|) (\| u \|_{B^{1+\frac{d}{2}}_{2,1}} + \| u \|^2_{B^{1+\frac{d}{2}}_{2,1}}).
\end{aligned}
\end{equation*}
For the first nonlinear term involved in $F(t,x)$, by using the fact that $B^{\frac{d}{2}}_{2,\infty} (\mathbb{R}^d) \cap L^{\infty} (\mathbb{R}^d)$ is a Banach algebra and $(I - \Delta)^{-1}$ is a $S^{-2}$-multiplier, we have
\begin{equation*}
\begin{aligned}
\| & [(\alpha(t) + \beta(t) u)\nabla S - (\alpha(t) + \beta(t) v)\nabla Q] \cdot \nabla v \|_{B^{\frac{d}{2}}_{2,\infty}} \\
&\leq C \left(\| [(\alpha(t) + \beta(t) u)\nabla (I-\Delta)^{-1} \omega \|_{B^{\frac{d}{2}}_{2,\infty} \cap L^{\infty}} + \| \beta(t) \omega (I-\Delta)^{-1} v] \|_{B^{\frac{d}{2}}_{2,\infty} \cap L^{\infty}}\right) \cdot \|v\|_{B^{1+\frac{d}{2}}_{2,\infty} \cap Lip}\\
& \leq C \left( |\alpha(t)| + |\beta(t)| \cdot \|u\|_{B^{\frac{d}{2}}_{2,1}} + |\beta(t)| \cdot \|v\|_{B^{\frac{d}{2}}_{2,1}} \right) \|v\|_{B^{1+\frac{d}{2}}_{2,1}} \|\omega\|_{B^{\frac{d}{2}}_{2,1}} \\
& \leq C (|\alpha(t)| + |\beta(t)|)(1+ \|u\|^2_{B^{1+\frac{d}{2}}_{2,1}} + \|v\|^2_{B^{1+\frac{d}{2}}_{2,1}})\|\omega\|_{B^{\frac{d}{2}}_{2,1}},
\end{aligned}
\end{equation*}
where we have used the fact of $S - Q = (I - \Delta)^{-1} \omega$. For the second term involved in $F(t,x)$, we use the fact of $(I - \Delta)^{-1} \Delta$ is a $S^0-$multiplier and obtain
\begin{equation*}
\begin{aligned}
&\left\| (\gamma(t) u + \xi(t) u^2) \Delta S - (\gamma(t) v + \xi(t) v^2) \Delta Q \right\|_{B^{\frac{d}{2}}_{2,\infty}} \\
%& = \left\| (\gamma(t) u + \xi(t) u^2) \Delta S - (\gamma(t) v + \xi(t) v^2) \Delta S + (\gamma(t) v + \xi(t) v^2) \Delta S - (\gamma(t) v + \xi(t) v^2) \Delta Q \right\|_{B^{\frac{d}{2}}_{2,\infty}} \\
&\leq C \left( \|\omega\|_{B^{\frac{d}{2}}_{2,\infty} \cap L^{\infty}} (|\gamma(t)| + |\xi(t)| \cdot \|u\|_{B^{\frac{d}{2}}_{2,\infty}\cap L^{\infty}} + |\xi(t)| \cdot \|v\|_{B^{\frac{d}{2}}_{2,\infty}\cap L^{\infty}}) \|u\|_{B^{\frac{d}{2}}_{2,\infty}\cap L^{\infty}} \right. \\
&\quad + \left. (|\gamma(t)| \cdot \|v\|_{B^{\frac{d}{2}}_{2,\infty} \cap L^{\infty}} + |\xi(t)| \cdot \|v\|^2_{B^{\frac{d}{2}}_{2,\infty} \cap L^{\infty}}) \cdot \|(I- \Delta)^{-1} \Delta \omega\|_{B^{\frac{d}{2}}_{2,\infty} \cap L^{\infty}} \right) \\
& \leq C \left(|\gamma(t)| + |\xi(t)| \cdot \|u\|^2_{B^{\frac{d}{2}}_{2,\infty}\cap L^{\infty}} + |\xi(t)| \cdot \|uv\|_{B^{\frac{d}{2}}_{2,\infty}\cap L^{\infty}} \right. \\
& \quad \left. + |\gamma(t)| \cdot \|v\|_{B^{\frac{d}{2}}_{2,\infty} \cap L^{\infty}} + |\xi(t)| \cdot \|v\|^2_{B^{\frac{d}{2}}_{2,\infty} \cap L^{\infty}}\right) \cdot \|\omega\|_{B^{\frac{d}{2}}_{2,1}} \\
& \leq C (|\gamma(t)| + |\xi(t)|)(1+\|u\|^2_{B^{\frac{d}{2}}_{2,\infty} \cap L^{\infty}} + \|v\|^2_{B^{\frac{d}{2}}_{2,\infty} \cap L^{\infty}}) \cdot \|\omega\|_{B^{\frac{d}{2}}_{2,1}} \\
& \leq C (|\gamma(t)| + |\xi(t)|)(1+\|u\|^2_{B^{1+\frac{d}{2}}_{2,1}} + \|v\|^2_{B^{1+\frac{d}{2}}_{2,1}}) \cdot \|\omega\|_{B^{\frac{d}{2}}_{2,1}}.
\end{aligned}
\end{equation*}
Plugging the last three estimates into \eqref{SCNS-e2} leads to
\begin{equation*}
\begin{aligned}
\|\omega(t)\|_{B^{\frac{d}{2}}_{2,\infty}}
&\leq e^{C \int_{0}^{t} (|\alpha(\tau)| + |\beta(\tau)|) (\|u(\tau)\|_{B^{1+\frac{d}{2}}_{2,1}} + \|u(\tau)\|^2_{B^{1+\frac{d}{2}}_{2,1}}) d\tau} \\
&\quad \times \left( \|\omega_0\|_{B^{\frac{d}{2}}_{2,\infty}} + C \int_{0}^{t} e^{-C\int_{0}^{\tau} (|\alpha(r)| + |\beta(r)|) (\|u(r)\|_{B^{1+\frac{d}{2}}_{2,1}} + \|u(r)\|^2_{B^{1+\frac{d}{2}}_{2,1}}) dr} \right. \\
&\quad \left. \times (|\alpha(\tau)| + |\beta(\tau)| + |\gamma(\tau)| + |\xi(\tau)|)(1+\|u(\tau)\|^2_{B^{1+\frac{d}{2}}_{2,1}} + \|v(\tau)\|^2_{B^{1+\frac{d}{2}}_{2,1}}) \cdot \|\omega(\tau)\|_{B^{\frac{d}{2}}_{2,1}} d\tau \right).
\end{aligned}
\end{equation*}
Denote
\begin{equation*}
\begin{aligned}
\mathscr{L}(\tau) := e^{-C\int_{0}^{\tau} (|\alpha(r)| + |\beta(r)|) (\|u(r)\|_{B^{1+\frac{d}{2}}_{2,1}} + \|u(r)\|^2_{B^{1+\frac{d}{2}}_{2,1}}) dr} \|\omega(\tau)\|_{B^{\frac{d}{2}}_{2,1}} .
\end{aligned}
\end{equation*}
In order to estimate the $B^{\frac{d}{2}}_{2,\infty}$-norm of $\omega$, one should appropriately estimate the term $\mathscr{L}(\tau)$ on the right-hand side of last inequality. To this end, we shall apply the Logarithmic type interpolation inequality (cf. Lemma \ref{interpolation}(2)) to $\|\omega(\tau)\|_{B^{\frac{d}{2}}_{2,1}}$ and get that
\begin{equation*}
\begin{aligned}
\|\omega(\tau)\|_{B^{\frac{d}{2}}_{2,1}} \leq C\|\omega(\tau)\|_{B^{\frac{d}{2}}_{2,\infty}} \log \left(e+ \frac{\| \omega(\tau) \|_{B^{\frac{d}{2} + \epsilon}_{2,\infty}}}{\|\omega(\tau) \|_{B^{\frac{d}{2}}_{2,\infty}}} \right).
\end{aligned}
\end{equation*}
Since $\| \omega \|_{B^{1+\frac{d}{2}}_{2,\infty}} \leq 1+\|u\|^2_{B^{1+\frac{d}{2}}_{2,1}} + \|v\|^2_{B^{1+\frac{d}{2}}_{2,1}}$, by the Young's inequality and embedding $B^{1+\frac{d}{2}}_{2,1}(\mathbb{R}^d) \subset B^{1+\frac{d}{2}}_{2,\infty}(\mathbb{R}^d)$. The term $\mathscr{L}(\tau)$ can be estimated as
\begin{equation*}
\begin{aligned}
\mathscr{L}(\tau) & \leq e^{-C\int_{0}^{\tau} (|\alpha(r)| + |\beta(r)|) (\|u(r)\|_{B^{1+\frac{d}{2}}_{2,1}} + \|u(r)\|^2_{B^{1+\frac{d}{2}}_{2,1}}) dr} \|\omega(\tau)\|_{B^{\frac{d}{2}}_{2,\infty}} \\
& \quad \times \log \left(e+ \frac{\| \omega(\tau) \|_{B^{\frac{d}{2} + \epsilon}_{2,\infty}}}{\|\omega(\tau) \|_{B^{\frac{d}{2}}_{2,\infty}}} \right) \\
& \leq e^{-C\int_{0}^{\tau} (|\alpha(r)| + |\beta(r)|) (\|u(r)\|_{B^{1+\frac{d}{2}}_{2,1}} + \|u(r)\|^2_{B^{1+\frac{d}{2}}_{2,1}}) dr} \|\omega(\tau)\|_{B^{\frac{d}{2}}_{2,\infty}} \\
& \quad \times \log \left(e+ \frac{\| \omega(\tau) \|_{B^{\frac{d}{2} + \epsilon}_{2,\infty}}}{e^{-C\int_{0}^{\tau} (|\alpha(r)| + |\beta(r)|) (\|u(r)\|_{B^{1+\frac{d}{2}}_{2,1}} + \|u(r)\|^2_{B^{1+\frac{d}{2}}_{2,1}}) dr}\|\omega(\tau) \|_{B^{\frac{d}{2}}_{2,\infty}}} \right). \\
\end{aligned}
\end{equation*}
Therefore,
\begin{equation}\label{SCNS-e3}
\begin{aligned}
\|\omega(t)\|_{B^{\frac{d}{2}}_{2,\infty}}
&\leq e^{C \int_{0}^{t} (|\alpha(\tau)| + |\beta(\tau)|) (\|u(\tau)\|_{B^{1+\frac{d}{2}}_{2,1}} + \|u(\tau)\|^2_{B^{1+\frac{d}{2}}_{2,1}}) d\tau} \\
&\quad \times \left[ \|\omega_0\|_{B^{\frac{d}{2}}_{2,\infty}} + C \int_{0}^{t} e^{-C\int_{0}^{\tau} (|\alpha(r)| + |\beta(r)|) (\|u(r)\|_{B^{1+\frac{d}{2}}_{2,1}} + \|u(r)\|^2_{B^{1+\frac{d}{2}}_{2,1}}) dr} \right. \\
&\quad  \times (|\alpha(\tau)| + |\beta(\tau)| + |\gamma(\tau)| + |\xi(\tau)|) \\
& \quad \times (1+\|u(\tau)\|^2_{B^{1+\frac{d}{2}}_{2,1}} + \|v(\tau)\|^2_{B^{1+\frac{d}{2}}_{2,1}}) \cdot \|\omega(\tau)\|_{B^{\frac{d}{2}}_{2,\infty}}  \\
&\quad \left. \times \log \left(e+ \frac{\|\omega(\tau)\|_{B^{\frac{d}{2} + \epsilon}_{2,\infty}}}{e^{-C\int_{0}^{\tau} (|\alpha(r)| + |\beta(r)|) (\|u(r)\|_{B^{1+\frac{d}{2}}_{2,1}} + \|u(r)\|^2_{B^{1+\frac{d}{2}}_{2,1}}) dr}\|\omega(\tau)\|_{B^{\frac{d}{2}}_{2,\infty}}} \right) d\tau \right].
\end{aligned}
\end{equation}
Setting
\begin{equation*}
\begin{aligned}
S(t) = e^{-C\int_{0}^{t} (|\alpha(\tau)| + |\beta(\tau)|) (\|u(\tau)\|_{B^{1+\frac{d}{2}}_{2,1}} + \|u(\tau)\|^2_{B^{1+\frac{d}{2}}_{2,1}}) d\tau}\|\omega(t)\|_{B^{\frac{d}{2}}_{2,\infty}}.
\end{aligned}
\end{equation*}
Due to the uniform bound \eqref{c-uniform-bound} and the Fatou-type lemma, we have
\begin{equation*}
\begin{aligned}
\|u(t)\|_{B^{1+\frac{d}{2}}_{2,1}} + \|v(t)\|_{B^{1+\frac{d}{2}}_{2,1}} & \leq  2Ch(1 + \|u_0\|_{B^{1+\frac{d}{2}}_{2,1}}) + 2Ch(1 + \|v_0\|_{B^{1+\frac{d}{2}}_{2,1}})\\
&\leq 4Ch(1 + \|u_0\|_{B^{1+\frac{d}{2}}_{2,1}} + \|v_0\|_{B^{1+\frac{d}{2}}_{2,1}}) \\
&\leq 4Ch(1 + 1+ \|u_0\|^2_{B^{1+\frac{d}{2}}_{2,1}} + 1 + \|v_0\|^2_{B^{1+\frac{d}{2}}_{2,1}}) \\
& \leq 4Ch(2+\mathscr{A}_0),
\end{aligned}
\end{equation*}
where the increasing function $h(x)$ is defined in \eqref{SCNS-t20}.

From the definition of $\mathscr{A}(t)$ and $S(t)$ we have
$$
\mathscr{A}(t) = 1 + \| u \|^2_{B^{1+\frac{d}{2}}_{2,1}} + \| v \|^2_{B^{1+\frac{d}{2}}_{2,1}} \leq 16C^2h^2(2+\mathscr{A}_0),
$$
and
\begin{equation*}
\begin{aligned}
S(t) & \leq \sup_{t \in [0,\infty)} e^{-C\int_{0}^{t} (|\alpha(\tau)| + |\beta(\tau)|) (\|u(\tau)\|_{B^{1+\frac{d}{2}}_{2,1}} + \|u(\tau)\|^2_{B^{1+\frac{d}{2}}_{2,1}}) d\tau}\|\omega(t)\|_{B^{\frac{d}{2}}_{2,\infty}}\\
%& \leq \sup_{t \in [0,\infty)}
%\|\omega(t)\|_{B^{\frac{d}{2}}_{2,\infty}}\\
& \leq \sup_{t \in [0,\infty)}
\|\omega(t)\|_{B^{\frac{d}{2} + \epsilon}_{2,\infty}}\\
& \leq \sup_{t \in [0,\infty)}
(\|u(t)\|_{B^{1+\frac{d}{2}}_{2,\infty}} + \|v(t)\|_{B^{1+\frac{d}{2}}_{2,\infty}}) \leq 4Ch(2+\mathscr{A}_0). \\
\end{aligned}
\end{equation*}
It then follows from the inequality \eqref{SCNS-e3} that
\begin{equation}\label{SCNS-e4}
\begin{aligned}
S(t) \leq S(0) + 16C^2h^2(1+\mathscr{A}_0) \int_{0}^{t} (|\alpha(\tau)| + |\beta(\tau)| + |\gamma(\tau)| + |\xi(\tau)|) F(S(\tau))d\tau,
\end{aligned}
\end{equation}
where
\begin{equation*}
F(x) =
\begin{cases}
x \log \left(e+\frac{4Ch(2+\mathscr{A}_0)}{x}\right), & x \in (0,4Ch(2+\mathscr{A}_0)], \\
0, & x=0. \\
\end{cases}
\end{equation*}
It is easy to verify that $F(x)$ is an Osgood modulus of continuity on $[0,4Ch(2+\mathscr{A}_0)]$, so we get by applying the Osgood lemma (cf. Lemma \ref{Osgood}) to inequality \eqref{SCNS-e4} that
\begin{equation}\label{SCNS-e5}
\begin{aligned}
-\int_{S(t)}^{1} \frac{dr}{F(r)} & + \int_{S(0)}^{1} \frac{dr}{F(r)} = \int_{S(0)}^{S(t)} \frac{dr}{F(r)} \\
& \leq 16C^2h^2(2+\mathscr{A}_0) \int_{0}^{t} (|\alpha(\tau)| + |\beta(\tau)| + |\gamma(\tau)| + |\xi(\tau)|)d\tau.
\end{aligned}
\end{equation}
Notice that for any given constant $C>0$, we have the inequality
\begin{equation*}
\begin{aligned}
x \log \left(e + \frac{C}{x}\right) \leq \log(e + C)(1 - \log x), \quad \forall x \in (0,1].
\end{aligned}
\end{equation*}
The left hand side of \eqref{SCNS-e5} can be estimated as
\begin{equation*}
\begin{aligned}
\int_{S(0)}^{S(t)} \frac{dr}{F(r)} & = \int_{S(0)}^{S(t)} \frac{dr}{r \log \left(e + \frac{1}{\frac{r}{4Ch(2+\mathscr{A}_0)}}\right)} \\
& \geq \int_{S(0)}^{S(t)} \frac{d \left(\frac{r}{4Ch(2+\mathscr{A}_0)}\right)}{\log(e+1)\frac{r}{4Ch(2+\mathscr{A}_0)} \left(1 - \log \frac{r}{4Ch(2+\mathscr{A}_0)} \right) } \\
& = -\frac{1}{\log(e+1)} \log \left( \frac{1 - \log \frac{S(t)}{4Ch(2+\mathscr{A}_0)} }{1 - \log \frac{S(0)}{4Ch(2+\mathscr{A}_0)}}\right),
\end{aligned}
\end{equation*}
which combined with \eqref{SCNS-e5} yield that
\begin{equation}\label{SCNS-e6}
\begin{aligned}
\log \left( \frac{1 - \log \frac{S(t)}{4Ch(2+\mathscr{A}_0)}}{1 - \log \frac{S(0)}{4Ch(2+\mathscr{A}_0)}} \right) & \leq 16 \log(e+1) C^2 h^2 (2+\mathscr{A}_0) \\
& \quad \times \int_{0}^{t} (|\alpha(\tau)| + |\beta(\tau)| + |\gamma(\tau)| + |\xi(\tau)|) d\tau.
\end{aligned}
\end{equation}
Solving the inequality \eqref{SCNS-e6} leads to
\begin{equation*}
\begin{aligned}
\frac{S(t)}{4eCh(2+\mathscr{A}_0)} \leq \left(\frac{S(0)}{4eCh(2+\mathscr{A}_0)}\right)^{\exp \{-16C^2 \log(e+1) h^2 (2+\mathscr{A}_0) \int_{0}^{t} (|\alpha(\tau)| + |\beta(\tau)| + |\gamma(\tau)| + |\xi(\tau)|) d\tau\}}.
\end{aligned}
\end{equation*}
From the definition of $S(t)$, the previous inequality implies the desired inequality. This completes the proof of Lemma \ref{lem:d1}.

\end{proof}
\begin{proof}[\textbf{\emph{Proof of Theorem \ref{thm1} (Uniqueness)}.}]
Assume that $u$ and $v$ are solutions to the system \eqref{SCNS-t1} with respect to the some initial data $u_0$, $v_0$. In this case we have $\omega_0 = u_0 - v_0 \equiv 0$, and it follows from \eqref{SCNS-e7} that $\|u - v\|_{B^{\frac{d}{2}}_{2,\infty}} \equiv 0 $, which implies that $u \equiv v$, and this proves the uniqueness part.
\end{proof}

\subsection{Continuity}
Note that the proof of continuity boils down to proving that, for any given initial data $u_0 \in B^{1+\frac{d}{2}}_{2,1}(\mathbb{R}^d)$, one can find a $T > 0$ and a neighborhood $\mathscr U$ of $u_0$ in $B^{1+\frac{d}{2}}_{2,1}(\mathbb{R}^d)$ such that the map
\begin{equation*}
\Lambda :
\begin{cases}
\mathscr U \subset B^{1+\frac{d}{2}}_{2,1}(\mathbb{R}^d) \to C([0,T];B^{1+\frac{d}{2}}_{2,1}(\mathbb{R}^d)), \\
u_0 \to u \text{ solution to system } \eqref{SCNS-t1} \text{ with initial data } u_0,\\
\end{cases}
\end{equation*}
is continuous. The proof of this conclusion will be divided into two steps.

{\bfseries Step 1:} Continuity in $C([0,T];B^{\frac{d}{2}}_{2,1}(\mathbb{R}^d))$.
Let us fix a $r > 0$ and a $u_0 \in B^{1+\frac{d}{2}}_{2,1}(\mathbb{R}^d)$. Define the closed bounded ball $B_r(u_0) \subset B^{1+\frac{d}{2}}_{2,1}(\mathbb{R}^d)$ as
$$
B_r(u_0) \doteq \{ v_0 \in B^{1+\frac{d}{2}}_{2,1}(\mathbb{R}^d); \|u_0 - v_0\|_{B^{1+\frac{d}{2}}_{2,1}} \leq r \}.
$$
We make a {\bfseries Claim: }there is a $T > 0$ and a $M > 0$ such that for any $v_0 \in B_r(u_0)$, the solution $v = \Lambda(v_0)$ to the system \eqref{SCNS-t1} belongs to $C([0,T];B^{1+\frac{d}{2}}_{2,1}(\mathbb{R}^d))$ and satisfies
\begin{equation}\label{SCNS-e8}
\begin{aligned}
\sup_{t \in [0,T]} \|v(t)\|_{B^{1+\frac{d}{2}}_{2,1}} \leq M.
\end{aligned}
\end{equation}
Indeed, it follows from the uniform bound that, for any $v_0 \in \partial B_r(u_0)$, i.e., $\|v_0\|_{B^{1+\frac{d}{2}}_{2,1}} = \|u_0\|_{B^{1+\frac{d}{2}}_{2,1}} + r$, one can find a lifespan $T^*$ satisfying
$$
\int_0^{T^*} (|\alpha(\tau)| + |\beta(\tau)| + |\gamma(\tau)| + |\xi(\tau)|) d\tau  \leq \frac{\ln 2}{12C^3 h^2(1 + \|u_0\|_{B^{1+\frac{d}{2}}_{2,1}} + r)},
$$
where $h(x)$ is a modulus of continuity defined in \eqref{SCNS-t20}. Notice that the lifespan $T^*$ is a decreasing function with respect to the norm of the initial data $u_0$. Thereby for any solution $v$ to the system \eqref{SCNS-t1} associated with $v_0 \in B_r(u_0)$, one can restrict $v$ on the interval $[0,T^*] \subset [0,\overline{T}]$, and
\begin{equation}\label{SCNS-e9}
\begin{aligned}
\sup_{t \in [0,T^*]} \|v(t)\|_{B^{1+\frac{d}{2}}_{2,1}} & \leq \sup_{t \in [0,\overline{T}]} \|v(t)\|_{B^{1+\frac{d}{2}}_{2,1}} \\
& \leq 2Ch(1+\|v_0\|_{B^{1+\frac{d}{2}}_{2,1}}) \\
& \leq 2Ch(1+\|u_0\|_{B^{1+\frac{d}{2}}_{2,1}} + r) \doteq M. \\
\end{aligned}
\end{equation}
Then we prove the Claim by choosing $T = T^*$.

Combining the above conclusion with Lemma \ref{lem:d1}, we obtain
\begin{equation}\label{SCNS-e10}
\begin{aligned}
\|\Lambda(u_0) - \Lambda(v_0)\|_{B^{\frac{d}{2}}_{2,\infty}} & = \|(u - v)(t)\|_{B^{\frac{d}{2}}_{2,\infty}} \\
& \leq Ch(3 + 2M^2) e^{Ch^2(3+2M^2)} \times \left(\frac{\|u_0 - v_0\|_{B^{\frac{d}{2}}_{2,\infty}}}{Ch(3 + 2M^2)}\right)^{e^{-Ch^2(3+2M^2)}}.
\end{aligned}
\end{equation}
By using the real interpolation inequality in Lemma \ref{lem:RI}, we deduce from \eqref{SCNS-e9} - \eqref{SCNS-e10} and the embedding $B^{\frac{d}{2} + \theta}_{2,1}(\mathbb{R}^d) \subset B^{\frac{d}{2}}_{2,1}(\mathbb{R}^d)$ for any $\theta \in (0,1)$ that
\begin{equation*}
\begin{aligned}
\|\Lambda(u_0) - \Lambda(v_0)\|_{B^{\frac{d}{2}}_{2,1}} & \leq \|\Lambda(u_0) - \Lambda(v_0)\|_{B^{\frac{d}{2} + \theta}_{2,1}} \\
& \leq C \|\Lambda(u_0) - \Lambda(v_0)\|^{1-\theta}_{B^{\frac{d}{2}}_{2,\infty}} \|\Lambda(u_0) - \Lambda(v_0)\|^{\theta}_{B^{1+\frac{d}{2}}_{2,\infty}} \\
& \leq Ch^{1- \theta}(3+2M^2) e^{C(1-\theta) h^2(3+2M^2)} \\
& \quad \times \left(\frac{\|u_0 - v_0\|_{B^{\frac{d}{2}}_{2,\infty}}}{Ch(3+2M^2)}\right)^{e^{-C(1-\theta) h^2(3+2M^2)}} \times (\|u_0\|_{B^{1+\frac{d}{2}}_{2,1}} + \|v_0\|_{B^{1+\frac{d}{2}}_{2,1}})^{\theta} \\
& \leq CM^{\theta} h^{1-\theta}(3 + 2M^2) e^{C(1-\theta) h^2(3+2M^2)} \times \left(\frac{\|u_0 - v_0\|_{B^{\frac{d}{2}}_{2,\infty}}}{Ch(3+2M^2)}\right)^{e^{-C(1-\theta) h^2(3+2M^2)}},
\end{aligned}
\end{equation*}
which implies that the map $\Lambda(\cdot): B^{1+\frac{d}{2}}_{2,1} \to C([0,T];B^{\frac{d}{2}}_{2,1})$ is H\"older continuous.

{\bfseries Step 2:} Continuity in $C([0,T];B^{1+\frac{d}{2}}_{2,1}(\mathbb{R}^d))$.
Denote by $u^n$ the solution corresponding to the initial data $u_0^n \in B^{1+\frac{d}{2}}_{2,1}(\mathbb{R}^d)$, for all $n \geq 1$. Assume that the sequence $(u^n_0)_{n \geq 1}$ converges to a function $u_0^{\infty}$ strongly in $B^{1+\frac{d}{2}}_{2,1}(\mathbb{R}^d)$ as $n \to \infty$. According to the first step, one can find a $T>0$ and a $M>0$ such that for all $n \geq 1$, $u^n$ is defined on $[0,T]$ and
\begin{equation}\label{SCNS-e11}
\begin{aligned}
\sup_{n \in \mathbb{N^+} \cup {\{ \infty \}}} \|u^n\|_{L^{\infty}([0,T];B^{1+\frac{d}{2}}_{2,1})} \leq M.
\end{aligned}
\end{equation}
Observing that in order to prove the desired result, it is sufficient to prove that $\nabla u^n$ tends to $\nabla u^{\infty}$ in $C([0,T];B^{\frac{d}{2}}_{2,1}(\mathbb{R}^d))$. Differentiating the system \eqref{SCNS-d1} with respect to $x$, we find that for each $i=1,2,...,d$, the $i$-th component $\nabla_i u^n$ of $\nabla u^n$ satisfies
\begin{equation*}
\left\{
\begin{aligned}
& \partial_t \nabla_i u^n + (\alpha(t) + \beta(t) u^n) \nabla S^n \cdot \nabla \nabla_i u^n = \nabla_i f^n(t,x), \\
& \nabla_i u^n|_{t=0} = \nabla_i u_0^n,
\end{aligned}
\right.
\end{equation*}
with $f^n(t,x) := -(\alpha(t) + \beta(t) u^n) \nabla_j S^n \nabla_j u^n -(\gamma(t) u^n + \xi(t) (u^n)^2) \Delta S^n$. Following the decomposition method by Kato (cf. \cite{Kato1975}), we decompose $u^n$ into $u^n = y^n + z^n$ with
\begin{equation*}
\left\{
\begin{aligned}
& \partial_t \nabla_i y^n + (\alpha(t) + \beta(t) u^n) \nabla S^n \cdot \nabla \nabla_i y^n = \nabla_i f^n - \nabla_i f^{\infty}, \\
& \nabla_i y^n|_{t=0} = \nabla_i u_0^n - \nabla_i u_0^{\infty},
\end{aligned}
\right.
\end{equation*}
and
\begin{equation}\label{SCNS-e12}
\left\{
\begin{aligned}
& \partial_t \nabla_i z^n + (\alpha(t) + \beta(t) u^n) \nabla S^n \cdot \nabla \nabla_i z^n = \nabla_i f^{\infty}, \\
& \nabla_i z^n|_{t=0} = \nabla_i u_0^{\infty}.
\end{aligned}
\right.
\end{equation}
Thanks to the basic properties of Besov spaces, one can easily check that $(f^n)_{n \in \{ \mathbb{N^+},\infty\}}$ is uniformly bounded in $C([0,T];B^{\frac{d}{2}}_{2,1}(\mathbb{R}^d))$. Moreover, there holds
\begin{equation*}
\begin{aligned}
\nabla_i f^n - \nabla_i f^{\infty} & = \nabla_i [(\alpha(t) + \beta(t) u^n) \nabla_j S^n] \nabla_j (u^n - u^{\infty}) \\
& + \nabla_j u^{\infty} \nabla_i [(\alpha(t) + \beta(t) u^n) \nabla_j (S^n - S^{\infty}) + \beta(t) (u^n - u^{\infty}) \nabla_j S^{\infty}] \\
& + \nabla_i [(\gamma(t) u^n + \xi(t) (u^n)^2) \Delta(S^n -S^{\infty})] \\
& + \nabla_i[\gamma(t) (u^n - u^{\infty}) \Delta S^{\infty} + \xi(t) (u^n + u^{\infty}) (u^n - u^{\infty}) \Delta S^{\infty}].
\end{aligned}
\end{equation*}
Then,
\begin{equation*}
\begin{aligned}
\| \nabla_i f^n - \nabla_i f^{\infty} \|_{B^{\frac{d}{2}}_{2,1}} & \leq (|\alpha(t)| + |\beta(t)| + |\gamma(t)| + |\xi(t)|) \\
& \times (\| u^n \|_{B^{1+\frac{d}{2}}_{2,1}} + \| u^n \|^2_{B^{1+\frac{d}{2}}_{2,1}} + \| u^{\infty} \|_{B^{1+\frac{d}{2}}_{2,1}} + \| u^{\infty} \|^2_{B^{1+\frac{d}{2}}_{2,1}}) \\
& \times (\| u^n -u^{\infty} \|_{B^{\frac{d}{2}}_{2,1}} + \| \nabla_j (u^n - u^{\infty}) \|_{B^{\frac{d}{2}}_{2,1}}).
\end{aligned}
\end{equation*}
Therefore, the product law in Besov spaces $B^{\frac{d}{2}}_{2,1}(\mathbb{R}^d)$ combined with the prior estimate in Lemma \ref{prior estimates} (in the case $r = 1$) lead to
\begin{equation*}
\begin{aligned}
\| \nabla_i y^n(t) \|_{B^{\frac{d}{2}}_{2,1}} & \leq e^{C \int_{0}^{t} \| \nabla[(\alpha(\tau) + \beta(\tau) u^n) \nabla S^n] \|_{B^{\frac{d}{2}}_{2,\infty}\cap L^{\infty}} d\tau } \times \| \partial_i u_0^n - \partial_i u_0^{\infty} \|_{B^{\frac{d}{2}}_{2,1}} \\
& \quad + C \int_{0}^{t} e^{C \int_{\tau}^{t} \| \nabla[(\alpha(r) + \beta(r) u^n) \nabla S^n] \|_{B^{\frac{d}{2}}_{2,\infty}\cap L^{\infty}} dr }  \\
& \quad \times (|\alpha(t)| + |\beta(t)| + |\gamma(t)| + |\xi(t)|) \\
& \quad \times (\| u^n \|_{B^{1+\frac{d}{2}}_{2,1}} + \| u^n \|^2_{B^{1+\frac{d}{2}}_{2,1}} + \| u^{\infty} \|_{B^{1+\frac{d}{2}}_{2,1}} + \| u^{\infty} \|^2_{B^{1+\frac{d}{2}}_{2,1}}) \\
& \quad \times (\| u^n -u^{\infty} \|_{B^{\frac{d}{2}}_{2,1}} + \| \nabla_j (u^n - u^{\infty}) \|_{B^{\frac{d}{2}}_{2,1}}) d\tau.
\end{aligned}
\end{equation*}
Since
\begin{equation*}
\begin{aligned}
\int_{0}^{t} \| \nabla[(\alpha(\tau) + \beta(\tau) u^n) \nabla S^n] \|_{B^{\frac{d}{2}}_{2,\infty}\cap L^{\infty}} d\tau & \leq \int_{0}^{t} \| (\alpha(\tau) + \beta(\tau) u^n) \nabla S^n \|_{B^{1+\frac{d}{2}}_{2,\infty}\cap Lip} d\tau \\
& \leq C \int_{0}^{t} (|\alpha(\tau)| + |\beta(\tau)|)(\| u^n \|_{B^{1+\frac{d}{2}}_{2,1}} + \| u^n \|^2_{B^{1+\frac{d}{2}}_{2,1}}) d\tau \\
& \leq C(M +M^2),
\end{aligned}
\end{equation*}
together with the uniform bound \eqref{SCNS-e11} yield that
\begin{equation}\label{e1}
\begin{aligned}
\|\nabla_i y^n(t)\|_{B^{\frac{d}{2}}_{2,1}}
&\leq C(1 + M + M^2) e^{C(M+M^2)} \Bigg(\|\partial_i u_0^n - \partial_i u_0^{\infty}\|_{B^{\frac{d}{2}}_{2,1}} \\
&\quad + \int_{0}^{t} \Big(|\alpha(\tau)| + |\beta(\tau)| + |\gamma(\tau)| + |\xi(\tau)|\Big) \\
&\quad \times \Big(\|(u^n - u^{\infty})(\tau)\|_{B^{\frac{d}{2}}_{2,1}} + \|\nabla_j (u^n - u^{\infty})(\tau)\|_{B^{\frac{d}{2}}_{2,1}}\Big) d\tau \Bigg).
\end{aligned}
\end{equation}
On the other hand, the field $a^n(t) := (\alpha(t) + \beta(t) u^n) \nabla S^n$ is uniformly bounded in $C([0,T];B^{1+\frac{d}{2}}_{2,1}(\mathbb{R}^d))$. Moreover, observing that
\begin{equation*}
\begin{aligned}
\|(a^n - a^{\infty})(t)\|_{B^{\frac{d}{2}}_{2,1}} & = \|(\alpha(t) + \beta(t) u^n) \nabla S^n - (\alpha(t) + \beta(t) u^{\infty}) \nabla S^{\infty}\|_{B^{\frac{d}{2}}_{2,1}}  \\
& \leq C (\|(\alpha(t) + \beta(t) u^n) \nabla (S^n - S^{\infty}) \|_{B^{\frac{d}{2}}_{2,1}} + |\beta(t)| \cdot \|(u^n - u^{\infty}) \nabla S^{\infty}\|_{B^{\frac{d}{2}}_{2,1}})  \\
& \leq C(|\alpha(t) + |\beta(t)|)( 1+ \| u^n \|_{B^{1+\frac{d}{2}}_{2,1}} + \| u^{\infty} \|_{B^{1+\frac{d}{2}}_{2,1}}) \| u^n - u^{\infty} \|_{B^{\frac{d}{2}}_{2,1}},
\end{aligned}
\end{equation*}
it then follows from the fact of $u^n \to u^{\infty}$ in $C([0,T];B^{\frac{d}{2}}_{2,1}(\mathbb{R}^d))$ (in view of the first step) that $a^n \to a^{\infty}$ in $L^1([0,T];B^{\frac{d}{2}}_{2,1}(\mathbb{R}^d))$ as $n \to \infty$.

To deal with the convergence of the system \eqref{SCNS-e12}, we recall the following useful lemma.
\begin{lemma}[\cite{Danchin2003}]\label{lem:da}
Let $d \in \mathbb{N}^+$. Let $(u^n)_{n \geq 1}$ be a sequence of functions belonging to $C([0,T];B^{1+\frac{d}{2}}_{2,1}(\mathbb{R}^d))$. Assume that $v^n$ solves the following equation
\begin{equation*}
\left\{
\begin{aligned}
& \partial_t u^n + a^n \cdot \nabla u^n = h,\\
& u^n(0,x) = u_0(x),
\end{aligned}
\right.
\end{equation*}
with $u_0 \in B^{\frac{d}{2}}_{2,1}(\mathbb{R}^d)$, $h \in L^1([0,T];B^{\frac{d}{2}}_{2,1}(\mathbb{R}^d))$ and $\sup_{n \in \mathbb{N^+}} \|a^n\|_{B^{1+\frac{d}{2}}_{2,1}} \leq \gamma(t)$, for some $\gamma \in L^1(0,T)$. If in addition $a^n$ tends to $a^{\infty}$ in $L^1([0,T];B^{\frac{d}{2}}_{2,1}(\mathbb{R}^d))$, then $u^n$ tends to $u^{\infty}$ in $C([0,T];B^{1+\frac{d}{2}}_{2,1}(\mathbb{R}^d))$.
\end{lemma}
By applying Lemma \ref{lem:da}, we infer that $\nabla_i z^n$ tends to $\nabla_i u^{\infty}$ in $C([0,T];B^{\frac{d}{2}}_{2,1}(\mathbb{R}^d))$ as $n \to \infty$. Therefore, for any $\epsilon >0$, we have for large enough $n \geq 1$ that
\[
\sup_{t \in [0,T]} \|(\nabla_i z^n - \nabla_i u^{\infty})(t)\|_{B^{\frac{d}{2}}_{2,1}} < \epsilon.
\]
Combining above inequality with estimates \eqref{SCNS-e11} and \eqref{e1} as well as the fact of $y^n = u^n - z^n$, we get for large enough $n \geq 1$ that
\begin{equation*}
\begin{aligned}
\|(\nabla_i u^n - \nabla_i u^{\infty})(t)\|_{B^{\frac{d}{2}}_{2,1}} & = \|(\nabla_i y^n + \nabla_i z^n - \nabla_i u^{\infty})(t)\|_{B^{\frac{d}{2}}_{2,1}} \\
& \leq C(\|(\nabla_i z^n - \nabla_i u^{\infty})(t)\|_{B^{\frac{d}{2}}_{2,1}} + \|\nabla_i y^n(t)\|_{B^{\frac{d}{2}}_{2,1}}) \\
& \leq \epsilon +C(1+M+M^2) e^{C(M+M^2)} \bigg(\| \partial_i u_0^n - \partial_i u_0^{\infty} \|_{B^{\frac{d}{2}}_{2,1}} \\
& \quad + \int_{0}^{t} (|\alpha(\tau)| + |\beta(\tau)| + |\gamma(\tau)| + |\xi(\tau)|) \\
& \quad \times (\| (u^n -u^{\infty}) (\tau)\|_{B^{\frac{d}{2}}_{2,1}} + \| \nabla_j (u^n - u^{\infty})(\tau) \|_{B^{\frac{d}{2}}_{2,1}}) d\tau \bigg) \\
& \leq \epsilon + \epsilon C(1+M+M^2)e^{C(M+M^2)}  + C(1+M+M^2)e^{C(M+M^2)} \\
& \quad \times \int_{0}^{t} (|\alpha(\tau)| + |\beta(\tau)| + |\gamma(\tau)| + |\xi(\tau)|) (\| \nabla_j (u^n - u^{\infty})(\tau) \|_{B^{\frac{d}{2}}_{2,1}}) d\tau,
\end{aligned}
\end{equation*}
where the last inequality used the assumption for $u_0^n$ and the fact that $u^n$ tends to $u^{\infty}$ in $C([0,T];B^{\frac{d}{2}}_{2,1}(\mathbb{R}^d))$ as $n \to  \infty$. By applying the Gronwall inequality, we get
\begin{equation*}
\begin{aligned}
\|(\nabla_i u^n - \nabla_i u^{\infty})(t)\|_{L^{\infty}([0,T];B^{\frac{d}{2}}_{2,1})}
\leq C(M)\left(\epsilon + \|\nabla_i u_0^n - \nabla_i u_0^{\infty}\|_{B^{\frac{d}{2}}_{2,1}}\right),
\end{aligned}
\end{equation*}
for $i = 1,2,...,d$, and some positive constant C(M) depending only on M, which implies the continuity of the data-to-solution map in $C([0,T];B^{1+\frac{d}{2}}_{2,1}(\mathbb{R}^d))$.

Combining the first step and second step, we completed the proof of Theorem \ref{thm1}.

\section{Blow-up phenomena}
%----------------------------------------------------
By means of the Littlewood-Paley decomposition theory, we shall prove that the solutions to \eqref{SCNS-t1} blows up at finite time in the Besov spaces $B^{1+\frac{d}{2}}_{2,1} (\mathbb{R}^d) $ and $B^s_{p,r}(\mathbb{R}^d) (s > 1+\frac{d}{p}, 1 \leq p < \infty, 1 \leq r \leq \infty)$, respectively. Note that in both cases, the existence and uniqueness of solutions are ensured by Theorem \ref{thm} and Theorem \ref{thm1} in this work.

\begin{proof}[\textbf{\emph{Proof of Theorem \ref{blow-up_1}.}}]
Suppose that $u$ is the solution to \eqref{SCNS-t1} with initial data $u_0 \in B^{1+\frac{d}{2}}_{2,1} (\mathbb{R}^d)$, and let $T^* > 0$ be the maximal existence time of the solution $u$.
Applying the Littlewood-Paley blocks $\Delta_q$ to both sides of the first equation in \eqref{SCNS-t1}, after rearranging the terms we infer that
\begin{equation*}
\begin{aligned}
\partial_t & \Delta_q u + (\alpha(t) + \beta(t) u) \nabla S \cdot \nabla \Delta_q u \\
& = (\alpha(t) + \beta(t) u) \nabla S \cdot \nabla \Delta_q u - \Delta_q [(\alpha(t) + \beta(t) u) \nabla S \cdot \nabla u] - \Delta_q [(\gamma(t)u + \xi(t) u^2) \Delta S].
\end{aligned}
\end{equation*}

Multiplying both sides by $2\Delta_q u$ and then integrating the resulted equality with respect to space variable, we get
\begin{equation}\label{b1}
\begin{aligned}
\frac{d}{dt} \|\Delta_q u\|^2_{L^2} & = - 2\int_{\mathbb{R}^d} \Delta_q u \cdot (\alpha(t) + \beta(t) u) \nabla S \cdot \nabla \Delta_q u dx \\
& \quad + 2 \int_{\mathbb{R}^d} \Delta_q u [(\alpha(t) + \beta(t) u) \nabla S, \Delta_q] \cdot \nabla u dx \\
& \quad - 2 \int_{\mathbb{R}^d} \Delta_q u \Delta_q [(\gamma(t)u + \xi(t) u^2) \Delta S] dx \\
& := T_1(t) + T_2(t) + T_3(t),
\end{aligned}
\end{equation}
where the commutator term in $T_2(t)$ is defined by
\[
[(\alpha(t) + \beta(t) u) \nabla S \cdot \nabla, \Delta_q] u = (\alpha(t) + \beta(t) u) \nabla S \cdot \nabla \Delta_q u - \Delta_q [(\alpha(t) + \beta(t) u) \nabla S \cdot \nabla u].
\]
By integrating by parts the term $T_1$ can be estimated by
\begin{equation*}
\begin{aligned}
T_1(t) & = - 2 \int_{\mathbb{R}^d} \Delta_q u \cdot (\alpha(t) + \beta(t) u) \nabla S \cdot \nabla \Delta_q u dx \\
& = - \int_{\mathbb{R}^d} (\alpha(t) + \beta(t) u) \nabla S \cdot \nabla (|\Delta_q u|^{2})dx \\
& = \int_{\mathbb{R}^d} |\Delta_q u|^{2} \div[(\alpha(t) + \beta(t) u) \nabla S]dx \\
& \leq \| \div[(\alpha(t) + \beta(t) u) \nabla S] \|_{L^{\infty}} \| \Delta_q u \|^2_{L^2}.
\end{aligned}
\end{equation*}
By using the H\"older inequality, we also have
\[
T_2(t) \leq C \| \Delta_q u \|_{L^2} \| [(\alpha(t) + \beta(t) u) \nabla S \cdot \nabla, \Delta_q]u \|_{L^2},
\]and
\[
T_3(t) \leq C \| \Delta_q u \|_{L^2} \| \Delta_q [(\gamma(t)u + \xi(t) u^2) \Delta S] \|_{L^2}.
\]
It then follows from \eqref{b1} that
\begin{equation*}
\begin{aligned}
\frac{d}{dt} \|\Delta_q u\|_{L^2} & \leq C \| [(\alpha(t) + \beta(t) u) \nabla S \cdot \nabla, \Delta_q]u \|_{L^2} + C \| \Delta_q [(\gamma(t)u + \xi(t) u^2) \Delta S] \|_{L^2} \\
& \quad + \| \div[(\alpha(t) + \beta(t) u) \nabla S] \|_{L^{\infty}} \| \Delta_q u \|_{L^2}.
\end{aligned}
\end{equation*}

Multiplying the both sides of last inequality by $2^{q(1+\frac{d}{2})}$ and then taking the $l^1$-norm with respect to $q \geq -1$, we get from the commutator estimate (cf. Lemma \ref{ce}) and Moser-type estimates (cf. Lemma \ref{me}) in Besov spaces and the basic properties for $B^{\frac{d}{2}}_{2,1}(\mathbb{R}^d)$ that
\begin{equation}
\begin{aligned}
\frac{d}{dt} \| u(t) \|_{B^{1+\frac{d}{2}}_{2,1}}
&\leq C \left\| \left\{2^{q(1+\frac{d}{2})} \left\| \left[(\alpha(t) + \beta(t) u) \nabla S \cdot \nabla, \Delta_q\right] u \right\|_{L^2} \right\} \right\|_{l^1}
+ C \left\| (\gamma(t)u + \xi(t) u^2) \Delta S \right\|_{B^{1+\frac{d}{2}}_{2,1}} \\
&\quad + \left\| \operatorname{div}\left[(\alpha(t) + \beta(t) u) \nabla S\right] \right\|_{L^{\infty}} \| u \|_{B^{1+\frac{d}{2}}_{2,1}} \\
&\leq C (|\alpha(t)| + |\beta(t)|) \left( \left\| \nabla ((1 + u) \nabla S) \right\|_{L^{\infty}} \| u \|_{B^{1+\frac{d}{2}}_{2,1}}
+ C \| \nabla u\|_{L^{\infty}} (1 + \| u\|_{L^{\infty}}) \| \nabla S \|_{B^{\frac{d}{2}}_{2,1}} \right. \\
&\quad + \| \nabla u \|_{L^{\infty}} \| \Delta S \|_{B^{\frac{d}{2}}_{2,1}}
+ \| \nabla u \|_{L^{\infty}} \| \nabla S \|_{B^{\frac{d}{2}}_{2,1}}
+ \| \nabla S \|_{L^{\infty}} \| \nabla u \|_{B^{\frac{d}{2}}_{2,1}} \\
&\quad + \left. \| \Delta S \|_{L^{\infty}} \| u \|_{B^{1+\frac{d}{2}}_{2,1}} \right)
+ C(|\gamma(t)| + |\xi(t)|) \left( (1+\|u\|_{L^{\infty}})\| u \|_{B^{1+\frac{d}{2}}_{2,1}} \| \Delta S \|_{L^{\infty}} \right. \\
&\quad + \left. \| u + u^2 \|_{L^{\infty}} \| \Delta S \|_{B^{1+\frac{d}{2}}_{2,1}} \right)
+ C (|\alpha(t)| + |\beta(t)|) \left\| \operatorname{div}\left[(1 + u) \nabla S\right] \right\|_{L^{\infty}} \| u \|_{B^{1+\frac{d}{2}}_{2,1}} \\
& \leq C(|\alpha(t)| + |\beta(t)| + |\gamma(t)| + |\xi(t)|) \\
& \quad \times (1 + \| \nabla S \|^2_{L^{\infty}} + \| u \|^2_{L^{\infty}} + \| \Delta S \|^2_{L^{\infty}} + \| \nabla u \|^2_{L^{\infty}}) \| u \|_{B^{1+\frac{d}{2}}_{2,1}}, \notag
\end{aligned}
\end{equation}
which is to say,
\begin{equation}\label{b2}
\begin{aligned}
\| u(t) \|_{B^{1+\frac{d}{2}}_{2,1}} & \leq \| u_0 \|_{B^{1+\frac{d}{2}}_{2,1}} + C \int_{0}^{t}(|\alpha(r)| + |\beta(r)| + |\gamma(r)| + |\xi(r)|) \\
& \quad \times (1 + \| \nabla S \|^2_{L^{\infty}} + \| u \|^2_{L^{\infty}} + \| \Delta S \|^2_{L^{\infty}} + \| \nabla u \|^2_{L^{\infty}}) \| u \|_{B^{1+\frac{d}{2}}_{2,1}} dr.
\end{aligned}
\end{equation}
Thanks to the $L^1-$integrability for the Bessel kernel $G(x)$ associated to the Bessel potential $(1-\Delta)^{-1}$ (cf. Section 5 in \cite{Stein2016}), one can estimate as
\[
\|S\|_{L^{\infty}} = \|G \star u\|_{L^{\infty}} \leq \|G\|_{L^1} \|u\|_{L^{\infty}} \leq C \|u\|_{L^{\infty}},
\]
and
\[
\|\nabla S\|_{L^{\infty}} = \|G \star \nabla u\|_{L^{\infty}} = \|\nabla G \star u\|_{L^{\infty}}= \|\nabla G\|_{L^1} \|u\|_{L^{\infty}} \leq C \|u\|_{L^{\infty}}.
\]
It follows from \eqref{b2} that
\begin{equation}\label{b3}
\begin{aligned}
\| u(t) \|_{B^{1+\frac{d}{2}}_{2,1}} & \leq \| u_0 \|_{B^{1+\frac{d}{2}}_{2,1}} + C \int_{0}^{t}(|\alpha(r)| + |\beta(r)| + |\gamma(r)| + |\xi(r)|) \\
& \quad \times (1 + \| u \|^2_{L^{\infty}} + \| \nabla u \|^2_{L^{\infty}}) \| u \|_{B^{1+\frac{d}{2}}_{2,1}} dr.
\end{aligned}
\end{equation}

An application of the Littlewood-Paley decomposition leads to $\|u\|_{L^{\infty}} = \|\sum_{j \in \mathbb{Z}} \dot{\Delta}_j u\|_{L^{\infty}} \leq \sum_{j \in \mathbb{Z}} \|\ \dot{\Delta}_j u\|_{L^{\infty}}$, which implies that $\dot{B}^0_{\infty,1} ({\mathbb{R}^d}) \hookrightarrow L^{\infty} ({\mathbb{R}^d})$. Therefore, we deduce from \eqref{b3} that
\begin{equation}\label{b4}
\begin{aligned}
\| u(t) \|_{B^{1+\frac{d}{2}}_{2,1}} & \leq \| u_0 \|_{B^{1+\frac{d}{2}}_{2,1}} + C \int_{0}^{t}(|\alpha(r)| + |\beta(r)| + |\gamma(r)| + |\xi(r)|) \\
& \quad \times (1 + \| u \|^2_{\dot{B}^0_{\infty,1}} + \| \nabla u \|^2_{\dot{B}^0_{\infty,1}}) \| u \|_{B^{1+\frac{d}{2}}_{2,1}} dr,
\end{aligned}
\end{equation}
which together with the Gronwall inequality leads to
\begin{equation}\label{b5}
\begin{aligned}
\| u(t) \|_{B^{1+\frac{d}{2}}_{2,1}} \leq \| u_0 \|_{B^{1+\frac{d}{2}}_{2,1}} e^{\int_{0}^{t} (|\alpha(r)| + |\beta(r)| + |\gamma(r)| + |\xi(r)|) (1 +  \| u \|^2_{\dot{B}^0_{\infty,1}} + \| \nabla u \|^2_{\dot{B}^0_{\infty,1}}) dr}.
\end{aligned}
\end{equation}
If the solution $u$ of \eqref{SCNS-t1} blows up at finite time $T^{*} > 0$ but
\[
e^{\int_{0}^{T^{*}} (|\alpha(r)| + |\beta(r)| + |\gamma(r)| + |\xi(r)|) (\| u \|^2_{\dot{B}^0_{\infty,1}} + \| \nabla u \|^2_{\dot{B}^0_{\infty,1}}) dr} < \infty,
\]
then inequality \eqref{b5} implies that
\[
\sup_{t \in [0,T^{*}]} \| u(t) \|_{B^{1+\frac{d}{2}}_{2,1}} < \infty \quad \Rightarrow \quad \| u(T^{*}) \|_{B^{1+\frac{d}{2}}_{2,1}} < \infty.
\]
Thereby, by virtue of Theorem \ref{thm} with initia data $u_0 := u(T^{*})$, there must be a positive time $T^1 >0 $ such that the solution of \eqref{SCNS-t1} can be uniquely extended to the interval $[0,T^{*} + \frac{T^1}{2}]$, which is contradict to the fact that $T^{*}$ is maximum existence time.
Now let us derive the lower bound for the lifespan $T^{*}$. Using the embedding $B^{\frac{d}{2}}_{2,1} \hookrightarrow L^{\infty}$, it follows from \eqref{b3} that
\begin{equation*}
\begin{aligned}
\| u(t) \|_{B^{1+\frac{d}{2}}_{2,1}} & \leq \| u_0 \|_{B^{1+\frac{d}{2}}_{2,1}} + C \int_{0}^{t}(|\alpha(r)| + |\beta(r)| + |\gamma(r)| + |\xi(r)|) \left(1+\| u \|_{B^{1+\frac{d}{2}}_{2,1}}\right)^3 dr,
\end{aligned}
\end{equation*}
which is to say,
\begin{equation}\label{b6}
\begin{aligned}
1+\| u(t) \|_{B^{1+\frac{d}{2}}_{2,1}} & \leq 1+\| u_0 \|_{B^{1+\frac{d}{2}}_{2,1}} \\
& \quad + C \int_{0}^{t}(|\alpha(r)| + |\beta(r)| + |\gamma(r)| + |\xi(r)|) \left(1+\| u \|_{B^{1+\frac{d}{2}}_{2,1}}\right)^3 dr.
\end{aligned}
\end{equation}
Solving the inequality \eqref{b6} leads to
\begin{equation*}
\begin{aligned}
\| u(t) \|_{B^{1+\frac{d}{2}}_{2,1}} & \leq \frac{1+\| u_0 \|_{B^{1+\frac{d}{2}}_{2,1}}}{\left[ 1 - C \left(1+\| u_0 \|_{B^{1+\frac{d}{2}}_{2,1}}\right)^2 \int_{0}^{t}(|\alpha(r)| + |\beta(r)| + |\gamma(r)| + |\xi(r)|)dr \right]^{\frac{1}{2}}}-1.
\end{aligned}
\end{equation*}
By the assumption of $\alpha$, $\beta$, $\gamma$, $\xi \in L^1(0,\infty; \mathbb{R})$, the indefinite integral on any finite $[0,t]$ is absolutely continuous, hence one can find a $T(u_0)$ defined in Theorem \ref{blow-up_1} such that the blow-up time satisfies $T^{*} \geq T(u_0)$.

This finishes the proof of the Theorem \ref{blow-up_1}.
\end{proof}

\begin{proof}[\textbf{\emph{Proof of Theorem \ref{blow-up_2}.}}]
For any given initial data $u_0 \in B^s_{p,r}$ with $s > 1+ \frac{d}{p}$, the Theorem \ref{thm} in this work implies that the system \eqref{SCNS-t1} admits a unique solution $u$. According to the argument in the proof Theorem \ref{blow-up_1}, we have obtained the inequality
\begin{equation}\label{b7}
\begin{aligned}
\frac{d}{dt} \|\Delta_q u\|_{L^p} & \leq C \| [(\alpha(t) + \beta(t) u) \nabla S \cdot \nabla, \Delta_q]u \|_{L^p} + C \| \Delta_q [(\gamma(t)u + \xi(t) u^2) \Delta S] \|_{L^p} \\
& \quad + \| \div[(\alpha(t) + \beta(t) u) \nabla S] \|_{L^{\infty}} \| \Delta_q u \|_{L^p},
\end{aligned}
\end{equation}
which also holds for the case $p = \infty$ by taking the limit.

By multiplying both sides of \eqref{b7} by $2^{qs}$ and then taking the $l^r$ norm with respect to $q$, we get
\begin{equation}\label{b8}
\begin{aligned}
\frac{d}{dt} \|u(t)\|_{B^s_{p,r}} & \leq C \left(\sum_{q \geq -1} 2^{qrs} \| [(\alpha(t) + \beta(t) u) \nabla S \cdot \nabla, \Delta_q]u \|^r_{L^p}\right)^{\frac{1}{r}} \\
& \quad + C \|(\gamma(t)u + \xi(t) u^2) \Delta S \|_{B^s_{p,r}} \\
& \quad + \| \div[(\alpha(t) + \beta(t) u) \nabla S] \|_{L^{\infty}} \| u \|_{B^s_{p,r}}.
\end{aligned}
\end{equation}
Thanks to the commutator estimate (cf. Lemma \ref{ce}) for $s>1 + \frac{d}{p}>0$, we deduce that
\begin{equation*}
\begin{aligned}
& \left(\sum_{q \geq -1} 2^{qrs} \left\| \left[(\alpha(t) + \beta(t) u) \nabla S \cdot \nabla, \Delta_q\right] u \right\|^r_{L^p} \right)^{\frac{1}{r}}\\
%& \leq C(\| \nabla [(\alpha(t) + \beta(t) u) \nabla S] \|_{L^{\infty}} \|u\|_{B^s_{p,r}} + \| \nabla u \|_{L^{\infty}} \|\nabla [(\alpha(t) + \beta(t) u) S]\|_{B^{s-1}_{p,r}}) \\
& \leq C(|\alpha(t)| + |\beta(t)|) \left(\| \Delta S \|_{L^{\infty}} + \| \nabla u \|_{L^{\infty}} \| \nabla S \|_{L^{\infty}} + \| u \|_{L^{\infty}} \| \Delta S \|_{L^{\infty}} \right) \|u \|_{B^s_{p,r}} \\
& \quad + C(|\alpha(t)| + |\beta(t)|) \left(\| S \|_{B^s_{p,r}} + \| uS \|_{B^s_{p,r}} \right) \| \nabla u \|_{L^{\infty}} \\
& \leq C (|\alpha(t)| + |\beta(t)|) \left(1 + \| u \|^2_{L^{\infty}} + \| \nabla S \|^2_{L^{\infty}} + \| \nabla u \|^2_{L^{\infty}} + \| S \|^2_{L^{\infty}}\right) \|u \|_{B^s_{p,r}} \\
& \quad + C(|\alpha(t)| + |\beta(t)|) \left( (1 + \| \nabla u \|^2_{L^{\infty}} ) \| S \|_{B^s_{p,r}} + | \nabla u \|_{L^{\infty}} (\| u \|_{L^{\infty}} \| S \|_{B^s_{p,r}} + \| u \|_{B^s_{p,r}} \| S \|_{L^{\infty}}) \right) \\
& \leq C(|\alpha(t)| + |\beta(t)|) \left( 1 + \| u \|^2_{L^{\infty}} + \| \nabla S \|^2_{L^{\infty}} + \| \nabla u \|^2_{L^{\infty}} + \| S \|^2_{L^{\infty}} \right)(\| u \|_{B^s_{p,r}} + \| S \|_{B^s_{p,r}}).
\end{aligned}
\end{equation*}
Using the fact of $\| S \|_{B^s_{p,r}} \leq C\| u \|_{B^s_{p,r}}$, $\| S \|_{L^{\infty}} \leq C\| u \|_{L^{\infty}}$ and $\| \nabla S \|_{L^{\infty}} \leq C\| \nabla u \|_{L^{\infty}}$, we get from last inequality that
\[
\left(\sum_{q \geq -1} 2^{qrs} \left\| \left[(\alpha(t) + \beta(t) u) \nabla S \cdot \nabla, \Delta_q\right] u \right\|^r_{L^p} \right)^{\frac{1}{r}} \leq C(|\alpha(t)| + |\beta(t)|) \left( 1 + \| u \|^2_{L^{\infty}} + \| \nabla u \|^2_{L^{\infty}} \right)\| u \|_{B^s_{p,r}}.
\]
The other two terms on the right-hand side of \eqref{b8} can be estimated by Moser estimates (cf. Lemma \ref{me}) that
\begin{equation*}
\begin{aligned}
\|(\gamma(t)u + \xi(t) u^2) \Delta S \|_{B^s_{p,r}} & \leq C\left(\|\gamma(t)u + \xi(t) u^2\|_{L^{\infty}} \| \Delta S \|_{B^s_{p,r}} + \|\gamma(t)u + \xi(t) u^2\|_{B^s_{p,r}} \| \Delta S \|_{L^{\infty}} \right) \\
& \leq C (|\gamma(t)| + |\xi(t)|) \left(\|u + u^2\|_{L^{\infty}} \| \Delta S \|_{B^s_{p,r}} + \|u + u^2\|_{B^s_{p,r}} \| \Delta S \|_{L^{\infty}} \right) \\
& \leq C (|\gamma(t)| + |\xi(t)|) \left( (1 + \| u \|^2_{L^{\infty}})\| u \|_{B^s_{p,r}} + \| u \|_{B^s_{p,r}} \| u \|_{L^{\infty}} + \| u \|_{B^s_{p,r}} \| u \|^2_{L^{\infty}} \right) \\
& \leq C (|\gamma(t)| + |\xi(t)|)(1 + \| u \|^2_{L^{\infty}})\| u \|_{B^s_{p,r}},
\end{aligned}
\end{equation*}
and \begin{equation*}
\begin{aligned}
\| \div[(\alpha(t) + \beta(t) u) \nabla S] \|_{L^{\infty}} \| u \|_{B^s_{p,r}} & \leq C (|\alpha(t)| + |\beta(t)|) \| \div[(1 + u) \nabla S] \|_{L^{\infty}} \| u \|_{B^s_{p,r}} \\
& \leq C (|\alpha(t)| + |\beta(t)|)(1 + \| u \|^2_{L^{\infty}})\| u \|_{B^s_{p,r}}.
\end{aligned}
\end{equation*}
Plugging the last three estimates into \eqref{b8} yields that
\begin{equation}\label{b9}
\begin{aligned}
\frac{d}{dt} \|u(t)\|_{B^s_{p,r}} & \leq C (|\alpha(t)| + |\beta(t)| + |\gamma(t)| + |\xi(t)|) \left( 1 + \| u \|^2_{L^{\infty}} + \| \nabla u \|^2_{L^{\infty}} \right)\| u \|_{B^s_{p,r}}.
\end{aligned}
\end{equation}
Recall the following Brezis-Gallouet-Wainger type estimate(cf. \cite{Kozono2002}):
\[
\|f\|_{L^{\infty}} \leq C \left(1 + \| u(t) \|_{\dot{B}^{\frac{d}{p}}_{p,\rho}} \log^{1 - \frac{1}{\rho}}(e + \|u(t) \|_{B^s_{q,\sigma}}))\right), \quad \forall f \in \dot{B}^{\frac{d}{p}}_{p,\rho} (\mathbb{R}^d) \cap B^s_{q,\sigma}(\mathbb{R}^d),
\]
where $q \in [1,\infty)$, $p$, $\rho$, $\sigma \in [1,\infty]$ and $s > 1+\frac{d}{p}$.

Applying the last inequality for $s>1+\frac{d}{p}$ with $p \in [1,\infty)$, $\sigma \in [1,\infty]$, we can estimate the $L^{\infty}$-norm on the right-hand side of \eqref{b9} that
\begin{equation*}
\begin{aligned}
\|u\|_{L^{\infty}} \leq C \left( 1 + \|u\|_{\dot{B}^0_{\infty,2}} \log^{\frac{1}{2}}(e + \|u\|_{B^s_{p,r}}) \right),
\end{aligned}
\end{equation*}
and
\begin{equation*}
\begin{aligned}
\| \nabla u\|_{L^{\infty}} \leq C \left( 1 + \|\nabla u(t)\|_{\dot{B}^0_{\infty,2}} \log^{\frac{1}{2}}(e + \|\nabla u\|_{B^{s-1}_{p,r}}) \right).
\end{aligned}
\end{equation*}
After simple calculation, we derive from last two estimates that
\begin{equation}\label{b10}
\begin{aligned}
\|u\|^2_{L^{\infty}} + \| \nabla u\|^2_{L^{\infty}} & \leq C \left(1 + (\|u\|^2_{\dot{B}^0_{\infty,2}} + \|\nabla u\|^2_{\dot{B}^0_{\infty,2}}) \log(e + \|u\|_{B^s_{p,r}})\right) \\
& \leq C \left(1 + \|u\|^2_{\dot{B}^0_{\infty,2}} + \|\nabla u\|^2_{\dot{B}^0_{\infty,2}}\right) \log(e + \|u\|_{B^s_{p,r}}).
\end{aligned}
\end{equation}
By \eqref{b9} and \eqref{b10} and the property of $\log(e + \|u\|_{B^s_{p,r}}) \leq \log(e^2 + \|u\|^2_{B^s_{p,r}})$, we see that
\begin{equation*}
\begin{aligned}
\frac{d}{dt} \|u\|_{B^s_{p,r}} & \leq C (|\alpha(t)| + |\beta(t)| + |\gamma(t)| + |\xi(t)|) \\
& \quad \times \left(1 + \|u\|^2_{\dot{B}^0_{\infty,2}} + \|\nabla u\|^2_{\dot{B}^0_{\infty,2}}\right) \log(e^2 + \|u\|^2_{B^s_{p,r}})\|u\|_{B^s_{p,r}}.
\end{aligned}
\end{equation*}
Multiplying both sides of last inequality by $\|u\|_{B^s_{p,r}}$ and using the fact of $(a+b)^2 \sim a^2+b^2$, the last inequality can be transformed into
\begin{equation*}
\begin{aligned}
\frac{d}{dt} (e^2 + \|u\|^2_{B^s_{p,r}}) & \leq C (|\alpha(t)| + |\beta(t)| + |\gamma(t)| + |\xi(t)|) \\
& \quad \times \left(1 + \|u\|^2_{\dot{B}^0_{\infty,2}} + \|\nabla u\|^2_{\dot{B}^0_{\infty,2}}\right) (e^2 + \|u\|^2_{B^s_{p,r}}) \log(e^2 + \|u\|^2_{B^s_{p,r}}).
\end{aligned}
\end{equation*}
It then follows that
\begin{equation*}
\begin{aligned}
\frac{d}{dt} \log \log (e^2 + \|u\|^2_{B^s_{p,r}}) \leq C (|\alpha(t)| + |\beta(t)| + |\gamma(t)| + |\xi(t)|) \left(1 + \|u\|^2_{\dot{B}^0_{\infty,2}} + \|\nabla u\|^2_{\dot{B}^0_{\infty,2}}\right),
\end{aligned}
\end{equation*}
which implies that
\begin{equation}\label{b11}
\begin{aligned}
\log (e^2 & + \|u\|^2_{B^s_{p,r}}) \leq \log (e^2 + \|u\|^2_{B^s_{p,r}}) \\
& \times e^{C \int_{0}^{t} (|\alpha(\tau)| + |\beta(\tau)| + |\gamma(\tau)| + |\xi(\tau)|) (1 + \|u(\tau)\|^2_{\dot{B}^0_{\infty,2}} + \|\nabla u(\tau)\|^2_{\dot{B}^0_{\infty,2}}) d\tau}.
\end{aligned}
\end{equation}
Due to \eqref{b11} and using the similar argument as that in the proof of Theorem \ref{blow-up_1}, one can also prove that the solution $u$ blows up in finite time. We shall omit the details here.

To derive the lower bound for the blow-up time $T^{*}$, we shall use a method which is different from \eqref{b6}. More specifically, by applying the logarithmic Sobolev inequality \eqref{b10}, we have
\begin{equation}\label{b12}
\begin{aligned}
\|u\|^2_{L^{\infty}} + \| \nabla u\|^2_{L^{\infty}} & \leq C \left[\left(1 + \|u\|_{\dot{B}^0_{\infty,\infty}} \log(e + \|u\|_{B^s_{p,r}}) \right)^2 \right. \\
& \quad + \left.\left(1 + \|\nabla u\|_{\dot{B}^0_{\infty,\infty}} \log(e + \|\nabla u\|_{B^{s-1}_{p,r}}) \right)^2 \right] \\
& \leq C (1 + \|u\|^2_{\dot{B}^0_{\infty,\infty}} + \|\nabla u\|^2_{\dot{B}^0_{\infty,\infty}}) \left(\log (e^2 + \|u\|^2_{B^s_{p,r}})\right)^2.
\end{aligned}
\end{equation}
Note that the existence of the quadratic function $\log^2 (e^2 + \|u\|^2_{B^s_{p,r}})$ in \eqref{b12} makes the combination of estimates \eqref{b9} and \eqref{b12} insufficient for the establishment of the blow-up criteria in space $\dot{B}^0_{\infty,\infty}$. Furthermore, in terms of the Sobolev embedding $B^s_{p,r} \hookrightarrow L^{\infty} (s >1+\frac{d}{p})$ and $L^{\infty} \hookrightarrow B^0_{\infty,\infty} \hookrightarrow \dot{B}^0_{\infty,\infty}$, we get from \eqref{b12} that
\begin{equation}\label{b13}
\begin{aligned}
\|u\|^2_{L^{\infty}} + \| \nabla u\|^2_{L^{\infty}} \leq C (1 + \|u\|^2_{B^s_{p,r}} + \|\nabla u\|^2_{B^{s-1}_{p,r}}) \left(\log (e^2 + \|u\|^2_{B^s_{p,r}})\right)^2.
\end{aligned}
\end{equation}
Combining the estimates \eqref{b9} and \eqref{b13} and using the inequality $\log (e^2 + x) \leq e^2 + x$ for all $x \geq 0$, we have
\begin{equation}\label{b14}
\begin{aligned}
\frac{d}{dt} \|u(t)\|_{B^s_{p,r}} & \leq C (|\alpha(t)| + |\beta(t)| + |\gamma(t)| + |\xi(t)|)  \\
& \quad \times (1 + \|u\|^2_{B^s_{p,r}} + \|\nabla u\|^2_{B^{s-1}_{p,r}}) \left(\log (e^2 + \|u\|^2_{B^s_{p,r}})\right)^2 \| u \|_{B^s_{p,r}} \\
& \leq C (|\alpha(t)| + |\beta(t)| + |\gamma(t)| + |\xi(t)|) \left(e^2 + \|u\|^2_{B^s_{p,r}}\right)^3 \| u \|_{B^s_{p,r}} \\
& \leq C (|\alpha(t)| + |\beta(t)| + |\gamma(t)| + |\xi(t)|) \left(e + \|u\|_{B^s_{p,r}}\right)^7.
\end{aligned}
\end{equation}
Solving the inequality \eqref{b14} leads to
\begin{equation*}
\begin{aligned}
e + \|u\|_{B^s_{p,r}} \leq \frac{e + \|u_0\|_{B^s_{p,r}}}{\left[1- C \int_{0}^{t}(|\alpha(t)| + |\beta(t)| + |\gamma(t)| + |\xi(t)|)dr \left(e + \|u_0\|_{B^s_{p,r}}\right)^6 \right]^{\frac{1}{6}}}.
\end{aligned}
\end{equation*}
Due to the condition of $\alpha$, $\beta$, $\gamma$, $\xi \in L^1(0,\infty; \mathbb{R})$, the indefinite integral on any finite $[0,t]$ is absolutely continuous, so there is a constant $T'(u_0) > 0$ defined in Theorem \ref{blow-up_2} such that the lifespan $T^{*}$ of the solution $u$ satisfies $T^{*} \geq T'(u_0)$.

This completes the proof of Theorem \ref{blow-up_2}.
\end{proof}

\section{appendix}
In this section, we recall some well-known facts of the Littlewood-Paley decomposition theory and the linear transport theory in Besov spaces.
\begin{lemma}[\cite{Bahouri2011}]\label{l1}
Denote by $\mathcal{C}$ the annulus of centre $0$, short radius $3/4$ and long radius $8/3$. Then there exists two positive radial functions $\chi$ and $\varphi$ belonging respectively to $C^{\infty}_c(B(0,4/3))$ and $C^{\infty}_c(\mathcal{C})$ such that
$$
\chi(\xi) + \sum_{q \geq 0} \varphi(2^{-q} \xi) = 1, \quad \forall \xi \in \mathbb{R}^d,
$$

$$
|p-q| \geq 2 \Rightarrow \textrm{supp} \varphi (2^{-q} \cdot) \cap \textrm{supp} \varphi (2^{-p} \cdot) = \emptyset,
$$
$$
q \geq 1 \Rightarrow \textrm{supp} \chi (\cdot) \cap \textrm{supp} \varphi (2^{-q} \cdot) = \emptyset,
$$
and
$$
\frac{1}{3} \leq {\chi}^2(\xi) + \sum_{q \geq 0} {\varphi}^2(2^{-q} \xi) \leq 1,  \quad \forall \xi \in \mathbb{R}^d.
$$
\end{lemma}
For the function $\chi$ and $\varphi$ satisfying Lemma \ref{l1}, we write $h = \mathcal{F}^{-1} \varphi$ and $\tilde{h} = \mathcal{F}^{-1} \chi$, where $\mathcal{F}^{-1} u$ denotes the inverse Fourier transformation of $u$. The nonhomogeneous  dyadic blocks $\Delta_j$ are defined by
$$
\Delta_q u \stackrel{\text{def}}{=} 0,\quad q \leq -2; \quad \Delta_{-1} u \stackrel{\text{def}}{=} \chi(D)u = \int_{\mathbb{R}} u(x-y) \tilde{h} (y)dy, \quad q=-1;
$$
$$
\Delta_q u \stackrel{\text{def}}{=} \varphi (2^{-j}D)u = \int_{\mathbb{R}^d} u(x-y) h (2^{-j}y)dy, \quad q \geq 0.
$$
The nonhomogeneous Littlewood-Paley decomposition of $ u \in \mathscr{S}'$ is given by
$$
u = \sum_{q \geq -1} \Delta_q u.
$$
We also define the high-frequency cut-off operator as
$$
S_q u \stackrel{\text{def}}{=} \sum_{p \leq q-1} \Delta_p u, \quad \forall q \in \mathbb{N}.
$$
\begin{definition}[\cite{Bahouri2011}]
For any $s \in \mathbb{R}$ and $p,r \in [1,\infty]$, the $d-$dimension nonhomogeneous Besov space $B^s_{p,r}$ is defined by
$$
B^s_{p,r} \stackrel{\text{def}}{=} \{ u\in \mathscr{S}';  \| u \|_{B^s_{p,r}} = \| (2^{qs} \| \Delta_q u \|_{L^p})_{l \geq {-1}} \|_{l^r} < \infty\}.
$$
If $s = \infty$, $B^{\infty}_{p,r} \stackrel{\text{def}}{=} \cap_{s \in \mathbb{R}} B^s_{p,r}$.
\end{definition}

\begin{lemma}[\cite{Bahouri2011}]
A smooth function $f: \mathbb{R}^d \to \mathbb{R}$ is said to be an $S^m$-multiplier: if $\forall \alpha \in \mathbb{N}^n$, there exists a constant $C_{\alpha} > 0$ such that
\[
|\partial^{\alpha} f(\xi)| \leq C_{\alpha} (1+ |\xi|)^{m - |\alpha|}, \quad \xi \in \mathbb{R}^d.
\]
If $f$ is a $S^m$-multiplier, then the operator $f(D)$ is continuous from $B^s_{p,r}(\mathbb{R}^d)$ to $B^{s-m}_{p,r}(\mathbb{R}^d)$, for all $s \in \mathbb{R}$ and $1 \leq p,r \leq \infty$.
\end{lemma}

\begin{lemma}[Interpolation inequalities \cite{Danchin2005}]\label{interpolation}
\begin{enumerate}
\item If $ u \in B^{s_1}_{p,r}(\mathbb{R}^d) \cap B^{s_2}_{p,r}(\mathbb{R}^d)$, then
\[
\| u \|_{B^{\theta s_1 + (1-\theta) s_2}_{p,r}} \leq C \| u \|^{\theta}_{B^{s_1}_{p,r}} \| u \|^{1 - \theta}_{B^{s_2}_{p,r}},\quad \forall \theta \in [0,1].
\]
\item For any $s \in \mathbb{R}$, $\epsilon >0$ and $1 \leq p \leq \infty$, there exists a constant $C > 0$ such that
\[
\| u \|_{B^s_{p,1}} \leq C \frac{\epsilon +1}{\epsilon} \| u \|_{B^s_{p,\infty}} \left( 1 + \log \frac{\| u \|_{B^{s+\epsilon}_{p,\infty}}}{\| u \|_{B^s_{p,\infty}}}\right).
\]
\end{enumerate}
\end{lemma}

\begin{lemma}[Commutator estimates \cite{Bahouri2011}]\label{ce}
Let $\sigma > 0$, $1 \leq r \leq \infty$ and $1 \leq p \leq p_1 \leq \infty$. Let $v$ be a vector field over $\mathbb{R}^d$. Then the following estimates hold,
\[
\| (2^{j\sigma} \|[v \cdot \nabla,\Delta_j]f \|_{L^p})_{j \in \mathbb{N}} \|_{l^r} \leq C(\| \nabla v \|_{L^{\infty}} \|f \|_{B^{\sigma}_{p,r}} + \| \nabla f \|_{L^{p_2}} \| v \|_{B^{\sigma - 1}_{p_1,r}}),
\]
where $\frac{1}{p} = \frac{1}{p_1} + \frac{1}{p_2}$. In addition, if $\sigma < 1+\frac{d}{p_1}$, we have
\[
\| (2^{j\sigma} \|[v \cdot \nabla,\Delta_j]f \|_{L^p})_{j \in \mathbb{N}} \|_{l^r} \leq C \| \nabla v \|_{B^{\frac{d}{p}}_{p,\infty} \cap L^{\infty}} \|f \|_{B^{\sigma}_{p,r}}.
\]
\end{lemma}

\begin{lemma}[Moser estimate \cite{Danchin2001}]\label{me}
Let $1 \leq p,r \leq \infty$.
\begin{enumerate}
\item For $s>0$,
\[
\| fg \|_{B^s_{p,r}} \leq C(\|f\|_{B^s_{p,r}} \|g\|_{L^{\infty}} + \|f\|_{L^{\infty}} \|g\|_{B^s_{p,r}}).
\]
\item $\forall s_1 \leq \frac{1}{p} < s_2$ $(s_2 \geq \frac{1}{p}$ if $r = 1)$  and $s_1 + s_2 > 0$, we have
\[
\| fg \|_{B^{s_1}_{p,r}} \leq C \| f \|_{B^{s_1}_{p,r}} \| g \|_{B^{s_2}_{p,r}}.
\]
\end{enumerate}
\end{lemma}

\begin{lemma}[Osgood Lemma \cite{Chemin1998}]\label{Osgood}
Let $\rho \geq 0$ be a measurable function, $\gamma > 0$ be a locally integrable function and $\mu$ be a increasing continuous function. Suppose that
\[
\rho(t) \leq a + \int_{t_0}^{t} \gamma(s) \mu( \rho(s)) ds, \quad \text{for some } a \geq 0.
\]
\begin{enumerate}
\item If $a>0$, then we have
\[
-\mathcal{M}(\rho(t)) + \mathcal{M}(a) \leq \int_{t_0}^{t} \gamma(s) \,  {d}s, \quad \mathcal{M} (x) \stackrel{ {def}}{=} \int_{x}^{1} \frac{1}{\mu(r)} \,  {d}r.
\]
\item If $a=0$ and $\mu$ satisfies the condition $\int_{0}^{1} \frac{dr}{\mu(r)}dr = +\infty$, then $\rho \equiv 0$.
\end{enumerate}
\end{lemma}

\begin{lemma}[Fatou-type lemma \cite{Bahouri2011}]\label{lem:FT}
Let $s \in \mathbb{R}$, $p,r \in [1,\infty]^2$. If $(f_k)_{k \geq 1}$ is a bounded sequence in $B^{s}_{p,r}$, and $f_k \overset{\mathscr{S}'}{\to} f$ as $k \to \infty$, where $\mathscr{S}'$ is the tempered distribution space.Then $f \in B^{s}_{p,r} $ and
$$
\| f \|_{B^{s}_{p,r}} \leq C \mathop{\lim\inf}\limits_{k \to \infty} \| f_k \|_{B^{s}_{p,r}}.
$$
\end{lemma}

Now let us state some useful results in the transport equation theory in Besov spaces, which are crucial to the proofs of our main theorems.
\begin{lemma}[\cite{Bahouri2011}] \label{prior estimates}
	Assume that $p,r \in [1, \infty]^2$ and $s > -\frac{d}{p}$. Let $v$ be a vector field such that $\nabla v$ belongs to $L^1([0,T];B^{s-1}_{p,r})$ if $s > 1+\frac{d}{p}$ or to $L^1([0,T];B^{\frac{d}{p}}_{p,r} \cap L^{\infty})$ otherwise. Suppose also that $f_0 \in B^s_{p,r}$, $F \in L^1([0,T];B^s_{p,r})$ and that $f \in L^{\infty} ([0,T];B^s_{p,r}) \cap C([0,T];\mathscr{S}')$ solves the $d$-dimensional linear transport equations
	\begin{equation}\label{l2}
		\left\{
		\begin{aligned}
			&\partial_t f + v \cdot \nabla f = F, \\
			& f|_{t=0} =f_0.
		\end{aligned}
		\right.
	\end{equation}
	Then there exists a constant C depending only on $s$, $p$ and $d$ such that the following statements hold:
	\begin{enumerate}
		\item If $r =1$ or $s \ne 1+ \frac{d}{p}$, then
		\[
		\| f(t) \|_{B^s_{p,r}} \leq \| f_0 \|_{B^s_{p,r}} + \int_{0}^{t} \| F(\tau) \|_{B^s_{p,r}} d\tau + C \int_{0}^{t} V'(\tau) \|f(\tau) \|_{B^s_{p,r}} d\tau,
		\]
		or
		\begin{equation}\label{l3}
			\| f(t) \|_{B^s_{p,r}} \leq e^{CV(t)} \left(\| f_0 \|_{B^s_{p,r}} +  \int_{0}^{t} e^{-CV(\tau)} \| F(\tau) \|_{B^s_{p,r}} d\tau \right),
		\end{equation}
		holds, where $V(t) = \int_{0}^{t} \| \nabla v(\tau) \|_{B^{\frac{d}{p}}_{p,r} \cap L^{\infty}} d\tau$ if $s < 1+\frac{d}{p}$ and $V(t) = \int_{0}^{t} \| \nabla v(\tau) \|_{B^{s-1}_{p,r}} d\tau$ else.
		
		\item If $s \leq 1+\frac{d}{p}$ and, in addition, $\nabla f_0 \in L^{\infty}$, $\nabla f \in L^{\infty}([0,T] \times \mathbb{R}^d)$ and $\nabla F \in L^1([0,T];L^{\infty})$, then
		\begin{equation}
			\begin{aligned}
				\| & f(t) \|_{B^s_{p,r}} + \| \nabla f(t) \|_{L^{\infty}} \\
				& \leq e^{CV(t)} \left(\| f_0 \|_{B^s_{p,r}} + \| \nabla f_0 \|_{L^{\infty}} + \int_{0}^{t} e^{-CV(\tau) } (\| F(\tau) \|_{B^s_{p,r}} + \| \nabla F(\tau) \|_{L^{\infty}})d\tau\right), \notag
			\end{aligned}
		\end{equation}
		with $V(t) = \int_{0}^{t} \| \nabla v(\tau) \|_{B^{\frac{d}{p}}_{p,r} \cap L^{\infty}} d\tau$.
		\item If $f = v$, then for all $s>0$, the estimate \eqref{l3} holds with $V(t) = \int_{0}^{t} \| \nabla v(\tau) \|_{L^{\infty}} d\tau$.
		
		\item If $r < + \infty$, then $f \in C([0,T];B^s_{p,r})$. If $r = +\infty $, then $f \in C([0,T];B^{s'}_{p,1})$ for all $s' < s$.
	\end{enumerate}
\end{lemma}

\begin{lemma}[\cite{Bahouri2011}]\label{euth}
Let $(p,p_1,r) \in [1,\infty]^3$. Assume that $s > -d \min \{\frac{1}{p_1}, \frac{1}{p'}\}$ with $p' = (1 - \frac{1}{p})^{-1}$. Let $f_0 \in B^s_{p,r}$ and $F \in L^1 ([0,T];B^s_{p,r})$. Let $v \in L^{\rho} ([0,T]; B^{-M}_{\infty,\infty})$ for some $\rho >1$, $M>0$ and $\nabla v \in L^1([0,T];B^{\frac{d}{p_1}}_{p_1,\infty} \cap L^{\infty})$ if $s < 1+\frac{d}{p_1}$, and $\nabla v \in L^1([0,T];B^{s-1}_{p_1,r})$ if $s > 1+\frac{d}{p_1}$ or $s = 1+\frac{d}{p_1}$ and $r=1$. Then \eqref{l2} has a unique solution $f \in L^{\infty} ([0,T];B^s_{p,r}) \cap (\cap_{s' < s} C([0,T];B^{s'}_{p,1}))$ and the inequalities in Lemma \ref{prior estimates} hold true. If, moreover, $r < \infty$, then we have $f \in C([0,T];B^s_{p,r})$.
\end{lemma}

\section*{Acknowledgements}

This work was partially supported by the National Key Research and Development Program of China (Grant No. 2023YFC2206100).

%\begin{thebibliography}{11}
\bibliographystyle{plain}%
\bibliography{myref}

\end{document}